\documentclass{amsart}
\usepackage{graphicx}

\usepackage[english]{babel}

\usepackage{multicol}
\usepackage{subcaption}

\usepackage{amsmath, amsfonts, amssymb, graphicx, xcolor, enumitem, url, hyperref}
\usepackage{amsthm}
\usepackage[justification=centering, font={sl,small}, labelfont={bf}, margin=1cm]{caption}
\usepackage[font={sl,footnotesize}]{subcaption}
\usepackage{blkarray}
\usepackage{overpic} 

\hypersetup{colorlinks=true, citecolor=darkblue, linkcolor=darkblue}

\makeatletter
\def\namedlabel#1#2{\begingroup
   \def\@currentlabel{#2}%
   \label{#1}\endgroup
}
\makeatother

\makeatletter
\renewcommand{\pod}[1]{\mathchoice
  {\allowbreak \if@display \mkern 6mu\else \mkern 6mu\fi (#1)}
  {\allowbreak \if@display \mkern 6mu\else \mkern 6mu\fi (#1)}
  {\mkern4mu(#1)}
  {\mkern4mu(#1)}
}

\makeatletter
\g@addto@macro\bfseries{\boldmath}
\makeatother

\usepackage{hyperref}

\makeatletter
\renewcommand\@makefntext[1]{%
  \noindent\makebox[0em][r]{\@makefnmark}#1}
\makeatother

\newtheorem{theorem}{Theorem}[section]
\newtheorem{proposition}[theorem]{Proposition}
\newtheorem{corollary}[theorem]{Corollary}
\newtheorem{lemma}[theorem]{Lemma}

\theoremstyle{definition}
\newtheorem{definition}[theorem]{Definition}
\newtheorem{example}[theorem]{Example}

\newtheorem{remark}[theorem]{Remark}
\newtheorem{question}[theorem]{Question}

\newcommand{\NN}{\mathbb{N}}

\newcommand{\RR}{\mathbb{R}}

\newcommand{\TP}{\mathbb{TP}}

\newcommand{\cK}{\mathcal{K}}

\newcommand{\cO}{\mathcal{O}}
\newcommand{\cP}{\mathcal{P}}

\newcommand{\cT}{\mathcal{T}}

\newcommand{\bfe}{\mathbf{e}}

\newcommand{\bfs}{\mathbf{s}}

\newcommand{\sfg}{\mathsf{g}}

\newcommand{\sfN}{\mathsf{N}}
\newcommand{\sfE}{\mathsf{E}}

\DeclareMathOperator{\conv}{conv} 
\DeclareMathOperator{\vol}{vol} 
\DeclareMathOperator{\supp}{supp} 
\DeclareMathOperator{\link}{link} 

\newcommand{\bigast}{%
  \mathop{ 
    \mathchoice{\dobigast\huge}
               {\dobigast\large}
               {\dobigast\normalsize}
               {\dobigast\small}
    }\displaylimits 
}

\newcommand{\dobigast}[1]{%
  \vcenter{#1\kern.2ex\hbox{$\ast$}\kern.2ex}}

\newcommand{\ngonp}[1]{P_{#1}}
\newcommand{\ngon}{\ngonp{n+1}}

\newcommand{\cyctri}[1]{\mathfrak{C}_{#1}}

\newcommand{\catblock}[1]{\mathcal{U}_{#1}}

\newcommand{\productp}[2]{\Delta_{{#1}}\times\Delta_{{#2}}} 
\newcommand{\product}{\productp{n}{\ol m}} 
\newcommand{\productn}{\productp{n}{\ol n}} 

\newcommand{\bipartitep}[2]{K_{{#1},{\overline #2}}} 
\newcommand{\bipartiteb}{\bipartitebp{n}{m}} 
\newcommand{\bipartitebn}{\bipartitebp{n}{n}} 
\newcommand{\bipartitebp}[2]{K_{{[#1]},{[\overline #2]}}} 
\newcommand{\bipartiteij}{\bipartitep{I}{J}} 

\newcommand{\asstri}[1]{{\mathfrak{A}}_{#1}} 

\newcommand{\asscomp}[1]{{\mathcal{A}}_{#1}}

\newcommand{\cyccomp}[1]{{\mathcal{C}}_{#1}}

\newcommand{\Tam}[1]{\operatorname{Tam}_{#1}}
\newcommand{\Asso}{\operatorname{Asso}}
\newcommand{\AssoIJ}{\Asso_{I,\ol J}}
\newcommand{\pa}{\rho}
\newcommand{\Cyclo}{\operatorname{Cyclo}}
\newcommand{\CycloIJ}{\Cyclo_{I,\ol J}}

\def\horiz{\operatorname{horiz}}

\newcommand{\can}{c}

\newcommand{\cost}[1]{\kappa({#1})}

\renewcommand{\ij}{(i,\ol j)}
\renewcommand{\IJ}{(I,\ol {J})}
\newcommand{\nn}{([n],[\ol n])}

\newcommand{\reverse}[1]{\overleftarrow{#1}}

\newcommand{\ha}{\mathsf{h}}
\newcommand{\hb}{\mathsf{h}}
\newcommand{\arrgtA}{\mathcal{H}} 
\newcommand{\arrgtB}{\mathcal{H}} 

\newcommand{\lth}{\ell} 

\newcommand{\ol}[1]{\overline{#1}}

\newcommand{\ra}{\rightarrow}

\definecolor{darkblue}{rgb}{0,0,0.7} 
\newcommand{\darkblue}{\color{darkblue}} 
\newcommand{\defn}[1]{\emph{\darkblue #1}} 

\usepackage{todonotes}

\subjclass[2010]{05E45, 05E10, 52B22}
\keywords{Triangulations of products of simplices, Tamari lattices, Associahedra, Cyclohedra, Tropical hyperplanes}

\begin{document}

\title[Geometry of $\nu$-Tamari lattices in types $A$ and $B$]{Geometry of $\nu$-Tamari lattices in types $A$ and $B$}

\author{Cesar Ceballos}
\address{Faculty of Mathematics, University of Vienna, Vienna, Austria}

\author{Arnau Padrol}
\address{ 
Sorbonne Universit\'es, Universit\'e Pierre et Marie Curie (Paris 6), Institut de Math\'ematiques de Jussieu - Paris Rive Gauche (UMR 7586), Paris, France} 

\author{Camilo Sarmiento}
\address{Max Planck Institute for Mathematics in the Sciences, Leipzig, Germany}

\thanks{The research of C.C. was supported by the Austrian Science Foundation FWF, grant F 5008-N15, in the framework of the Special Research Program ``Algorithmic and Enumerative Combinatorics''; A.P. was supported by the program PEPS Jeunes Chercheur-e-s 2016 from the INSMI; and C.S. was partially supported by CDS Magdeburg.}

\begin{abstract}
In this paper, we exploit the combinatorics and geometry of triangulations of products of simplices to derive new results in the context of Catalan combinatorics of $\nu$-Tamari lattices.
In our framework, the main role of ``Catalan objects'' is played by $(I,\overline{J})$-trees: bipartite trees associated to a pair $(I,\overline{J})$ of finite index sets that stand in simple bijection with lattice paths weakly above a lattice path $\nu=\nu(I,\overline{J})$. Such trees label the maximal simplices of a triangulation whose dual polyhedral complex 
gives a geometric realization of the $\nu$-Tamari lattice introduced by Pr\'evile-Ratelle and Viennot. 
In particular, we obtain geometric realizations of $m$-Tamari lattices as polyhedral subdivisions of associahedra induced by an arrangement of tropical hyperplanes, giving a positive answer to an open question of F.~Bergeron. 

The simplicial complex underlying our triangulation endows 
the $\nu$-Tamari lattice with a full simplicial complex structure. 
It is a natural generalization of the classical simplicial associahedron, alternative to the rational associahedron of Armstrong, Rhoades and Williams, whose $h$-vector entries are given by a suitable generalization of the Narayana numbers.

Our methods are amenable to cyclic symmetry, which we use to present type $B$ analogues of our constructions. Notably, we define a partial order that generalizes the type $B$ Tamari lattice, introduced independently by Thomas and Reading, along with corresponding geometric realizations.

\end{abstract}

\maketitle

\tableofcontents

\section{Introduction}

The \defn{Tamari lattice} is a partial order on Catalan objects that has been widely studied since it was first introduced by Tamari in his doctoral thesis in 1951~\cite{TamariPhD}. Its covering relation can be described in terms of flips in polygon triangulations, rotations on binary trees and certain elementary transformation on Dyck paths. Tamari lattices appear naturally in many areas of mathematics: algebra, combinatorics, geometry, topology... We refer to the monograph~\cite{TamariFestschrift} (and its references) for a presentation of many of its remarkable properties, connections, applications, and generalizations.

In this paper we are especially interested in two extensions of the Tamari lattice.
The \defn{$m$-Tamari lattice} is a lattice structure on the set of Fuss-Catalan Dyck paths introduced by F.~Bergeron and Pr\'eville-Ratelle in their combinatorial study of higher diagonal coinvariant spaces~\cite{BergeronPrevilleRatelle2012}.
It recovers the classical Tamari lattice for $m=1$, and has attracted considerable attention in other areas such as representation theory and Hopf algebras~\cite{Bergeron2012,BCP2013, NovelliThibon2014, Novelli2014}. 
The enumerative properties of their intervals have an interesting story going back to Chapoton~\cite{chapoton_sur_2005} in the classical case, and followed by F.~Bergeron, Bousquet-M\'elou, Chapuy, Fusy and Pr\'eville-Ratelle~\cite{BergeronPrevilleRatelle2012,BCP2013,BFP2011} in the Fuss-Catalan case. 
The motivation for the study of these intervals comes from algebra, where their enumeration is conjecturally interpreted as the dimension of the alternating component of the trivariate Garsia--Haiman space of the same order; a labeled version of these intervals corresponds to the dimension of the entire trivariate Garsia--Haiman space; we refer to~\cite{BFP2011} for an overview and origin of these conjectures. These remarkable algebraic connections also motivated the introduction of the
\defn{$\nu$-Tamari lattice $\Tam{\nu}$} by Pr\'eville-Ratelle and Viennot~\cite{PrevilleRatelleViennot2017}, a partial order on the set of lattice paths lying weakly above any given lattice path~$\nu$. These broader lattices do not only generalize the classical Tamari lattice, but  bring to light a finer structure by decomposing the latter into smaller pieces~\cite{PrevilleRatelleViennot2017}. The enumeration of intervals purely contained in these pieces coincides with the enumeration of non-separable planar maps~\cite{fang_enumeration_2017}, a formula discovered by Tutte way back in the sixties~\cite{tutte_census_1963}.

This paper presents a more geometric approach to $m$-Tamari lattices and $\nu$-Tamari lattices, as well as partial extensions to other Coxeter groups. We expect that the geometric and combinatorial techniques developed here may be useful for the connections mentioned above. Our results are subdivided in three different directions: We construct geometric realizations, provide a full simplicial complex structure, and consider its analogues in type~$B$.

\subsection{Geometric realizations of the \texorpdfstring{$\nu$}{v}-Tamari lattice}

One of the striking characteristics of the Tamari lattice is that its Hasse diagram can be realized as the edge graph of a polytope, the associahedron. The realization problem of associahedra as polytopes was explicitly posed by Stasheff in 1963~\cite{St63}, who constructed it as a cellular ball. The first constructions as a polytope are due to Haiman (1984)~\cite{Ha84} and Lee (1989)~\cite{Lee2}. Subsequently, many systematic construction methods emerged: as a secondary polytope~\cite{GZK90,GZK91}, from the cluster complexes of root systems of type~$A_n$~\cite{CFZ02,FZ03}, as a generalized permutahedron~\cite{Postnikov2009}, and
several other constructions with quite remarkable geometric properties~\cite{CeballosSantosZiegler2015,hohlweg_realizations_2007,hohlweg_permutahedra_2011, JoswigKulas2010,Lo04,RoteSantosStreinu2003,ShniderSternberg93}.  

It is natural to ask if $\nu$-Tamari lattices admit similar constructions. 
This question was posed by Bergeron, who in~\cite[Figures~4 and~6]{Bergeron2012} presented geometric realizations of a few small $m$-Tamari lattices as the edge graph of a subdivision of an associahedron (reproduced in Figure~\ref{fig:Bergeron} below) and asked if such realizations exist in general. We provide a positive answer to this question. Our approach gives rise to three (equivalent) geometric realizations.

\begin{theorem}[Corollaries~\ref{cor:IJtriang} and~\ref{cor:cayley} and Theorem~\ref{thm:nurealization}]\label{thm:realizations}
Let $\nu$ be a lattice path from $(0,0)$ to $(a,b)$. The Hasse diagram of the $\nu$-Tamari lattice $\Tam{\nu}$ can be realized geometrically as:
\begin{enumerate}
 \item\label{it:triang} the dual of a regular triangulation of a subpolytope of the Cartesian product of two simplices $\productp{a}{b}$;
 \item\label{it:mixed} the dual of a coherent fine mixed subdivision of a generalized permutahedron (in $\RR^a$ and in $\RR^b$);
 \item\label{it:tropical} the edge graph of a polyhedral complex induced by an arrangement of tropical hyperplanes (in $\TP^a\cong \RR^{a}$ and in $\TP^b\cong\RR^{b}$).
\end{enumerate}
\end{theorem}

\begin{figure}[htbp]
\begin{tabular}{ccccc}
\includegraphics[width= 0.27\textwidth]{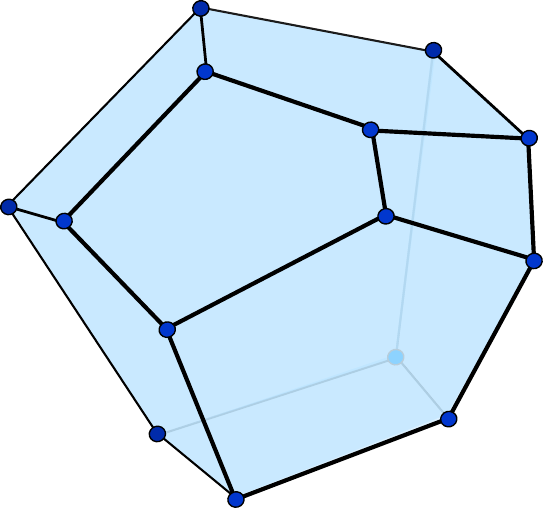}&\qquad& 
\includegraphics[width= 0.25\textwidth]{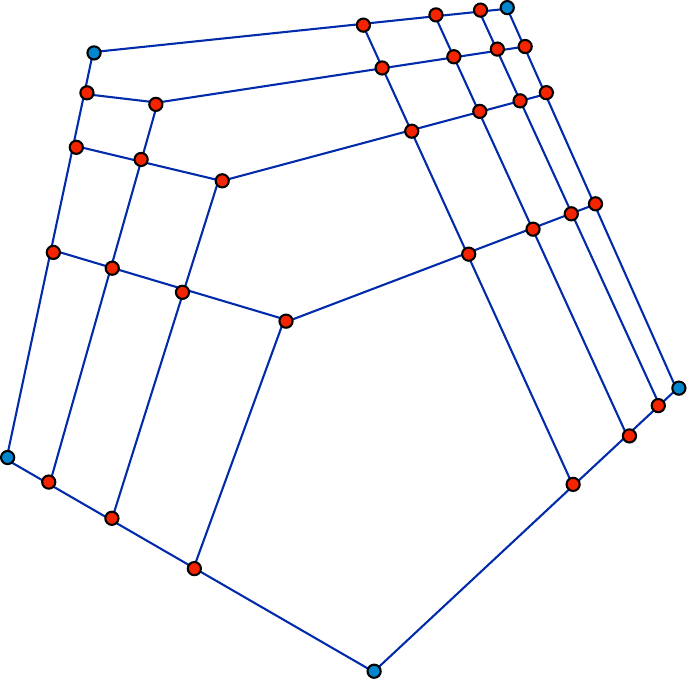}&\qquad&
\includegraphics[width= 0.22\textwidth]{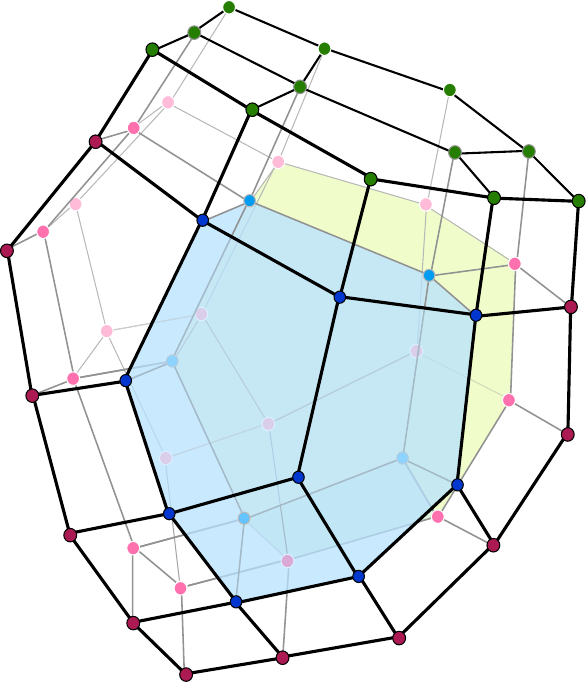}\\
$1$-Tamari $n=4$&&
$4$-Tamari $n=3$&&
$2$-Tamari $n=4$
\end{tabular}
\caption{Bergeron's pictures ``by hand" of $m$-Tamari lattices reproduced with permission from~\cite[Figures 4, 5 and 6]{Bergeron2012}.}
\label{fig:Bergeron}
\end{figure}

\begin{figure}[htbp]
\begin{tabular}{ccccc}
\includegraphics[width= 0.25\textwidth]{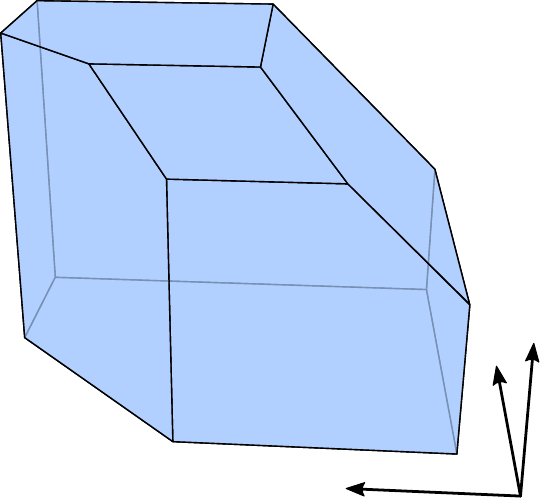}&\qquad& 
\includegraphics[width= 0.22\textwidth]{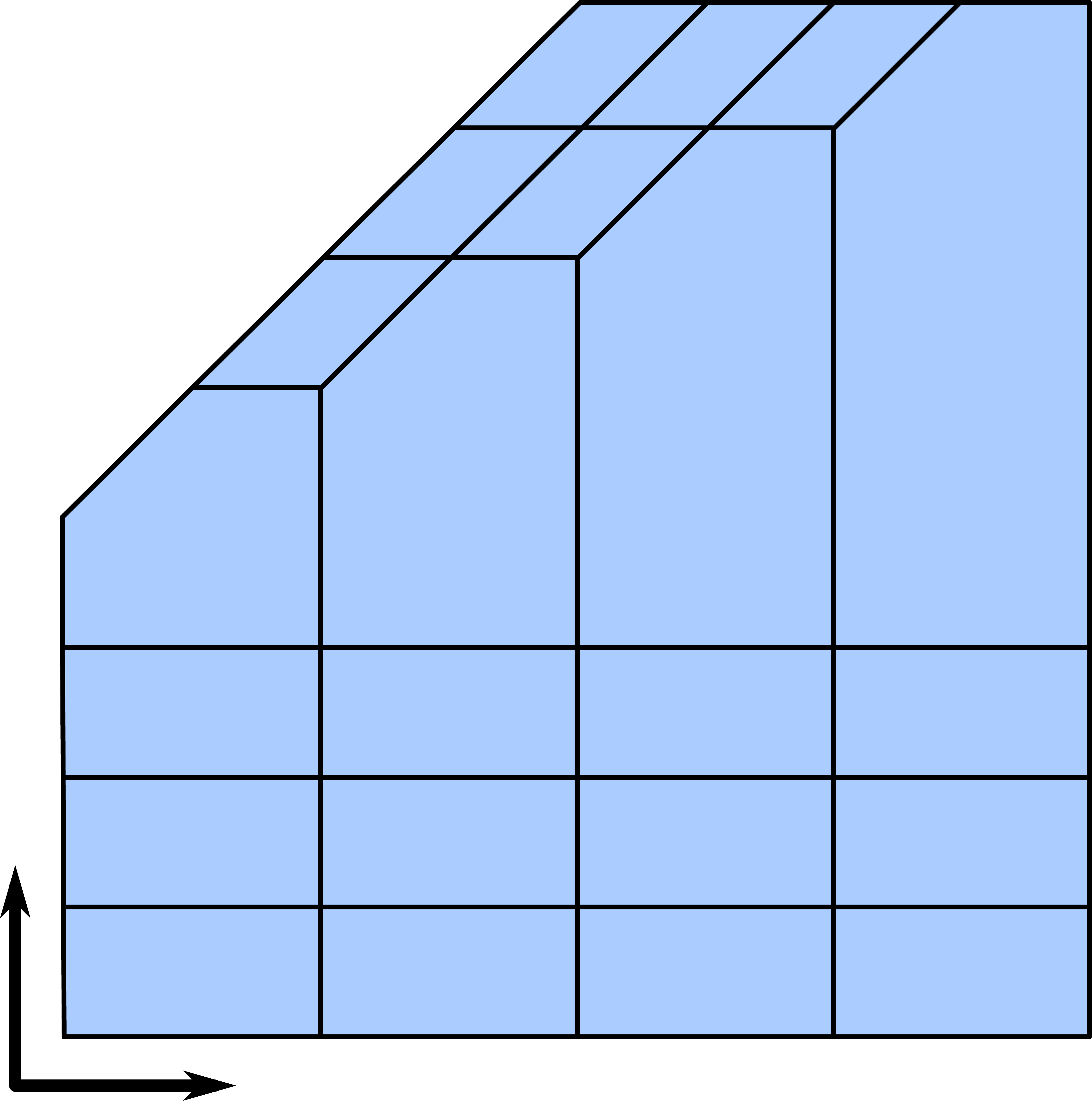}&\qquad&
\includegraphics[width= 0.3\textwidth]{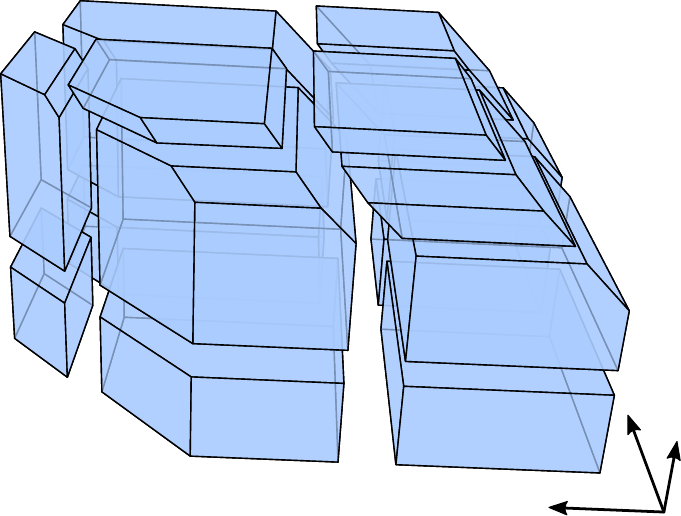}\\
$1$-Tamari $n=4$&&
$4$-Tamari $n=3$&&
$2$-Tamari $n=4$
\end{tabular}
\caption{Geometric realizations of $m$-Tamari lattices by cutting classical associahedra with tropical hyperplanes. Compare with Bergeron's pictures in Figure~\ref{fig:Bergeron}.}
\label{fig:Us}
\end{figure}

Our analogues of Figure~\ref{fig:Bergeron}, depicted in Figure~\ref{fig:Us}, are obtained via the third construction\footnote{Our three-dimensional pictures were produced using \texttt{polymake}~\cite{polymake} and \texttt{SageMath}\cite{sagemath}. Animated versions of few three-dimensional constructions are available in GIF format at~\cite{Web}, and will be included as ancillary files in the final arXiv version.}. 
The polyhedral complex of this tropical realization of $\Tam{\nu}$ is called the \defn{$\nu$-associahedron}~$\Asso_\nu$; in the Fuss-Catalan case we refer to it as the \defn{$m$-associahedron}.

 \begin{remark}\label{rem:othergeometricrealization}
We should point out that there is another simple geometric realization of the $\nu$-Tamari lattice: Since it is an interval of a larger Tamari lattice~(cf. \cite{PrevilleRatelleViennot2017} and Proposition~\ref{prop:interval} below), its Hasse diagram can be obtained as the dual of a subcomplex of the normal fan of Loday's associahedron, which consists of the maximal cones contained in this interval.\footnote{We thank an anonymous referee for pointing this out.} This realization has the disadvantage that its dimension is larger than needed, and no dual convex polytopal subdivision is directly induced. 
For more details on this realization see Remark~\ref{rem:Reading_latticecongruences} and Figure~\ref{fig:Reading_vTamaricongruence}. 
\end{remark}

\begin{theorem}[Corollary~\ref{cor:convex_regsubdivision}
, Proposition~\ref{prop:cells} and Theorem~\ref{thm:linearOrientation}]
\label{thm:subdassociahedra}
The $m$-associahedron is a regular (tropical) subdivision of a classical associahedron into Cartesian products of associahedra. 
More generally, this result holds for any $\nu$-associahedron for which $\nu$ does not contain two (non-initial) consecutive north steps.

The edges of the $\nu$-associahedron can be oriented by a linear functional to give rise to the $\nu$-Tamari lattice.
\end{theorem}

Our starting point is a ubiquitous triangulation of a subpolytope $\catblock{n}$ of the Cartesian product of simplices $\productn$, which essentially dates back (at least) to the work of Gelfand, Graev and Postnikov~\cite{GelfandGraevPostnikov1997} and of Pitman and Stanley~\cite{PitmanStanley2002}. It is known (under a different guise) to be flag, regular and, as a simplicial complex, isomorphic to the join of a simplex with the boundary of a simplicial $(n-1)$-associahedron. We call it the \defn{associahedral triangulation $\asstri{n}$}. In Section~\ref{sec:previously} we briefly
review some of its earlier manifestations in the literature. 

The fact that $\asstri{n}$ is embedded in the product of two simplices has several advantages. For instance, one can consider its restriction
to faces of $\productn$, which are also products of simplices. As we will see, for each lattice path $\nu$ there is a pair $I,\ol J\subseteq[n],[\ol {n}] $ such that the restriction of $\asstri{n}$ to its face $\productp{I}{\ol J}$ induces a triangulation $\asstri{I,\ol J}$ dual to $\Tam{\nu}$. 
Its maximal cells are indexed by \defn{$\IJ$-trees}, which are shown to be in bijection with lattice paths lying weakly above $\nu$. Two $\IJ$-trees are related by a flip if and only if the corresponding paths are related by a covering relation of the $\nu$-Tamari lattice.


\subsection{The \texorpdfstring{$\nu$}{v}-Tamari complex}

The simplicial complex underlying our triangulation in Theorem~\ref{thm:realizations}\eqref{it:triang} gives the $\nu$-Tamari lattice the structure of a 
full simplicial complex, the \defn{$\nu$-Tamari complex}. Its facets correspond to {$\IJ$-trees}, and its lower dimensional faces to \defn{$\IJ$-forests}.
This complex shares several properties with the classical simplicial associahedron and provides definitions for their extensions in the setup of Fuss-Catalan and rational Catalan combinatorics (see Section~\ref{sec:IJcomplex}). 

For example, using a natural shelling order on the $\nu$-Tamari complex, we show that the $\ell$th entry of its $h$-vector is equal to the number of $\nu$-Dyck paths with exactly $\ell$ valleys (Theorem~\ref{thm:thehvector}). These numbers generalize the classical Narayana numbers for classical Dyck paths, and we refer to them as the \defn{$\nu$-Narayana numbers}. In the Fuss-Catalan case, these numbers were considered in~\cite{Athanasiadis2005, AthanasiadisTzanaki2006, FominReading05, Tzanaki2006} (for more general finite Coxeter groups beyond type $A$). In the rational Catalan case, $\nu$-Narayana numbers appeared in work of Armstrong, Rhoades, and Williams~\cite{ARW13}, who introduced a simplicial complex called the rational associahedron, different from ours, whose $h$-vector entries are given by the corresponding $\nu$-Narayana numbers. It would be interesting to understand the differences between the $\nu$-Tamari complex and the rational associahedron.


\subsection{Type \texorpdfstring{$B$}{B} \texorpdfstring{$\IJ$}{IJ}-Tamari posets, complexes and associahedra}

The associahedral triangulation $\asstri{n}$ can be regarded as a ``non-crossing'' object, whose ``non-nesting'' counterpart is the restriction of the staircase triangulation of $\productn$ to certain subpolytope. The first occurrence of the latter we are aware of is in work of Stanley, who constructed it as the standard triangulation of an order polytope~\cite{Stanley1986}; it has also appeared under different guises alongside $\asstri{n}$ in~\cite{GelfandGraevPostnikov1997, PetersenPylyavskyySpeyer2010, SantosStumpWelker2014, PitmanStanley2002} (cf. Section~\ref{sec:previously}).

In our previous work~\cite{CeballosPadrolSarmiento2015} on partial triangulations of $\product$, we constructed the \defn{Dyck path triangulation} of $\productn$ as the orbit of the non-nesting analogue of $\asstri{n}$ under a cyclic action. It turns out that $\asstri{n}$ itself is also amenable to this cyclic symmetry; its orbit under the same cyclic action brings about a flag regular triangulation of $\productn$. Combinatorially, it is the join of a simplex with the boundary complex of a simplicial $n$-cyclohedron. For this reason we call it the \defn{cyclohedral triangulation $\cyctri{n}$}. 
(A related triangulation has been found recently and independently by Ehrenborg, Hetyei and Readdy~\cite{EHR16}, cf. Section~\ref{sec:previously}.)

There are several connections between associahedra and Coxeter groups (see~\cite{TamariFestschrift} and references therein). The \defn{generalized associahedra} are a family of simple polytopes that encode the mutation graphs of cluster algebras of finite types~\cite{FominZelevinsky-ClusterAlgebrasI,FominZelevinsky-ClusterAlgebrasII,FZ03}, and for which various realizations have been  found~\cite{ChapotonFominZelevinsky,hohlweg_permutahedra_2011,PilaudStump-brickPolytopes,ReadingSpeyer,stella_polyhedral_2011}. 
For the $A_n$ root system, one obtains a classical $n$-dimensional associahedron. The generalized associahedron corresponding to $B_n$ is the $n$-dimensional \defn{cyclohedron}. It had appeared earlier in the work of Bott and Taubes~\cite{bott_selflinking_1994}, and was later realized as a convex polytope by Markl~\cite{Markl1999} and Simion~\cite{Sim03}.

We identify maximal simplices of $\cyctri{n}$ with \defn{cyclic $\IJ$-trees}, which, by analogy with $\IJ$-trees, can be naturally given the structure of a poset that we call the \defn{cyclic $\IJ$-Tamari poset}. This poset is a generalization of the type $B$ Tamari lattice, independently discovered by Thomas~\cite{Thomas06} and Reading~\cite{Reading06}. The \emph{Cambrian lattices} introduced by Reading~\cite{Reading06} extend the Tamari lattice even further in the context of finite Coxeter groups. 
Their Hasse diagrams can be realized geometrically as the edge graphs of the aforementioned {generalized associahedra}. 

Using the same techniques as in type~$A$, we obtain type~$B$ analogues of Theorem~\ref{thm:realizations} (Corollary~\ref{IJcyclic_triangulation}, Theorem~\ref{thm:cyclicnurealization} and Theorem~\ref{thm:cyclicverticesandcoordinates}). Figures~\ref{fig:FussCatalanCyclohedra} and~\ref{fig:FussCatalan3-Cyclohedra} display $\IJ$-cyclohedra corresponding to the first few Fuss-Catalan cases in dimensions two and three. Note that they are polyhedral subdivisions of classical cyclohedra into Cartesian products of associahedra and at most one cyclohedron, see Theorem~\ref{thm:supportCyclohedron} and Proposition~\ref{prop:cells_IJcyclohedron}.

\begin{figure}[h]
\includegraphics[width= \textwidth]{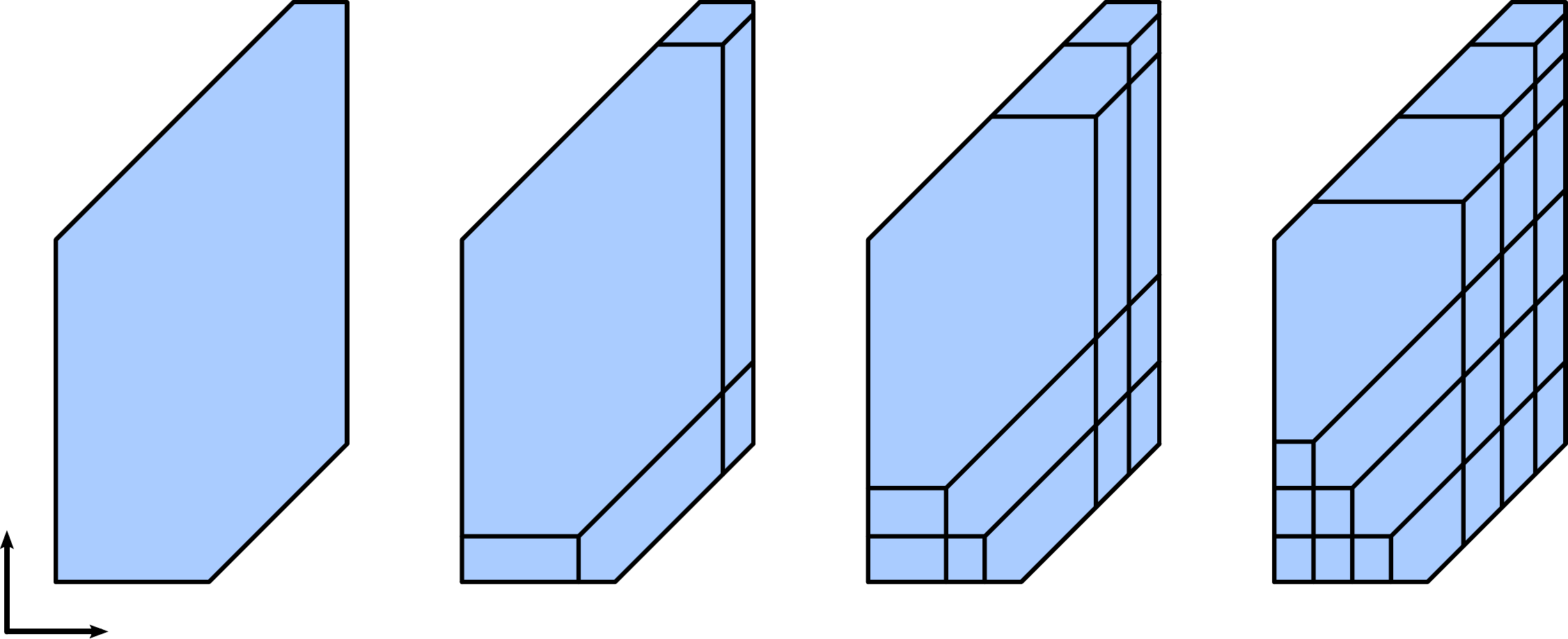}
\caption{Some Fuss-Catalan $\IJ$-cyclohedra in dimension two.}
\label{fig:FussCatalanCyclohedra}
\end{figure}

\begin{figure}[h]
\includegraphics[width= 0.9\textwidth]{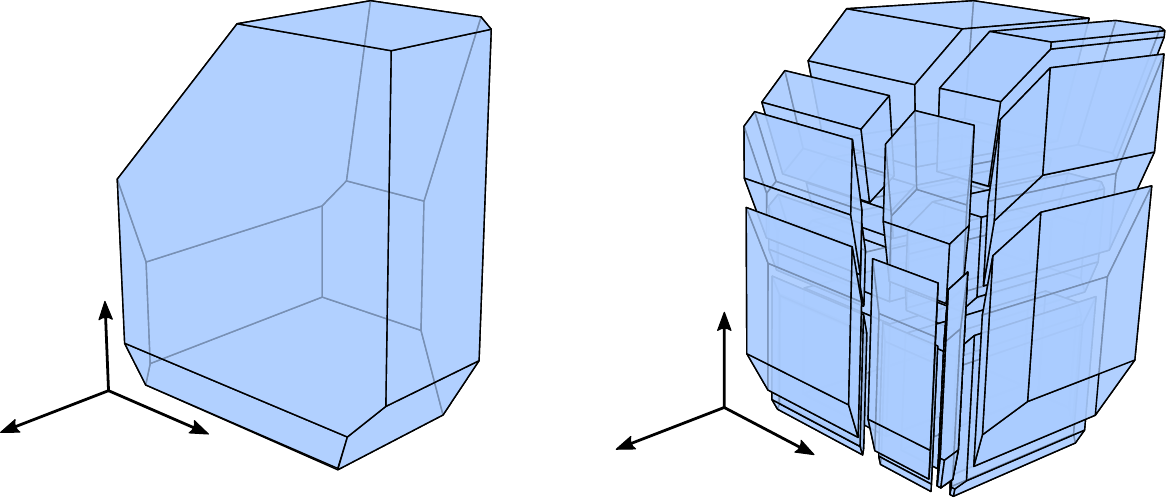}
\caption{Some Fuss-Catalan $\IJ$-cyclohedra in dimension three.}
\label{fig:FussCatalan3-Cyclohedra}
\end{figure}


\subsection{The associahedral triangulation in the literature}\label{sec:previously}

Relatives of the associahedral triangulation have been found independently several times, under various guises, in a number of different contexts. To the best of our knowledge, it was first constructed by Gelfand, Graev and Postnikov in~\cite{GelfandGraevPostnikov1997} as the triangulation of the polytope $\conv\{0,\bfe_i-\bfe_j\colon 1\leq i<j \leq n\}\subset\RR^n$, also known as the \emph{root polytope} of type $A_{n-1}$. Their triangulation arises from $\asstri{n}$ by projecting $\catblock{n}$ along the subspace spanned by the points $\{(\bfe_i,\bfe_{\ol i})\colon i\in [n]\}$. Independently, in~\cite{PitmanStanley2002}, Pitman and Stanley introduced the \emph{Pitman-Stanley polytope} as a section of the \emph{order polytope} of the product of a $2$-chain with a $n$-chain, along with a polyhedral subdivision of it connected with the associahedron. This order polytope is the image of $\catblock{n}$ under a unimodular transformation, and it is a nice exercise to show that the associahedral triangulation $\asstri{n}$ is essentially the result of applying the Cayley trick from~\cite{HRS00} to their ``associahedral'' subdivision of the Pitman-Stanley polytope\footnote{Incidentally, the similarity between~\cite[Figure~2]{PitmanStanley2002} and~\cite[Figure~6]{CeballosPadrolSarmiento2015} originated this project.}. More recently, Petersen, Pylyavskyy and Speyer considered a family of triangulations of cones of \emph{Gelfand-Tsetlin patterns}, from an alternative viewpoint, of which the associahedral triangulation is a special case~\cite{PetersenPylyavskyySpeyer2010}. Some of their results were also discovered independently by Santos, Stump and Welker in their investigation of non-crossing and non-nesting triangulations of the order polytope of the product of a $k$-chain with a $n$-chain, which are slices of a subfamily of the cones in~\cite{PetersenPylyavskyySpeyer2010}. 

In~\cite{Mes11a}, M\'esz\'aros studied a remarkable connection between subdivisions of \emph{acyclic root polytopes} and reduced forms of monomials in certain associative algebras. These polytopes are defined from a non-crossing tree, and generalize the root polytope of type $A_{n-1}$, which arises as special case when the non-crossing tree is the path graph on $\{1,\ldots,n\}$. Although the subpolytopes of $\catblock{n}$ we consider project to a family of polytopes that strictly contains the acyclic root polytopes of~\cite{Mes11a}, we lack M\'esz\'aros' beautiful algebraic interpretation of subdivisions. It would be interesting to know if such an interpretation is possible in our slightly more general setting.

To the best of our knowledge, the realization of the classical associahedron as a polyhedral cell in a tropical hyperplane arrangement was first constructed by Rote, Santos and Streinu in~\cite{RoteSantosStreinu2003}, although its tropical nature was first pointed out by Joswig and Kulas in~\cite{JoswigKulas2010}. 

While preparing the final version of this manuscript, we became aware of the recent work by Ehrenborg, Hetyei and Readdy in~\cite{EHR16}. There, the authors realize Simion's type~$B$ associahedron as a pulling triangulation of the boundary of the \emph{Legendre polytope} $\conv\{\bfe_i-\bfe_j\colon 1 \leq i, j\leq n,\ i\neq j\}$, also known as the \emph{full root polytope of type $A_{n-1}$}~\cite{EHR16}. We can recover their triangulation by projecting our cyclohedral triangulation of $\productn$ along the span of the vectors $\{(\bfe_i,\bfe_{\ol i})\colon i\in [n]\}$. Our perspective gives a geometric justification for the formula $h_i=\binom{n}{i}^2$ expressing the entries of the $h$-vector of Simion's type $B$ associahedron (see Theorem~\ref{thm:typeBhVector}).


\subsection{Structure of the paper}

Most of the results of this paper appear in two different versions, for types~$A$ and $B$. To avoid too much repetition, whenever the two proofs are essentially the same, we tried to put it only for type~$A$ and leave type~$B$ for the reader; whenever $A$ is a direct corollary of a more involved type~$B$ argument, we usually put the second; and when each type had its inherent subtlety, we wrote two proofs or sketched the changes for type~$B$. Nevertheless, the first half of the paper is mostly self-contained and can be read with only few jumps to the second.

There are two ways to read the complete paper. One can read it as presented, first type~$A$ and then type~$B$. This has the advantage of having a compact presentation for each type, but while reading the second part it might be necessary to occasionally go back. Alternatively, one can read the sections alternating between the two types: \ref{sec:assTri}$\to$\ref{sec:cyctri}$\to$\ref{sec:IJTamari}$\to$\ref{sec:cycTamari}$\to$\ref{sec:IJcomplex}$\to\cdots$


\part{Type \texorpdfstring{$A$}{A}}

\section{The associahedral triangulation}

\label{sec:assTri}

In this paper, the terms \emph{vertex} and \emph{edge} appear in several contexts. To avoid confusion, we reserve the names \emph{vertices} and \emph{edges} exclusively for simplicial and polyhedral complexes, and call \emph{nodes} and \emph{arcs} the corresponding graph notions.

Let $\NN$ denote the set of natural numbers including the zero, and $\ol\NN$ the same set with numbers decorated with an overline and a total order $<$ inherited from $\NN$. If $n$ is a natural number, define $[n]:=\{0,1,\ldots, n\}$, and likewise $[\ol n]:=\{\ol 0,\ol 1,\ldots, \ol n\}$\footnote{\emph{Caveat:} Note that our convention departs from the more standard definition $[n]=\{1,\ldots, n\}$ used in combinatorics. We have made this choice to avoid other nonstandard variants such as~$[0..n]$, which are notationally cumbersome for our purposes.}. Regard $\NN\sqcup \ol \NN$ as the totally ordered set with covering relations $i\prec\ol i$ and $\ol i \prec i+1$. 

The Cartesian product of two standard simplices is the convex polytope:
\begin{equation*}
\product:=\conv\left\{ (\bfe_i, \bfe_{\ol j} )\colon i \in [n],\ \ol j\in[\ol m]\right\}\subset\RR^{n+m+2},
\end{equation*}
where $\bfe_i$ and $\bfe_{\ol j}$ denote the standard basis vectors of $\RR^{n+1}$ and $\RR^{m+1}$, respectively. 
Overlined indices are introduced to distinguish the labels of the two factors~\footnote{We try to be especially consistent with this notation when we refer to the \emph{value} $j$ of an index $\ol j\in \NN$ as opposed to just its label (see for example Lemma~\ref{lem:non-crossing}, where a function of the index pair $(i,\ol j)$ depends on the difference $j-i$.)}.

The main ingredient for the results in this paper is what we call the \defn{associahedral triangulation} of a \defn{subpolytope} of $\productn$, where subpolytope means that its vertices are vertices of $\productn$. Recall that a triangulation of a polytope~$P$ is a collection of subpolytopes of $P$ that are simplices, intersect properly (in the sense that they intersect in common faces and their relative interiors are disjoint), and cover~$P$ (see~\cite{DeLoeraRambauSantos2010} for a very complete treatment of triangulations of polytopes and point configurations).

In order to describe this triangulation we consider a convex $(n+2)$-gon $P_{n+2}$, 
whose edges we label counterclockwise from~$0$ to~$n+1$, and which we depict with $n+1$ as its single upper edge, as in Figure~\ref{fig:triang2tree}.
To each triangulation $T$ of $P_{n+2}$ we will associate a spanning tree of $\bipartitebn$ as follows. 
First replace each boundary edge $0\leq i\leq n$ of $P_{n+2}$ with an arc $(i,\ol{i})$ of~$\bipartitebn$. 
Then, turn every remaining diagonal of $T$ into an arc $(i,\ol j)$ of $\bipartitebn$, where $i$ and~$j$ are the leftmost and rightmost boundary edges of $P_{n+2}$ covered by the diagonal, respectively. See Figure~\ref{fig:triang2tree}, where nodes from $\NN$ are drawn black and nodes from $\ol \NN$ white, as in the rest of this paper. 

It is not hard to see that the result is always a spanning tree of $\bipartitebn$ (see Lemma~\ref{lem:cyclicIJtreesaretrees} for a proof). 

\begin{figure}[htpb]
\centering 
 \includegraphics[width=\linewidth]{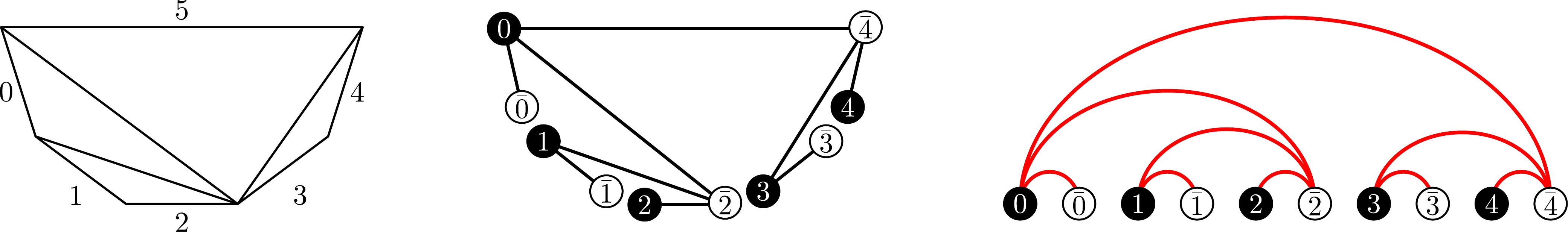}
 \caption{From a triangulation to a non-crossing alternating tree.}\label{fig:triang2tree}
\end{figure}

We call the trees obtained this way \defn{$([n],[\ol n])$-trees}. 
Observe that each $([n],[\ol n])$-tree indexes a simplex of $\productn$ that belongs to the subpolytope $\catblock{n}$:
\begin{equation*}
\catblock{n}:=\conv\left\{ (\bfe_i, \bfe_{\ol j} ) \colon 0\preceq i\prec \ol j\preceq \ol n\right\}\subset\productn.
\end{equation*}

If we take all such trees, we obtain a triangulation of $\catblock{n}$, which we call the \defn{Associahedral triangulation $\asstri{n}$}. The name is motivated by the fact our map from triangulations of $P_{n+2}$ to $\nn$-trees is actually a bijection (see Remark~\ref{rem:bijectionTriTree}); moreover, two maximal cells of $\asstri{n}$
are adjacent if and only if the corresponding triangulations of $P_{n+2}$ differ by a flip. This implies that the dual of $\asstri{n}$ is the (simple) associahedron. 
That $\asstri{n}$ is indeed a triangulation has been discovered in different contexts (see Section~\ref{sec:previously}). It is an immediate corollary of Theorem~\ref{thm:cyclohedrontriangulation}, for which we provide a complete proof. 

\begin{theorem}[{\cite[Theorem 6.3]{GelfandGraevPostnikov1997}, \cite[Theorem 18]{PitmanStanley2002}}]\label{thm:asstriang}
The set of $([n],[\ol n])$-trees indexes the maximal simplices of a flag regular triangulation of $\catblock{n}$, which we call the \defn{$n$-associahedral triangulation $\asstri{n}$}.
\end{theorem}
 
\begin{remark}\label{rmk:triang2dualtree}
 Consider a triangulation $T$ of $P_{n+1}$ and the corresponding $([n],[\ol n])$-tree. Representing the arcs $(i,\ol j)$ of the latter as points of an $n\times \ol n$ square grid, we get a \emph{grid representation} of $T$ (as in the grid representation of triangulations of products of simplices~\cite[Sec.~6.2]{DeLoeraRambauSantos2010}). To $T$ corresponds a dual rooted binary tree $\cT$ on $n$ leaves, as shown in Figure~\ref{fig:triang2dualtree} (left), from which we can also obtain the grid representation of $T$ directly: For each node $v$ of $\cT$, we collect the pair $(i,\ol j)$ where~$i$ is the number of leaves of $\cT$ strictly preceding $v$ when $\cT$ is traversed in \emph{preorder} (Figure~\ref{fig:triang2dualtree} (center)), and ${j+1}$ is the number of leaves of $\cT$ weakly preceding $v$ when $\cT$ is traversed in \emph{postorder} (Figure~\ref{fig:triang2dualtree} (right)).  
\end{remark}

\begin{figure}[htpb]
\centering 
 \includegraphics[width=.8\linewidth]{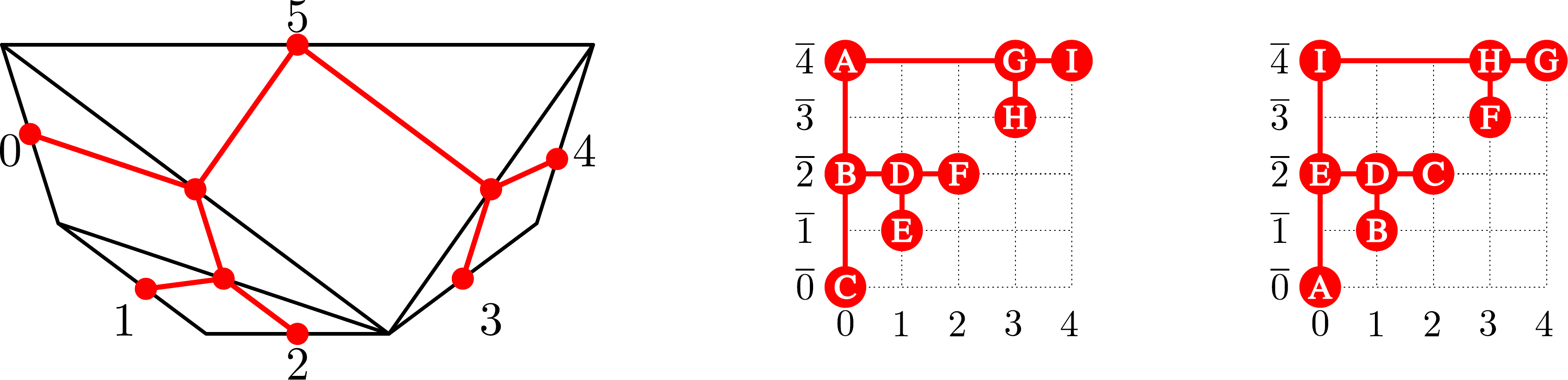}
 \caption{Grid representation of the dual binary tree. Latin letters indicate the preorder (center) and the postorder (right) traversal of the vertices of the binary tree.}\label{fig:triang2dualtree}
\end{figure}


\subsection{The \texorpdfstring{$\IJ$}{IJ}-associahedral triangulation}

As faces of Cartesian products are Cartesian products of faces, 
faces of $\productn$ are of the form
\begin{equation*}
\productp{I}{\ol J}=\conv\left\{ (\bfe_i, \bfe_{\ol j} )\ :\  i\in I,  \ol j\in\ol J\right\}
\end{equation*}
for pairs $(I,\ol J)$ with $I\subseteq [n]$ and $\ol J\subseteq [\ol n]$.
Since supporting hyperplanes for faces $\productp{I}{\ol J}$ are also supporting for any subpolytope of $\productn$, we see that $\catblock{n}\cap \productp{I}{\ol J}$ is a face of $\catblock{n}$. We write it as follows:
\begin{equation*}
 \catblock{I,\ol J}:=\catblock{n}\cap \productp{I}{\ol J}=\conv\left\{ (\bfe_i, \bfe_{\ol j} )\ \colon\ i\in I, \ol j\in \ol J \text{ and } i\prec \ol j\right\}.
\end{equation*}
Notice that those $\ol j\in\ol J$ preceding $\min(I)$ and those $i\in I$ larger than $\max(\ol J)$ are irrelevant in the definition of $\catblock{I,\ol J}$. For this reason, through the type~$A$ part of the paper (Sections~\ref{sec:assTri} to \ref{sec:nuassociahedron}), we will implicitly assume that $\min(I\sqcup\ol J)\in I$ and $\max(I\sqcup \ol J)\in \ol J$. 

Every (regular) triangulation of a polytope induces a (regular) triangulation on each of its faces. In particular, the associahedral triangulation $\asstri{n}$ of $\catblock{n}$ induces a regular triangulation $\asstri{I,\ol J}$ of $\catblock{I,\ol J}$, whose cells are the intersection of the cells of $\asstri{n}$ with $\productp{I}{\ol J}$. 
We call this restricted triangulation the \defn{$\IJ$-associahedral triangulation}.
Its maximal cells are given by $\IJ$-trees, which can be characterized as follows.

Given nonempty finite sets $I\subset\NN$ and $\ol J\subset \ol \NN$, denote by \defn{$\bipartiteij$} the complete bipartite graph with node set $I\sqcup \ol J$ and arc set $\{(i,\ol j)\colon i\in I,\ \ol j\in\ol J\}$. 

\begin{definition}\label{def:IJtree}
Let $I$ and $\ol J$ be nonempty finite subsets of $\NN$ and $\ol \NN$, respectively, such that $\min(I\sqcup\ol J)\in I$ and $\max(I\sqcup \ol J)\in \ol J$. An \defn{$\IJ$-forest} is a subgraph of~$\bipartiteij$ that is
\begin{enumerate}
 \item \textbf{Increasing:} each arc $(i,\ol j)$ fulfills $i\prec \ol j$ (i.e. $i\leq j$); and
 \item \textbf{Non-crossing:} it does not contain two arcs $(i,\ol j)$ and $(i', \ol j')$ satisfying ${i\prec  i'\prec \ol j \prec \ol j '}$  (i.e. $i<  i'\leq j <  j '$).
\end{enumerate} 
An \defn{$\IJ$-tree} is a maximal $\IJ$-forest.
\end{definition}

\begin{remark}
\label{rem:bijectionTriTree}
The $\nn$-trees arising from this definition are exactly those that we obtained with the construction of Figure~\ref{fig:triang2tree}, and are in bijection with triangulations of a convex $(n+2)$-gon $P_{n+2}$. Indeed, it is straightforward to see that the aforementioned construction provides $\nn$-trees. For the converse, notice that each $\nn$-tree is a tree in the graph theoretical sense (cf. Lemma~\ref{lem:cyclicIJtreesaretrees}), and therefore it contains exactly $2n+1$ arcs, corresponding to the $n+2$ boundary edges and a set of $n-1$ non-crossing diagonals of $P_{n+2}$; that is, a triangulation.
\end{remark}

An example of $\IJ$-tree is depicted in Figure~\ref{fig:IJtrees}, as well as its ``completion" to an $\nn$-tree where all nodes not in $I\sqcup \ol J$ are added as leaves.

\begin{figure}[htpb]
\centering 
 \includegraphics[width=.7\linewidth]{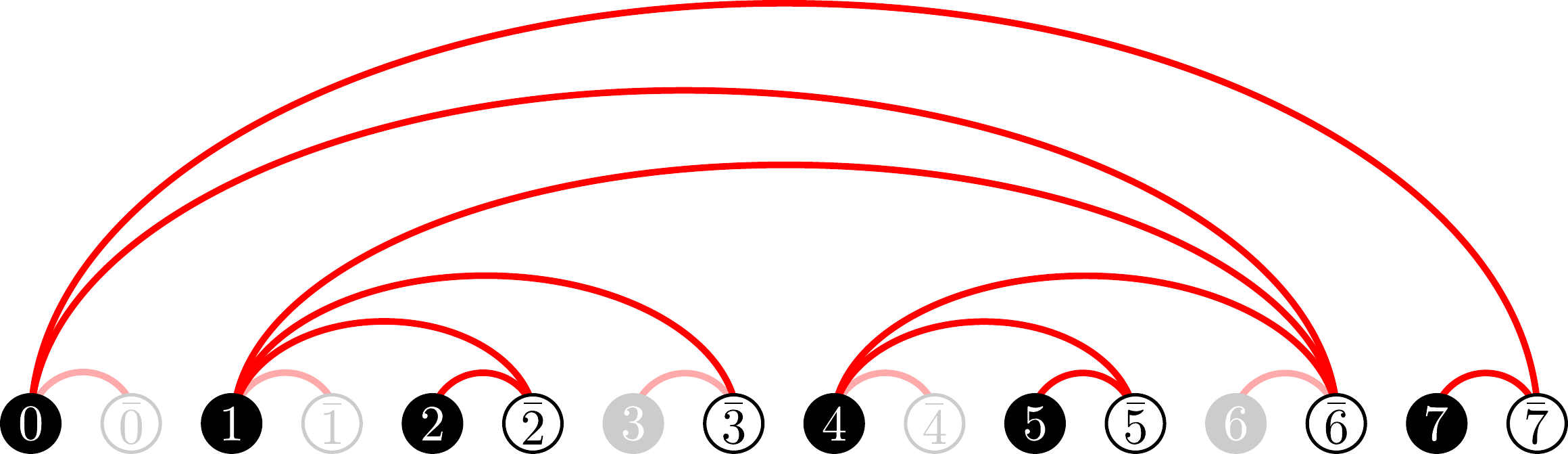}
 \caption{An $\IJ$-tree for $I=\{0,1,2,4,5,7\}$ and $\ol J=\{\ol 2, \ol 3, \ol 5, \ol 6, \ol 7\}$. Adding shaded vertices gives its completion to an $\nn$-tree.}\label{fig:IJtrees}
\end{figure}

\begin{corollary}
\label{cor:IJtriang}
The set of $\IJ$-trees indexes the maximal simplices of a flag regular triangulation of $\catblock{I,\ol J}$. We call it the \defn{$\IJ$-associahedral triangulation $\asstri{I,\ol J}$}.
\end{corollary}

Recall that every height function $\ha\colon \{(i,\ol j)\in I\times \ol J\colon i\prec \ol j \}\to \RR$ induces a regular subdivision of $\catblock{I,\ol J}$ as the projection of the lower envelope of $\conv\{(\bfe_i,\bfe_{\ol j},\ha(i,\ol j))\colon (\bfe_i,\bfe_{\ol j})\in \catblock{I,\ol J}\}$, which is a triangulation whenever $\ha$ is sufficiently generic (see~\cite[Section~4.3]{DeLoeraRambauSantos2010} for a treatment of regular triangulations of point configurations). Regularity has many important implications, and is key for our subsequent development. Below is a characterization of the polyhedral cone of height functions on the vertices of $\catblock{I,\ol J}$  that induce $\asstri{I,\ol J}$ as a regular triangulation. It follows as a direct application of \cite[Lemma~3.3]{SantosStumpWelker2014}, which we have included as Lemma~\ref{lem:regularitycharacterization}, specialized to triangulations of $\product$. 

\begin{proposition}\label{prop:assheights}
Let $\ha:\{(i,\ol j)\in I\times \ol J\colon i\prec \ol j \}\to \RR$ be a height function. 
The regular triangulation of $ \catblock{I,\ol J}$ induced by $\ha$ is $\asstri{I,\ol J}$ if and only if 
\begin{equation}\label{eq:propnoncrossing}
\ha(i,\ol j)+\ha(i',\ol{j'})<\ha(i',\ol j)+\ha(i,\ol{j'}) 
\end{equation}
whenever $i<i'<j'<j$. We say any such height function is \defn{non-crossing}. 
\end{proposition}

\begin{lemma}\label{lem:non-crossing}
The height function $\ha(i,\ol j)= f(j-i)$ is non-crossing for any strictly concave function $f$. 
Explicit examples of non-crossing height functions are $\ha({i,\ol j})=-{(j-i)^2}$, $\ha({i,\ol j})=\sqrt{j-i}$,  
or ${\ha({i,\ol j})=-c^{i-j}}$ for some $c>1$. 
\end{lemma}

In particular, the height function ${\ha({i,\ol j})=-c^{i-j}}$ shows that $\asstri{I,\ol J}$ is the pulling triangulation of $\catblock{I,\ol J}$ with respect to every order of $\{(i,\ol j)\in I\times \ol J\colon i\prec \ol j \}$ that extends the partial order 
${(i,\ol j)<(i',\ol j') \Leftrightarrow j-i > j'-i'}$ (see~\cite[Section~4.3.2]{DeLoeraRambauSantos2010}). 

\begin{proof}
If $f$ is strictly concave then $f(a+x)-f(a)>f(b+x)-f(b)$ for any $a<b$ and $x\in \RR$. 
Let $i<i'<j'<j$ and consider $a=j'-i'$, $b=j-i'$ and $x=i'-i$. We have that 
\begin{align*}
f(a+x)-f(a)&>f(b+x)-f(b), \\
f(j'-i)-f(j'-i')&>f(j-i)-f(j-i'), \\
\ha(i, \ol j')-\ha(i', \ol j')&>\ha(i,\ol j)-\ha(i',\ol j).
\end{align*}
This shows that $\ha$ satisfies Equation~\eqref{eq:propnoncrossing} as desired.
\end{proof}


\subsection{Geometric realization as a subdivision of a generalized permutahedron}

To conclude this section, we present an alternative representation of triangulations of subpolytopes of $\product$ via \emph{fine mixed subdivisions of generalized permutahedra}. It allows us to make pictures of high dimensional triangulations in fewer dimensions. This is done via the ``Cayley trick''~\cite{HRS00}, that provides a bijection between triangulations of Cayley polytopes and mixed subdivisions of Minkowski sums.
We only include the definitions needed for our purposes, and refer the reader to~\cite{Postnikov2009} for a thorough treatment.

\begin{definition}
A \defn{generalized permutahedron} is a polytope whose normal fan coarsens that of the standard permutahedron. One particular family of generalized permutahedra can be obtained in terms of Minkowski sums of faces of a simplex. 
Given a nonempty subset $I\subseteq[n]$, define $\Delta_{I}:=\conv\{\bfe_i\colon i\in I\}\subset\RR^{n+1}$ where, as before $\bfe_i$ ranges over the standard basis vectors of $\RR^{n+1}$. 
Any Minkowski sum of the form:
\begin{equation*}
\sum_{I\subseteq [n]} y_I \Delta_I,
\end{equation*}
where the $y_I$ are nonnegative real numbers, is a generalized permutahedron.
\end{definition}

Let $I\subset\NN$, $\ol J\subset\ol \NN$ be nonempty finite sets. We associate to $\catblock{I,\ol J}$ and to an $\IJ$-tree $T$ the generalized permutahedra $\cP_{I,\ol J}$ and $\cP_T$, respectively, as follows:
\begin{align*}
\cP_{I,\ol J}:=\sum_{i\in I} \Delta_{\{\ol j \in \ol J\colon i\prec \ol j\}}, && \cP_T:=\sum_{i\in I} \Delta_{\{\ol j\colon (i,\ol j)\in T\}}
\end{align*}

\begin{corollary}
\label{cor:cayley}
The polyhedral cells $\cP_T$ as $T$ ranges over $\IJ$-trees 
form a coherent fine mixed subdivision of $\cP_{I,\ol J}$ (see Figure~\ref{fig:mixed35}).
\end{corollary}

\begin{figure}[tbhp]
\centering 
 \includegraphics[width=\linewidth]{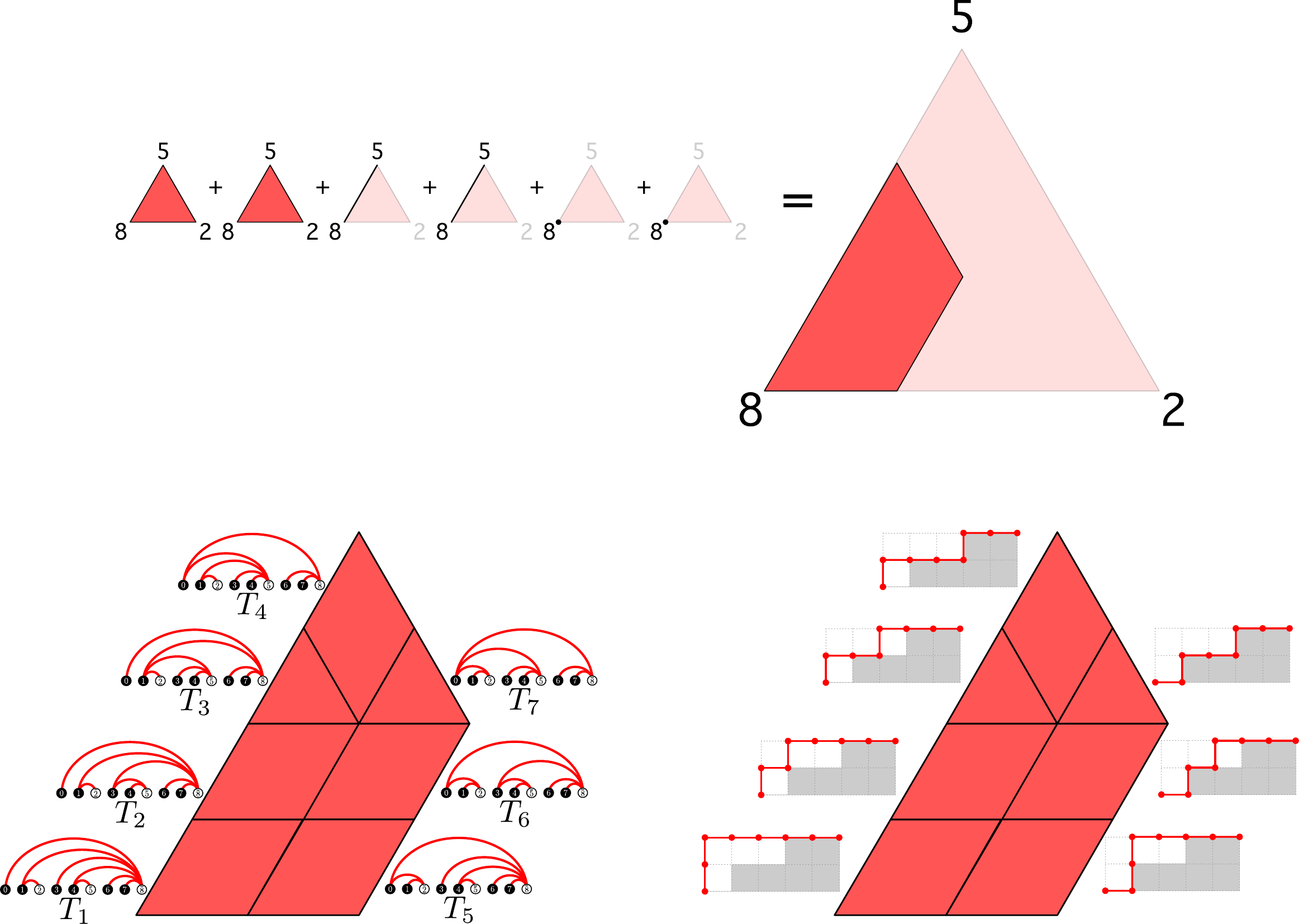}
 \caption{Top: The generalized permutahedron $\cP_{I,\ol J}$, for $I=\{0,1,3,4,6,7\}$, $\ol J=\{\ol 2, \ol 5, \ol 8\}$. Bottom left: Its fine mixed subdivision corresponding to the $\IJ$-associahedral triangulation. Bottom right: The same subdivision labeled by lattice paths (see Section~\ref{sec:IJTamari}).}\label{fig:mixed35}
\end{figure}


\section{The \texorpdfstring{$\protect{\IJ}$}{IJ}-Tamari lattice and \texorpdfstring{$\nu$}{v}-Dyck paths}
\label{sec:IJTamari}

In this section, we endow $\IJ$-trees with a partial order structure $\Tam{I,\ol J}$. We will present then a natural bijection between $\IJ$-trees and lattice paths lying above a certain fixed path $\nu=\nu\IJ$, which are known as $\nu$-Dyck paths. With it, we show that
$\Tam{I,\ol J}$ is isomorphic to the $\nu$-Tamari lattice from~\cite{PrevilleRatelleViennot2017}.
This lattice structure will be exploited later in Section~\ref{sec:Narayana} to construct a shelling of $\asstri{I,\ol J}$ and compute its $h$-vector. 


\subsection{Flips and the \texorpdfstring{$\IJ$}{IJ}-Tamari lattice}

We say that two $\IJ$-trees $T$ and $T'$ are related by a \defn{flip} if they share all the arcs but one; that is, there are arcs $\ij$ and $(i',\ol j')$ such that $T':=T\setminus \ij\cup (i',\ol j')$.
In such a case, we say that the arc $\ij$ is \defn{flippable}. It is easy to see that $\ij$ is flippable if and only if it is neither a leaf (some endpoint has degree~$1$) nor $(\min I, \max \ol J)$. 

Flips come in two flavors, depending on the position of $(i',\ol j')$: a flip is called \defn{increasing} if $i'>i$ and $\ol j'>\ol j$, and \defn{decreasing} if $i'<i$ and $\ol j'<\ol j$. Thus, in an 
increasing flip, $i'$ is the smallest node of $I$ such that $i'>i$ and  $(i',\ol j)\in T$, and $j'$ is the smallest node of $\ol J$ such that $\ol j'>\ol j$ and $(i,\ol j')\in T$. The description of $i'$ and $\ol j'$ for a decreasing flip is analogous.  
Notice in particular that a flippable arc $(i,\ol j)$ supports an increasing flip if and only if $i$ is the smallest neighbor of $\ol j$.
An increasing flip is depicted both schematically and for a particular example in Figure~\ref{fig:IJflipschema}.

\begin{figure}[htpb]
\centering 
 \includegraphics[width=\linewidth]{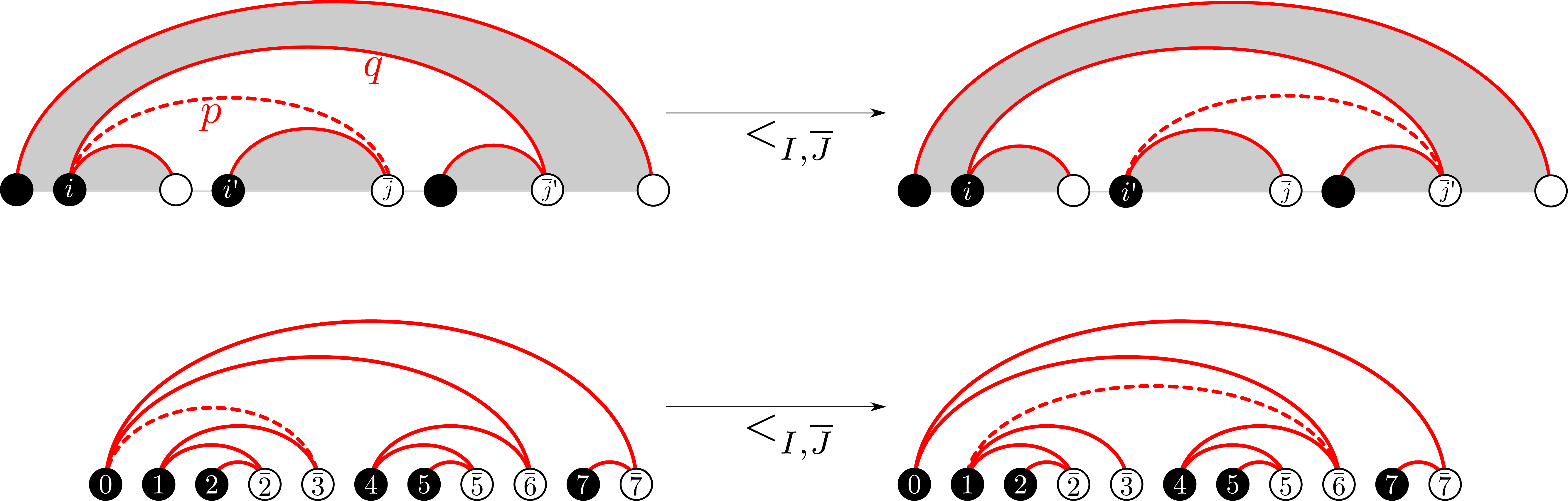}
 \caption{An increasing $\IJ$-flip schematically (top) and in a particular $\IJ$-tree (bottom). The flipped arc is shown dashed.}\label{fig:IJflipschema}
\end{figure}

Let $T,T'$ be $\IJ$-trees. Introduce a cover relation $T<_{I,\ol J} T'$ whenever $T'$ is obtained from $T$ by an increasing flip. 

\begin{lemma}\label{lem:IJposet}
The transitive closure of the relation $T <_{I,\ol J} T'$ is a partial order on the set of $(I,\ol J)$-trees. 
\end{lemma}

Lemma~\ref{lem:IJposet} is a direct corollary of Lemma~\ref{lem:cyclicIJposet}. We call this partially ordered set the \defn{$(I,\ol J)$-Tamari lattice}, and denote it by \defn{$\Tam{I,\ol J}$} (that it is a lattice will be proved in Proposition~\ref{prop:interval}). Two examples are illustrated in Figure~\ref{fig:IJTamaris}.

\begin{figure}[htpb]
\centering 
 \includegraphics[width=.9\linewidth]{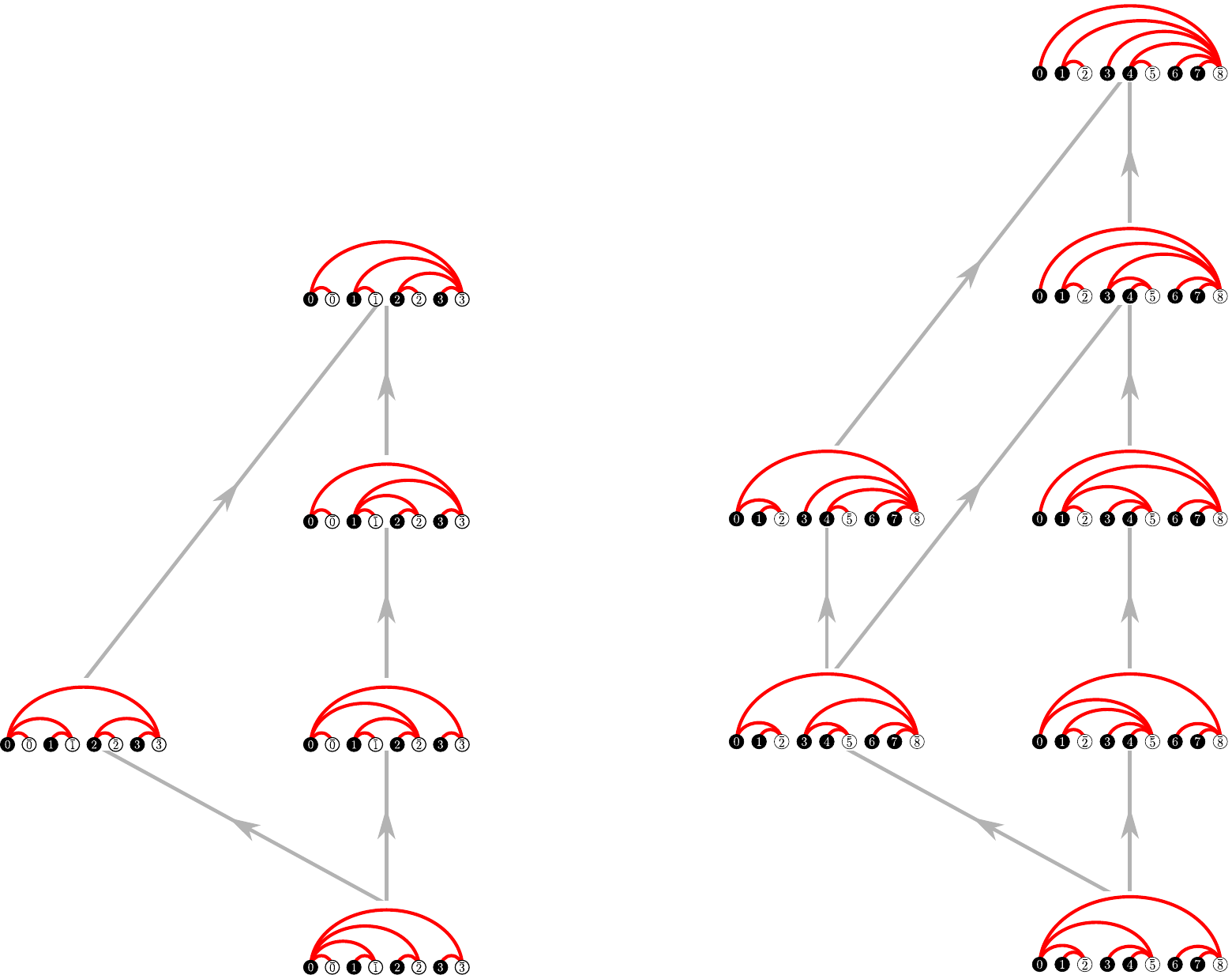}
 \caption{The $\IJ$-Tamari lattices for $\IJ=(\{0,1,2,3\},\{\ol 0, \ol 1, \ol 2, \ol 3\})$ (left) and
  for $\IJ=(\{0,1,3,4,6,7\},\{\ol 2,\ol 5,\ol 8\})$ (right).}\label{fig:IJTamaris}
\end{figure}

\subsection{\texorpdfstring{$\nu$}{v}-Dyck paths and the \texorpdfstring{$\nu$}{v}-Tamari lattice}
\label{subsec:vTamari}

We identify \defn{lattice paths} that start at $(0,0)$ and consist of a finite number of north and east unit steps with words on the alphabet $\{\sfN,\sfE\}$. 
For a fixed lattice path~$\nu$, a \defn{$\nu$-Dyck} path is a lattice path that is weakly above $\nu$ and finishes at the end of $\nu$. In terms of words, a $\nu$-Dyck path is a word that contains the same number of occurrences of $\sfN$ and $\sfE$ as $\nu$, and such that each prefix has at least as many $\sfN$ steps as the corresponding prefix in $\nu$. Note that, if $\nu$ is the path $(\sfN\sfE)^n$, this definition recovers classical Dyck paths, and if $\nu$ is $(\sfN\sfE^ m)^n$ one obtains their Fuss-Catalan generalization, $m$-Dyck paths. Rational Dyck paths~\cite{ARW13} can also be interpreted as special cases of $\nu$-Dyck paths.

Recently, Pr\'eville-Ratelle and Viennot~\cite{PrevilleRatelleViennot2017} endowed the set of $\nu$-Dyck paths with a poset structure that generalizes the Tamari and the $m$-Tamari lattices. It is induced by the following covering relation:

Let $\mu$ be a $\nu$-Dyck path. For a lattice point $p$ on $\mu$ define the distance $\horiz_\nu(p)$ to be the maximum number of horizontal steps that can be added to the right of $p$ without crossing $\nu$. Suppose that $p$ is a \defn{valley}, i.e. it is preceded by an east step $\sfE$ and followed by a north step $\sfN$. Let $q$ be the first lattice point in $\mu$ after $p$ such that $\horiz_\nu(q)=\horiz_\nu(p)$, and let $\mu_{[p,q]}$ be the subpath of $\mu$ that starts at $p$ and finishes at $q$. Finally, let $\mu '$ be the path obtained from $\mu$ by switching $\sfE$ and $\mu_{[p,q]}$. The covering relation is defined by $\mu <_\nu \mu'$. 
The poset \defn{$\Tam{\nu}$} is the transitive closure~$<_\nu$ of this relation. 

An example of the covering relation is depicted in Figure~\ref{fig:vFlip}. The lattice paths $\mu=\mathsf{ENENENENE}$ and $\mu'=\mathsf{ENNENEENE}$ are both $\nu$-paths, where $\nu=\mathsf{EENNEENNE}$. Moreover, $\mu$ is covered by $\mu_2$ under $<_\nu$. Indeed, the fourth point of $\mu$, denoted $p$, is a valley. It is at distance $0$ from $\nu$, and the next point at this distance is $q$. Switching the east step preceding $p$ with the interval $\mu_{[p,q]}$ gives rise to $\mu'$. 
\begin{figure}[htpb]
\centering 
 \includegraphics[width=.8\linewidth]{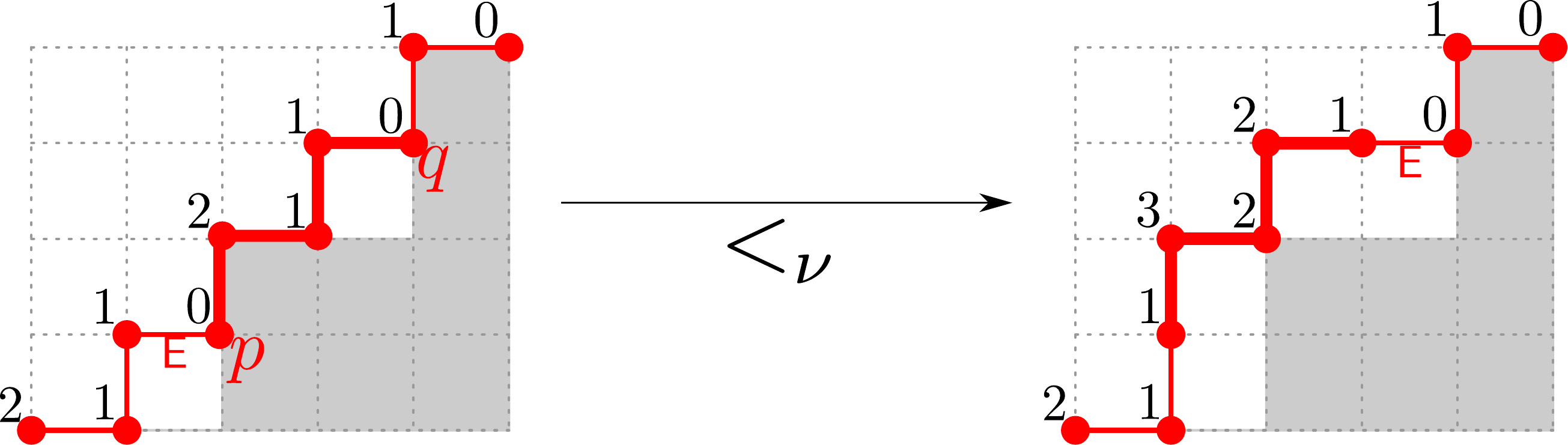}
 \caption{Two $\nu$-paths related by a covering for $<_\nu$, where the region below $\nu$ is colored gray. The distance $\horiz_\nu$ is marked on top of each point.}\label{fig:vFlip}
\end{figure}

Pr\'eville-Ratelle and Viennot showed that this poset is a lattice~\cite[Theorem~1.1]{PrevilleRatelleViennot2017} by showing that it is isomorphic to a certain interval in the classical Tamari lattice~\cite[Theorem~1.3 and Proposition~5.2]{PrevilleRatelleViennot2017} (cf. Proposition~\ref{prop:interval}). Consequently, $\Tam{\nu}$ is called the \defn{$\nu$-Tamari lattice}. Two examples are illustrated in Figure~\ref{fig:tamarispaths} (note the similarity with Figure~\ref{fig:IJTamaris}).

\subsection{Bijection between \texorpdfstring{$(I,\ol J)$}{IJ}-trees and \texorpdfstring{$\nu$}{v}-Dyck paths}
\label{subsec:TreesToPaths}

The $\nu$-Tamari lattice is intimately related to $(I,\ol J)$-trees, as we will now show. To this end, 
we will associate a lattice path $\nu(I,\ol J)$ to each pair $(I,\ol J)$ of nonempty subsets of $\NN$ with $\min (I\sqcup \ol J)\in I$ and $\max (I\sqcup \ol J)\in \ol J$. 
The lattice path $\nu(I,\ol J)$ is from $(0,0)$ to $(|I|-1,|\ol J|-1)$. Its $k$th step of $\nu(I,\ol J)$ is east if the $(k+1)$st element of $I\sqcup \ol J$ (ordered according to $\preceq$) belongs to $I$, and north otherwise. 

If $I$ and $\ol J$ form a partition of $[n]$ then this becomes slightly simpler: The $k$th step of $\nu(I,\ol J)$ is east if
$k\in I$ and north if $\ol k\in \ol J$ for $1\leq k\leq n-1$. For example, 
$\nu(\{0,1,2,5,6,9\},\{\ol 3,\ol 4, \ol 7,\ol 8, \ol {10}\})=\sfE\sfE\sfN\sfN\sfE\sfE\sfN\sfN\sfE$.

\begin{figure}[htpb]
\centering 
 \includegraphics[width=\linewidth]{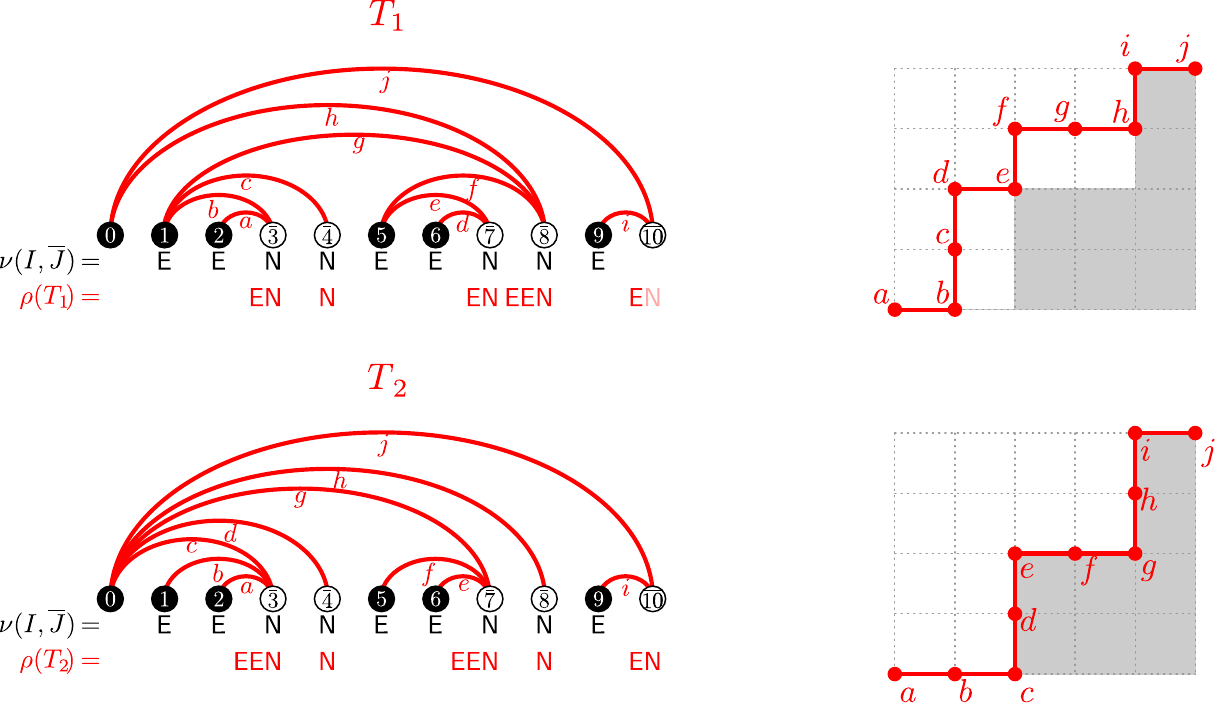}
 \caption{Two $(I,\ol J)$-trees for $I=\{0,1,2,5,6,9\}$ and $\ol J=\{\ol 3,\ol 4, \ol 7,\ol 8, \ol {10}\}$, and the corresponding $\nu(I,\ol J)$-paths. The arc-point correspondence is marked with letters.}\label{fig:IJtree2path}
\end{figure}

We also associate to each $(I,\ol J)$-tree $T$ a lattice path $\pa(T)$ from $(0,0)$ to $(|I|-1,|\ol J|-1)$ as follows: For each $\ol j\in \ol J$, do $d_T({\ol j})-1$ east steps and one north step, where
$d_T({\ol j})$ is the in-degree of $\ol j$ in $T$. Then remove the last north step. That is, if $\ol J=\{\ol j_1, \dots, \ol j_k\}$, then 
\[\pa(T)=\underbrace{\sfE\cdots\sfE}_{d_T({\ol j_1})-1} \sfN\underbrace{\sfE\cdots\sfE}_{d_T({\ol j_2})-1}\sfN\cdots \underbrace{\sfE\cdots\sfE}_{d_T({\ol j_k})-1}.\]
An alternative description is to say that $\pa(T)$ is the unique path such that the number of lattice points at height $k$ coincides with the in-degree of the $k$th element of $\ol J$. This induces a natural correspondence between arcs of $T$ and lattice points of~$\pa(T)$. For each $\ol j\in \ol J$, order its incident arcs with decreasing order of $I$-neighbors (from right to left in our pictures). Then associate the $\ell$th lattice point at height $k$ of $\pa(T)$ with the $\ell$th arc adjacent to $\ol j_k$, the $k$th element of $\ol J$. This can be seen in Figure~\ref{fig:IJtree2path}.

The tree $T_2$ of Figure~\ref{fig:IJtree2path} has the property that $\pa(T_2)=\nu(I,\ol J)$. This generalizes to arbitrary pairs $(I,\ol J)$: 
\begin{lemma}
 \label{lem:Tmin}
Let \defn{$T_{\min}$}
 be the $(I,\ol J)$-tree that contains all the arcs $(i,\min \{\ol j: i\prec \ol j\in \ol J\})$ together with the arcs of the form $(\min I,\ol j)$. Then $\pa(T_{\min})=\nu(I,\ol J)$.
\end{lemma}
\begin{proposition}\label{prop:IJtree2vpath}
Let $I\subset \NN$, $\ol J\subset \ol \NN$ be nonempty finite subsets with $\min (I\sqcup \ol J)\in I$ and $\max (I\sqcup \ol J)\in \ol J$, and let $\nu=\nu(I,\ol J)$. 
Then $\pa$ is a bijection from the set of $(I,\ol J)$-trees to the set of $\nu$-paths.
 Moreover, for each path $\nu$ from $(0,0)$ to $(a,b)$ there is a partition $I\sqcup J$ of $[a+b+1]$ such that $\nu(I,\ol J)= \nu$.
\end{proposition}
\begin{proof}
The second statement is straightforward, take $I=\{0\}\cup \{i\ :\ i\text{th step of }\nu=\sfE\}$ and $\ol J=\{\ol{a+b+1}\}\cup \{\ol j\ :\  j\text{th step of }\nu=\sfN\}$.
For the first statement we have already seen that there is an $(I,\ol J)$-tree $T_{\min}$ such that $\pa(T_{\min})=\nu$. To see that every lattice path associated to an $(I,\ol J)$-tree lies above $\nu$, observe that it suffices to show that, for each $k<|\ol J|$, $\nu$ has the maximal possible number of east steps before the $k$th north step. This is the number of arcs $(i,\ol j)\in T$ with $\ol j\prec \ol j_k$, where $\ol j_k$ is the $k$th element of $\ol J$. Since the subgraph induced by these arcs is acyclic, the maximal number of arcs is attained when it is connected, which happens for~$T_{\min}$.

We conclude that $\pa$ is a bijection by noting that the number of $(I,\ol J)$-trees coincides with the number of $\nu$-paths. One way to see this is to note that the restriction of the staircase triangulation of $\productp{I}{\ol J}$ (see~\cite[Thm.~6.2.13]{DeLoeraRambauSantos2010}) to $\catblock{I,\ol J}$ is a triangulation of $\catblock{I,\ol J}$ whose maximal cells are indexed by $\nu$-paths. 
Since $\productp{I}{\ol J}$ is a unimodular polytope (all simplices spanned by their vertices have the same volume, see~\cite[Prop.~6.2.11]{DeLoeraRambauSantos2010}), this triangulation has the same number of maximal cells as $\asstri{I,\ol J}$, whose maximal cells are indexed by $(I,\ol J)$-trees.
\end{proof}

The inverse of the map $\pa$ can be easily described as follows. Let $\mu$ be a $\nu$-path with $\nu=\nu\IJ$, and regard the ordered sequence of nodes from $I$ and $\ol J$ as a graph without arcs. We add arcs connecting the nodes in $\ol J$ in increasing order: For $k$ varying from 1 to $|\ol J|$, we include as many arcs to the $k$th node $\ol j\in\ol J$ as lattice points at height $k-1$ in $\mu$, such that the endpoints $i\in I$ ($i\prec \ol j$) are as right-most as possible and no crossings are formed. The resulting tree is the inverse $\pa^{-1}(\mu)$.


\subsection{\texorpdfstring{$\IJ$}{IJ}-Tamari lattices and \texorpdfstring{$\nu$}{v}-Tamari lattices are equivalent}
We are now ready to prove that $\IJ$-Tamari lattices and $\nu$-Tamari lattices are equivalent.

\begin{theorem}\label{prop:IJTamarivTamari}
An $(I,\ol J)$-tree $T'$ is obtained from $T$ via an increasing flip if and only if $\pa(T')$ covers $\pa(T)$ in the partial order $<_{\nu(I,\ol J)}$. 
Consequently, the partial orders $\Tam{I,\ol J}$ and $\Tam{\nu(I,\ol J)}$ are isomorphic.
\end{theorem}
\begin{proof}
 Let $T$ be an $(I,\ol J)$-tree and $\nu=\nu(I,\ol J)$. 
 We need to see how the concepts used in the definition of the covering relations get translated under the correspondence between arcs of $T$ and lattice points of the $\nu$-Dyck path $\pa(T)$.
 
 First, we say that a node $i'$ \defn{jumps over} the arc $\ij$ if $i'<i$ and it has a neighbor $\ol j'\geq \ol j$.
 It is not hard to see that if $p\in \pa(T)$ corresponds to $(i,\ol j)$ then $\horiz_\nu(p)$ coincides with the number of 
 nodes jumping over $(i,\ol j)$ (they are to the east steps present in $T_{\min}$ and not in $T$).
 
 Second, the point associated to $(i,\ol j)$ is a valley if and only if $\ol j\neq \max \ol J$, there is some $i'>i$ incident to~$\ol j$ ($p$ is preceded by $\sfE$) and there is no $i'<i$ incident to~$\ol j$ ($p$ is not followed by $\sfE$). Notice that this is equivalent to the arc being flippable and supporting an increasing flip.

 So let $(i,\ol j)$ be such an arc, associated to the valley~$p$. Consider the smallest $\ol j'>\ol j$ such that $(i,\ol j')\in T$. This element must exist: Since the graph is connected, the subgraph below the arc $(i,\ol j)$ can only be connected to the rest of the graph through such a $\ol j'$.

 We claim that $(i,\ol j')$ corresponds to $q$, the first lattice point in $\pa(T)$ after $p$ with $\horiz_\nu(q)=\horiz_\nu(p)$. 
 Observe first that every node jumping over $\ij$ also jumps over any arc $(i'',\ol j'')$ with $\ol j< \ol j''\leq \ol j'$, because $i''\geq i$ or otherwise it would cross $(i,\ol j')$. 
 If moreover  $\ol j''<\ol j'$, then it is jumped over by at least one extra node, namely $i$. Finally, every node jumping over $(i,\ol j')$ also jumps over $(i,\ol j)$ (see the schematic flip in Figure~\ref{fig:IJflipschema}). 
  
 Flipping $(i,\ol j)$ switches an arc adjacent to $\ol j$ for an arc adjacent to $\ol j'$. That is, it removes a
 horizontal step at $p$'s height and adds it at $q$'s height, which is precisely the transformation defining the covering relation $<_\nu$.
\end{proof}

\begin{figure}[htpb]
\centering 
\includegraphics[width=0.8\textwidth]{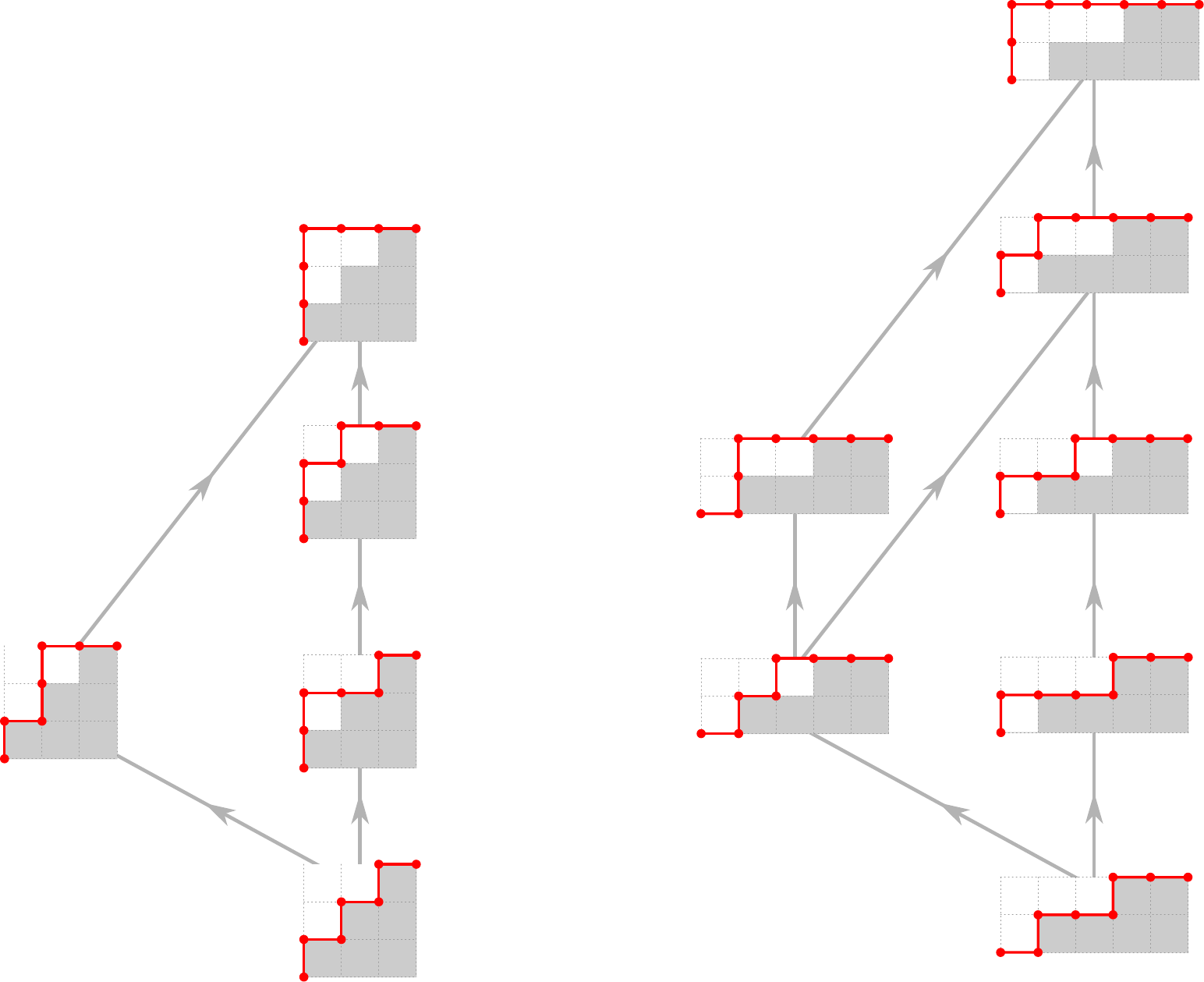}
 \caption{Representation of the $\IJ$-Tamari lattices of Figure~\ref{fig:IJTamaris} in terms of $\nu\IJ$-paths.}
 \label{fig:tamarispaths}
\end{figure}

\subsection{The canopy and combinatorial duality}\label{sec:canopy}

Pr\'eville-Ratelle and Viennot introduced the partial order $\Tam{\nu}$ (that is, $\Tam{I,\ol J}$) in~\cite{PrevilleRatelleViennot2017}, where they showed that it is an interval of the Tamari lattice. 
Their main tool is the concept of canopy, which has an easy interpretation from the point of view of $(I,\ol J)$-trees. Define $I\subseteq [n],\ol J\subseteq [\ol n]$ to be an \defn{$[n]$-canopy} if they partition $[n]$ ($I\cap J=\varnothing$ and $I\sqcup J=[n]$), $0\in I$ and $\ol n\in \ol J$. 
We will usually call it a \defn{canopy} when $[n]$ is clear from the context. 

We associate an $[n]$-canopy \defn{$\can(T)$} to each $\nn$-tree $T$ with $n>0$: The node set of the graph obtained after removing the leaves of $T$. To see that this is indeed a canopy, observe that for each $i\in [n]$, $(i,\ol i)$ must be an arc of $T$ since it cannot cross any other arc. Moreover, one of the two must be a leaf since otherwise it would induce a crossing and, unless $n=0$, it cannot be both because the graph is connected. It is straightforward to observe that, after removing $0$ and $\ol n$, this provides the canopy of the binary tree associated to $T$ (cf. Remark~\ref{rmk:triang2dualtree}), as defined in~\cite[Section~2]{PrevilleRatelleViennot2017} (originally from \cite{LodayRonco1998}). 

\begin{proposition}[{{\cite[Proposition~4.3 and Theorem~1.1]{PrevilleRatelleViennot2017}}}]\label{prop:interval}
 The set of $\nn$-trees with $[n]$-canopy $(I,\ol J)$ is an interval of the ordinary Tamari lattice isomorphic to $\Tam{I,\ol J}$. In particular, $\Tam{I,\ol J}$ is a lattice for every $I,\ol J$ and $\Tam{\nu}$ is a lattice for every lattice path $\nu$.
\end{proposition}

Hence, the $\nn$-Tamari lattice can be partitioned into the intervals corresponding to the canopies.
The main ingredient for its proof is \cite[Corollary~4.2]{PrevilleRatelleViennot2017}, which states that $\can (T)\prec \can (T')$ whenever $T<_{[n],[\ol n]} T'$. Here, $\can (T)\prec \can (T')$ refers to the partial order on the set of $[n]$-canopies induced by the relations $(I,\ol J)\prec (I',\ol J')$ when there is some $i\in[n]$ such that $I'=I\setminus i$ and $\ol J'= \ol J\cup \ol i$. The monotonicity of the canopy with respect to the Tamari order can be seen from our description of flips: if there is an increasing flip that changes the $k$th element of the canopy, then it must replace an arc $(i,\ol k)$ with an arc $(k,\ol j)$ (cf. Figure~\ref{fig:IJflipschema}). In the canopy, $k$ is replaced by $\ol k$, which is an increasing change.

\begin{remark}\label{rem:Reading_latticecongruences}
It is worth mentioning that the fact that $\Tam{\nu}$ is a lattice follows from earlier work of Reading on lattice congruences of the weak order.\footnote{We thank an anonymous referee for pointing out this important observation.} 
Indeed, the decomposition of the ordinary Tamari lattice into intervals of fixed canopy~\cite[Theorem~1.3 and Proposition~4.3]{PrevilleRatelleViennot2017} coincides with an interval decomposition already considered by Reading in a slightly different context~\cite{reading_lattice_2005}. Reading's decomposition is determined by a lattice congruence which gives rise to the Boolean lattice as a lattice quotient of the Tamari lattice. Geometrically, this can be explained as the inclusion map from the maximal normal cones of Loday's associahedron to the maximal normal cones of a combinatorial cube. The vertices of this cube represent the possible canopies~$\nu$; the preimage (or orbit) of $\nu$ is the congruence class of elements in the Tamari lattice with fixed canopy $\nu$. In particular, each $\nu$-congruence class is an interval in the Tamari lattice isomorphic to $\Tam{\nu}$, and therefore has the structure of a lattice. See Figure~\ref{fig:Reading_vTamaricongruence} for an example. We refer to the nice survey article~\cite{reading_tamari_2012} and to~\cite{reading_lattice_2005} for more detailed information about these topics.

\begin{figure}[htbp]
\begin{center}
\includegraphics[width=\textwidth]{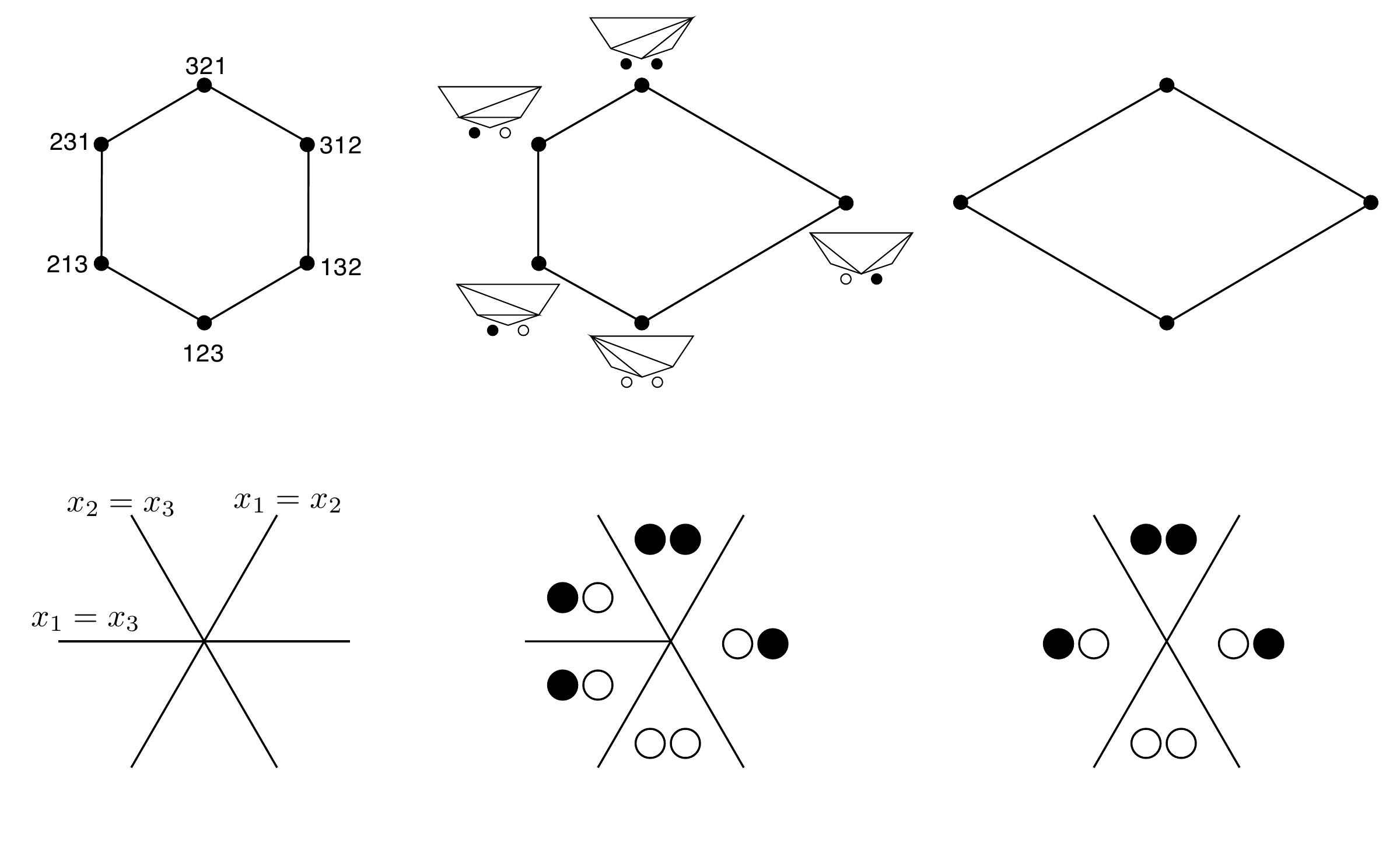}
\caption{Reading's lattice congruence description of the interval decomposition of the Tamari lattice into canopy classes. The two dimensional polytopes and their normal fans are: the permutahedron (left), the Loday associahedron (middle), and a combinatorial cube (right). The black and white balls encode the canopy. The path $\nu$ is obtained by replacing each $\bullet$ by an east step and each $\circ$ by a north step.}
\label{fig:Reading_vTamaricongruence}
\end{center}
\end{figure}

\end{remark}

Another result of Pr\'eville-Ratelle and Viennot is that $\Tam{\nu}$ is isomorphic to the dual of $\Tam{\reverse\nu}$~\cite[Theorem~1.2]{PrevilleRatelleViennot2017}, where $\reverse \nu$ is the path obtained by reading $\nu$ backwards and replacing the east steps by north steps and vice versa. This amounts to ``transposing'' the path.

This has a nice geometric interpretation in our setup as the exchange of the two factors of the product $\productp{I}{\ol J}$ in the triangulation $\asstri{I,\ol J}$. 

In terms of $\IJ$-trees, denote by $\reverse\IJ$ the pair $(n-J, \ol{n-I})$, where $n-J:=\{n-j\ \colon\ j\in J\}$, $n-I:=\{n- i\ \colon\ i\in  I\}$, and $n=\max(I\sqcup J)$. 
Of course, $\reverse\IJ$-trees are just mirror images of $\IJ$-trees, while clearly the covering relation gets reversed.

\begin{proposition}[{\cite[Theorem~1.2]{PrevilleRatelleViennot2017}}]\label{prop:reverse}
As a simplicial complex, $\asstri{I,\ol J}$ is isomorphic to $\asstri{\reverse{I,\ol J}}$, and as a lattice $\Tam{I,\ol J}$ is isomorphic to the dual of $\Tam{\reverse{I,\ol J}}$. 
\end{proposition}


\section{The \texorpdfstring{$\IJ$}{IJ}-Tamari complex and \texorpdfstring{$\IJ$}{IJ}-Narayana numbers}

\label{sec:IJcomplex}

We define the \defn{$\IJ$-Tamari complex $\asscomp{I,\ol J}$} as the underlying simplicial complex of the $\IJ$-associahedral triangulation\footnote{There are two objects that are referred to as `associahedron' in the literature: the simple polytope whose graph is the Tamari lattice and the simplicial complex dual to its boundary. To avoid confusion, we reserve the term `$\IJ$-associahedron' for our analogue of the simple polytope, and use `$\IJ$-Tamari complex' for the simplicial complex.}. That is, the simplicial complex of $\IJ$-forests (see Definition~\ref{def:IJtree}). In this section, we will see that it shares several interesting properties with the classical simplicial associahedron. In particular, in view of the bijection with $\nu$-Dyck paths presented in Section~\ref{subsec:TreesToPaths}, this provides alternative definitions for simplicial associahedra in the setup of Fuss-Catalan and rational Dyck paths (see~\cite{ARW13} for a different approach to define them). 
This is complemented by the computation of the $h$-vector of $\asscomp{I,\ol J}$, 
whose entries are a natural generalization of the Narayana numbers: $\nu$-Dyck paths with $k$ valleys.

\begin{definition}\label{def:TamariComplex}
The \defn{$\IJ$-Tamari complex $\asscomp{I,\ol J}$} is the flag simplicial complex on $\{(i,\ol j)\in I\times \ol J\ \colon\ i\prec \ol j\}$ whose minimal non-faces are the pairs
 $\{(i,\ol j),(i', \ol j')\}$ with $i\prec  i'\prec \ol j \prec \ol j '$.
\end{definition}

These complexes generalize the boundary complexes of simplicial associahedra, the polar of (simple) associahedra. Indeed, $\asscomp{[n],[\ol n]}$ is an $(n+2)$-fold pyramid over the simplicial complex of non-crossing inner diagonals of an $(n+2)$-gon. These cone points account for the arc $(0,\ol n)$ and the $n+1$ arcs of the form $(i,\ol i)$ which are present in every $\nn$-tree. (A \defn{cone point} of a simplicial complex is an element of the ground set that belongs to all maximal faces.)

Recall that, for simplicial complexes $\cK$ and $\cK'$ with disjoint support their 
\defn{join} is the simplicial complex 
\begin{equation*}
\cK\ast \cK':=\{F\cup F'\colon F\in \cK, F'\in \cK'\}.
\end{equation*}

\begin{proposition}
\label{prop:bijectionpolygon}
$\asscomp{[n],[\ol n]}$ is isomorphic to the join of an $(n+1)$-dimensional simplex and the boundary complex of a simplicial $(n-1)$-associahedron.
\end{proposition}


\subsection{Faces, facets, interior faces}

The faces of $\asscomp{I,\ol J}$ are given by $\IJ$-forests. Our first goal is to study which of these faces are interior.
The boundary of $\asscomp{I,\ol J}$ is the simplicial complex induced by the codimension-$1$ faces that are contained in exactly one facet.
The faces of $\asscomp{I,\ol J}$ not contained in the boundary are called \defn{interior}, or \defn{interior simplices}.

\begin{lemma}
\label{lem:interiorsimplex}
The interior simplices of $\asscomp{I,\ol J}$ are the \defn{$\IJ$-forests} that include the arc $(\min I,\max \ol J)$ and have no isolated nodes. We refer to such {$\IJ$-forests} as \defn{covering $\IJ$-forests}.
\end{lemma}

\begin{proof}
A codimension-$1$ boundary simplex of $\asscomp{I,\ol J}$ is obtained from an $\IJ$-tree by removing a non-flippable arc; that is, a leaf or the arc $(\min I,\max \ol J)$. In the former case, the boundary simplex represents an $\IJ$-forest with an isolated node, whereas in the latter it represents an $\IJ$-forest missing the arc $(\min I,\max \ol J)$. Therefore, a simplex of $\asscomp{I,\ol J}$ is interior if and only if it corresponds to a covering $\IJ$-forest.
\end{proof}

Our next goal is to study the links of the faces of $\asscomp{I,\ol J}$. Recall that the \defn{link} of a face $F$ in the simplicial complex $\cK$ is the simplicial complex
\begin{equation*}
\link_{\cK}(F):=\left\{ G\setminus F\ \colon \ \ F\subseteq G\in \cK \right\}.
\end{equation*}

Links of faces of simplicial associahedra are joins of simplicial associahedra. A similar behavior happens for $\IJ$-Tamari complexes. Denote by $\cK\setminus v$ the simplicial complex obtained
by deleting a ground set element $v$ from the simplicial complex~$\cK$. 
We will only use this operation when $v$ is a cone point of $\cK$, and in this case $\cK\setminus v=\link_{\cK}(v)$.

\begin{lemma}
\label{lem:links}
Up to cone points, the link of any $\IJ$-forest $F$ in $\asscomp{I,\ol J}$ is a join of Tamari complexes (see Figure~\ref{fig:IJfaces}). Precisely:
\begin{equation}\label{eq:link}
\link_{\asscomp{I,\ol J}}(F)\ast F\cong \left(\bigast_{\substack{(i,\ol j)\in F\\ (i,\ol j)\neq (i_0,\ol j_0)}} \asscomp{I_{(i,\ol j)},\ol{J_{(i,\ol j)}}}\setminus (i,\ol j)\right)\ast\asscomp{I_{(i_0,\ol j_0)},\ol{J_{(i_0,\ol j_0)}}}
\end{equation}
where $(i_0,\ol j_0)=(\min I,\max \ol J)$, $I_{(i,\ol j)}$ and $\ol{J_{(i,\ol j)}}$ are the restrictions of $I,\ol J$ to
\begin{align*}
[i,\ol j]\setminus\left(\bigcup_{\substack{(i',\ol j')\in F\\ [i',\ol j']\subsetneq [i,\ol j]}}]i',\ol j'[\right),%
\end{align*}
and $[i,\ol j]$ and $]i,\ol j[$ represent the corresponding closed and open intervals of $\NN\sqcup \ol \NN$, respectively.

In particular, the link of an arc $\ij$ in $\asscomp{I,\ol J}$ is is a join of Tamari complexes: 
\begin{equation}\label{eq:vertexlink}
\link_{\asstri{I,\ol J}}\ij \cong \left(\asscomp{I',\ol {J'}}\setminus \ij\right) \ast \left(\asscomp{I'',\ol {J''}}\setminus \ij\right),
\end{equation}
where $I'=I\cap [i,\ol j]$ and $\ol {J'}= \ol J \cap [i,\ol j]$, $I''=I\setminus\, ]i,\ol j[$, $\ol {J''}=\ol J \setminus\, ]i,\ol j[$.

If moreover $F$ is a covering $\IJ$-forest, then $\link_{\asscomp{I,\ol J}}(F)$ is a join of boundary complexes of simplicial associahedra.
\end{lemma}

\begin{figure}[htpb]
 \centering
 \includegraphics[width=\linewidth]{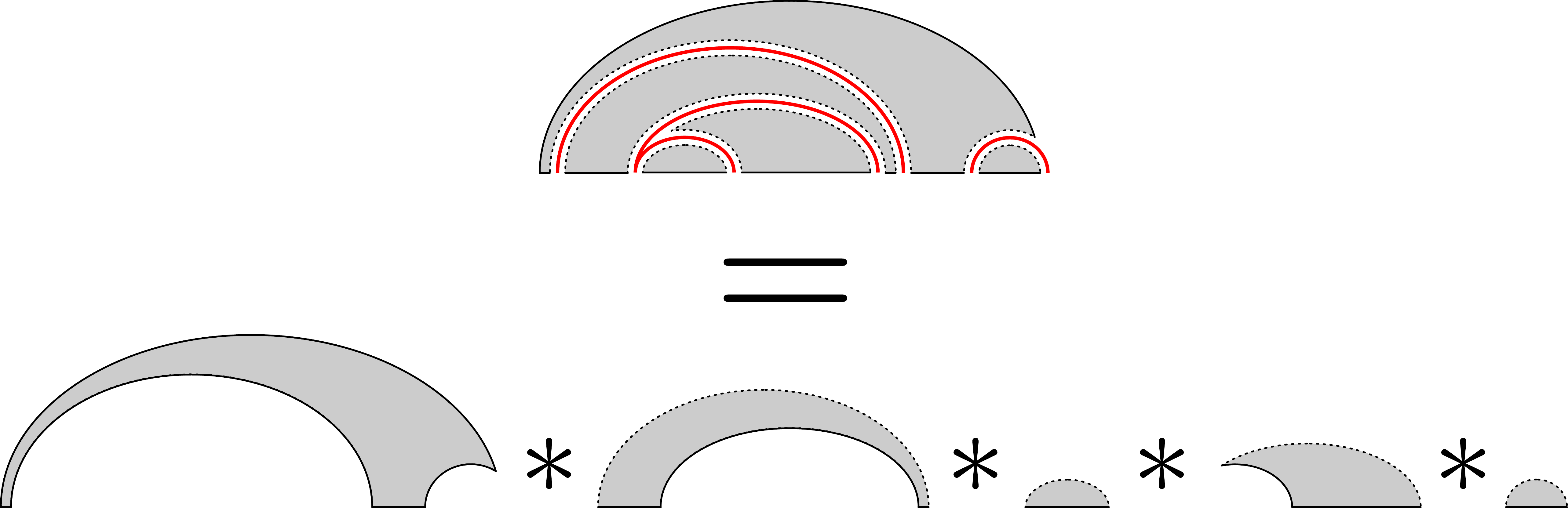}
 \caption{The link of an $\IJ$-forest is a join of Tamari complexes. Dashed arcs represent removed cone points.} 
 \label{fig:IJfaces}
\end{figure}

 \begin{proof}
 Equation~\eqref{eq:link} is a straightforward consequence of Equation~\eqref{eq:vertexlink}, by induction. The underlying idea is schematically depicted in Figure~\ref{fig:cyclicIJfaces}.
 
 To prove~\eqref{eq:vertexlink}, observe that if two different arcs $\ij$ and $(i',\ol j')$ are non-crossing, then either $[i',\ol j']\subsetneq [i,\ol j]$, $[i,\ol j]\subsetneq [i',\ol j']$ or $[i',\ol j']\cap[i,\ol j]=\emptyset$. Hence, any $\IJ$-forest containing $\ij$ can be written up to cone points as the join of an $(I',\ol{J'})$-forest and an $(I'',\ol{J''})$-forest. Conversely, up to cone points, the join of an $(I',\ol{J'})$-forest and an $(I'',\ol{J''})$-forest gives a $\IJ$-forest containing~$\ij$.

 For the last statement, assume from now on that $F$ is a covering $\IJ$-forest, and let $\ij\in F$. To understand the contribution of $\asscomp{I_{(i,\ol j)},\ol{J_{(i,\ol j)}}}$ to the join in~\eqref{eq:link}, it suffices to characterize the restricted subsets $I_{\ij},\ol{J_{\ij}}$.  Concretely, the statement will follow once we realize that $I_{\ij}\sqcup\ol{J_{\ij}}$ alternates between elements of $I$ and $\ol J$, up to removing removing $\min I_{\ij}, \max \ol{J_{\ij}}$ or both, if necessary.
 
 Denote by $F_{\ij}$ the restriction of $F$ to $I_{\ij}\sqcup\ol{J_{\ij}}$. By construction, $F_{\ij}$ is a covering forest of $\asscomp{I_{(i,\ol j)},\ol{J_{(i,\ol j)}}}$. Let $i_1\prec i_2\in I_{\ij}$ be two consecutive elements of $I_{\ij}\sqcup\ol{J_{\ij}}$, and consider the arcs $(i_1,\ol j_1), (i_2,\ol j_2)\in F_{\ij}$  with $\ol j_1$ and $\ol j_2$ maximal in $\ol{J_{\ij}}$. Since $F_{\ij}$ is non-crossing, we have $[i_2,\ol j_2]\subsetneq [i_1,\ol j_1]$. Further, $i_1$ must equal $\min I_{\ij}$, because otherwise $[i_2,\ol j_2]\subsetneq [i_1,\ol j_1]\subsetneq[i,\ol j]$ and $i_2\in ]i_1,\ol j_1[$, so $i_2\notin I_{\ij}$. By the same token, we necessarily have $\ol j_1=\max \ol{J_{\ij}}$. Therefore, $I_{\ij}\sqcup\ol{J_{\ij}}$ has at most two consecutive elements from $I_{\ij}$, which are then the least elements of $I_{\ij}\sqcup\ol{J_{\ij}}$. An analogous argument shows that $I_{\ij}\sqcup\ol{J_{\ij}}$ has at most two consecutive elements from $\ol{J_{\ij}}$, which are the maximum elements of $I_{\ij}\sqcup\ol{J_{\ij}}$.
 
 Setting $n=\max\{|I_{\ij}|,|\ol{J_{\ij}}|\}-1$, it follows that $\asscomp{I_{(i,\ol j)},\ol{J_{(i,\ol j)}}}\cong \asscomp{[n],[\ol n]}$ (up to cone points if $I_{\ij}\sqcup\ol{J_{\ij}}$ has two consecutive elements from $I_{\ij}$ or from $\ol{J_{\ij}})$, which is up to cone points a simplicial associahedron by Proposition~\ref{prop:bijectionpolygon}. Removing all these cone points from the join in~\eqref{eq:link}, we deduce that $\link_{\asscomp{I,\ol J}}(F)$ is a join of boundary complexes of simplicial associahedra.
 \end{proof}


\subsection{Shellings, the \texorpdfstring{$h$}{h}-vector and the \texorpdfstring{$\IJ$}{IJ}-Narayana numbers}\label{sec:Narayana}

We will now compute the $h$-vectors of $\asstri{I,\ol J}$ using shellings, and relate them to certain lattice path enumeration.

Recall that a \defn{shelling} of a simplicial complex is an ordering $\cO=(F_1,F_2,\ldots, F_r)$ of its facets such that for every $\ell<m$ there is some $k<m$ such that $F_\ell\cap F_m\subseteq F_k\cap F_m = F_m\setminus e$ for some $e\in F_m$~\cite{Ziegler95}. That is, the intersection of $F_m$ with the subcomplex generated by the facets preceding it is pure and of codimension $1$. 

Simplicial complexes admitting a shelling are called \defn{shellable}.  These have nonnegative $h$-vectors whose entries can be computed combinatorially from the shelling as follows. For a fixed shelling $\cO$, define the restriction set of facet $F_m$ as $R(F_m):=\{e\in F_m\colon F_m\setminus e \subseteq F_\ell \text{ for some $\ell<m$}\}$; then:
\begin{equation}
\label{eq:hshelling}
h_k=\left|\left\{m\colon |R(F_m)|=k\right\}\right|.
\end{equation}

Our first observation states that the order relation of the $\IJ$-Tamari lattice is amenable to shellings for the $\IJ$-associahedral complex.

\begin{lemma}\label{lem:shellingorder}
Let $\cO=(T_1,T_2,\ldots, T_r)$ be an ordering of the $\IJ$-trees that is a linear extension of the $\IJ$-Tamari lattice or of its opposite lattice. Then $\cO$ is a shelling order for $\asscomp{I,\ol J}$.
\end{lemma}
\begin{proof}
We consider only the case when $\cO$ extends the $\IJ$-Tamari lattice, the other case being equivalent by Proposition~\ref{prop:reverse}. Let $T_i,T_j$ be $\IJ$-trees with $i<j$ with respect to the ordering $\cO$. 
It is not hard to see that there is some $\IJ$-tree $T\leq T_i\wedge T_j$ such that $T\supset T_i\cap T_j$. Indeed, we can look at each arc of $T_i\cap T_j$ as dividing the support $\IJ$ in two pieces (as in the proof of Lemma~\ref{lem:links}); and then it suffices to take the minimum in each of the pieces. 
Now, consider a sequence $\bfs=(T_{s_0}=T, T_{s_1},\ldots, T_{s_w}=:T_j)$ of $\IJ$-trees, where each tree is obtained from the preceding one by an \emph{increasing flip} that does not change the arcs in $T_i\cap T_j$. Again such a sequence exists because we can do it separately in every piece. 
Finally, take the previous to last tree $T_k:=T_{s_{w-1}}$ in the sequence $\bfs$, so that in particular $k<j$ in the linear order $\cO$. We conclude that $T_i\cap T_j\subseteq T_k\cap T_j=T_j\setminus e$ for some arc $e\in T_j$, so $\cO$ actually gives a shelling order.  
\end{proof}

\begin{theorem}
\label{thm:thehvector}
The $h$-vector $(h_0,h_1,\ldots)$ of the $\IJ$-Tamari complex is determined by:
\begin{align*}
\label{eq:thehvector}
h_\ell&=\left|\left\{ T\colon \text{ $\IJ$-tree with exactly $\ell$ non-leaf nodes in $\ol J\setminus\{\ol j_{\max}\}$} \right\}\right| \\
&=\text{number of $\nu\IJ$-paths with exactly $\ell$ valleys},
\end{align*}
where a valley of a path is an occurrence of $\sfE\sfN$.
\end{theorem}

\begin{proof}
For the first equality, let $\cO=(T_1,T_2,\ldots, T_r)$ be a linear ordering of the $\IJ$-trees that extends the opposite of the $\IJ$-Tamari lattice, and $T_m\in \cO$. We have that $|R(T_m)|$ equals the number $\IJ$-trees covering $T_m$ in the $\IJ$-Tamari lattice, which is just the number of increasing flips that can be performed on $T_m$. As seen in the proof of Proposition~\ref{prop:IJTamarivTamari}, the latter is precisely the number of nodes in $\ol J\setminus\{\ol j_{\max}\}$ that are not leaves of $T_m$, which proves the equality. 

To prove the second equality, recall from Proposition~\ref{prop:IJTamarivTamari} that increasing flips of an $\IJ$-tree $T$ are in bijection with occurrences of $\sfE\sfN$ in its associated $\nu\IJ$-path $\pa(T)$, that is, with valleys of $\pa(T)$.
\end{proof}

Classical Narayana numbers are defined as the number of paths above $\{\sfE\sfN\}^n$ with $\ell$ valleys. Similarly, {rational Narayana numbers} count the number of lattice paths from $(0,0)$ to $(b,a)$ that stay above the line of slope $\frac ab$ and have $\ell$ valleys. In the same vein, we call the number of $\nu$-paths with exactly $\ell$ valleys the \defn{$\nu$-Narayana numbers} (or \defn{$\IJ$-Narayana numbers} if $\nu=\nu\IJ$).

\begin{question}
 The rational associahedron introduced in~\cite{ARW13} is a simplicial complex whose $h$-vector is also counted by $\nu$-Narayana numbers for certain $\nu$. Which is its relation with the corresponding $\asstri{I,\ol J}$?
\end{question}


\section{A tropical realization of the \texorpdfstring{\protect{$\IJ$}}{IJ}-associahedron}
\label{sec:nuassociahedron}

Recall that the facets of the $\IJ$-associahedral triangulation $\asstri{I,\ol J}$ are encoded by $\IJ$-trees and that two facets are adjacent if and only if the corresponding trees are connected by a flip. As a consequence, the Hasse diagram of the $\IJ$-Tamari lattice is realized as the dual graph of~$\asstri{I,\ol J}$. 
Via the Cayley trick~\cite{HRS00},~$\asstri{I,\ol J}$ can be interpreted as a fine mixed subdivision of a generalized permutahedron $\cP_{I,\ol J}\subset \RR^{|\ol J|-1}$, which is explicitly described in Corollary~\ref{cor:cayley}. 
Combining this with the fact that $\IJ$-Tamari lattice and $\nu$-Tamari lattices are equivalent (Theorem~\ref{prop:IJTamarivTamari}), we conclude the first two geometric realizations in our main geometric result Theorem~\ref{thm:realizations}. The third realization is obtained via the duality between regular triangulations of~$\product$ and tropical hyperplane arrangements, conceived by Develin and Sturmfels in~\cite{DevelinSturmfels2004} and further developed in~\cite{ArdilaDevelin2009, FinkRincon2015}.

In this section we examine this tropical realization in detail. In particular, we completely describe its polyhedral cells and present explicit coordinates for its vertices, as well as some remarkable properties.
Our presentation assumes no previous knowledge of tropical geometry and focuses on the realization in the Euclidean space.

Throughout the section, we assume that $I\subset\NN$, $\ol J\subset \ol\NN$ are nonempty finite subsets such that $\min (I\sqcup \ol J)\in I$ and ${\max (I\sqcup \ol J)\in \ol J}$; that is, that the lattice path
$\nu=\nu(I,\ol J)$ from Section~\ref{subsec:TreesToPaths} is well defined. 
We also consider a non-crossing height function $\ha$ from the vertex set of $\catblock{I,\ol J}$ to~$\RR$, that induces $\asstri{I,\ol J}$ as a regular triangulation (as in Proposition~\ref{prop:assheights}). That is, a function $\ha$ defined on the pairs $\ij$ with $i\prec \ol j$ 
such that, for every quadruple $i\prec i'\prec \ol j\prec\ol j'$, it assigns less weight to the non-crossing matching:
\begin{equation}\label{eq:asso_height}
\ha({i,\ol j'})+\ha({i',\ol j})< \ha({i,\ol j})+\ha({i',\ol j'}).  
\end{equation}
We extend $\ha$ to the whole $I\times \ol J$ by setting $\ha(i,\ol j)=\infty$ whenever $i\succ\ol j$.

In tropical geometry it is convenient to work in the \emph{tropical projective space} 
\[\TP^{|\ol J|-1}=\left((\RR\cup \infty)^{|\ol J|}\setminus (\infty,\infty,\dots,\infty)\right)/\RR(1,1,\dots,1).\]

A \emph{tropical hyperplane} is the ``tropical vanishing locus" of a linear equation in the tropical semiring $(\mathbb R\cup \infty, \oplus, \odot)$ where the tropical addition $\oplus$ and tropical multiplication~$\odot$ are defined by $a \oplus b = \min(a, b)$ and $a \odot b = a + b$. 
The main ingredient in our realization is the arrangement $\arrgtA=(H_i)_{i\in I}$ of \emph{inverted tropical hyperplanes} at each point $v_i=\left(\ha(i,\ol j) \right)_{\ol j\in \ol J}$ for $i\in I$. These hyperplanes can be explicitly defined as  
\begin{equation}\label{eq:tropical_hyp_vertices}
H_i=\left\{x\in \TP^{{|\ol J|}-1}: \max_{\ol j\in \ol J}\{ -\ha(i,\ol j)+x_{\ol j} \} \ \text{is attained twice}\right\}.
\end{equation}

When some of the $\ha(i,\ol j)$ equal $\infty$, such $H_i$ corresponds to a ``degenerate'' tropical hyperplane~\cite{FinkRincon2015}.
This arrangement induces a polyhedral decomposition of $\TP^{|\ol J|-1}$ whose poset of bounded cells is anti-isomorphic to the poset of interior cells of the triangulation (see~\cite[Proof of Theorem~1]{DevelinSturmfels2004}).

Restricting to the points with finite coordinates, this polyhedral complex can be interpreted as a polyhedral decomposition of $\RR^{|\ol J|}$ in which every cell contains the vector $(1,1,\dots, 1)$ in its linearity space. Intersecting with the hyperplane $x_{\ol j_{\max}}=0$, where $\ol j_{\max}=\max \ol J$, we obtain a polyhedral decomposition of $\RR^{|\ol J|-1}$. We explicitly present this realization, which is the one giving rise to the pictures presented throughout the paper. (This dehomogenization could have been performed by intersecting with other hyperplanes, giving rise to different projectively equivalent realizations.)

\begin{definition}\label{def:nuassociahedronH}
The \defn{$\IJ$-associahedron $\Asso_{I,\ol J}(\ha)$} is the polyhedral complex of bounded cells induced by the arrangement of inverted tropical hyperplanes~$\arrgtA$.    
For a lattice path~$\nu$, we define the \defn{$\nu$-associahedron} \defn{$\Asso_\nu(\ha)$} as the $\IJ$-associahedron, where $I,\ol J$ are such that $\nu=\nu\IJ$, as in Subsection~\ref{subsec:TreesToPaths}.
To simplify notation, some times we omit~$\ha$ when it is clear from the context.
\end{definition}

By the discussion above, which follows from results in~\cite{DevelinSturmfels2004} and~\cite{FinkRincon2015}, we deduce the following result (which could be considered as a combinatorial definition of the $\IJ$-associahedron). 

\begin{theorem}
\label{thm:nurealization}
The $\IJ$-associahedron $\Asso_{I,\ol J}(\ha)$ is a polyhedral complex whose poset of cells is anti-isomorphic to the poset of interior faces of the $\IJ$-Tamari complex. 
In particular, 
\begin{enumerate}
\item Its faces correspond to \emph{covering} $\IJ$-forests.
\item Its vertices correspond to $\IJ$-trees.
\item Two vertices are connected if and only if the corresponding trees are connected by a flip. That is, the edge graph of $\Asso_{I,\ol J}(\ha)$ is isomorphic to the Hasse diagram of the $\IJ$-Tamari lattice.
\end{enumerate}
\end{theorem}

Some examples of $\IJ$-associahedra are depicted in Figure~\ref{fig:Us}.


\subsection{Defining inequalities and coordinates of the vertices}

We now give a more precise description of the polyhedral cells and vertices of the $\IJ$-associahedron.
Our description follows directly from~\cite{DevelinSturmfels2004} and~\cite{FinkRincon2015} and will be essential to present some geometric properties in the upcoming sections. The construction is illustrated in Example~\ref{ex:ExampleIJtrees}.

\begin{definition}\label{def:cells}
 To each arc $(i,\ol j)$ with $i\prec\ol j$ we associate the following polyhedron \defn{$\tilde \sfg{\ij}$} in~$\RR^{\ol J}$
 \[\tilde \sfg{\ij}=\left\{x\in\RR^{\ol J}\ \colon\ x_{\ol k}-x_{\ol j}\leq \ha(i,\ol k)-\ha(i,\ol j) \text{ for each }\ol k\in \ol J\right\}.\]
 For a subgraph $G$ of $K_{I,\ol J}$, define the polyhedron \defn{$\tilde\sfg(G)$} in~$\RR^{\ol J}$ as
\[\tilde\sfg(G):=\bigcap_{\ij\in G}\tilde \sfg{\ij}.\]
It is easy to see that it is not empty if and only if it is an $\IJ$-forest (cf. Lemma~\ref{lem:emptycrossing}), and then it always contains the vector $(1,\dots,1)$ in its linearity space. We dehomogenize by intersecting with the hyperplane $x_{\ol j_{\max}}=0$, to obtain the 
polyhedral cell \defn{$\sfg(G)$} in $\RR^{\ol J\setminus \{\ol j_{\max}\}}$:
\[\sfg(G):=\tilde\sfg(G)\cap\{x_{\ol j_{\max}}=0\}=\bigcap_{\ij\in G} \sfg{\ij},\]
where $\sfg\ij:=\tilde\sfg\ij\cap\{x_{\ol j_{\max}}=0\}$.
\end{definition}

\begin{lemma}\label{lem:emptycrossing}
For a subgraph $G$ of $K_{I,\ol J}$, the polyhedron \defn{$\tilde\sfg(G)$} is non-empty if and only if $G$ is an $(I,\ol J)$-forest (i.e. it is increasing and non-crossing, cf.~Definition~\ref{def:IJtree}).
\end{lemma}
\begin{proof}
 If $G$ is not increasing, then there is $(i,\ol j)\in G$ with $\ol j\prec i$. In this case $\ha(i,\ol j)=\infty$ and the inequality $x_{\ol k}-x_{\ol j}\leq \ha(i,\ol k)-\ha(i,\ol j)$ has no feasible point.
 
 If $G$ has a crossing, that is $(i,\ol j), (i',\ol j')\in G$ for $i\prec i'\prec \ol j\prec\ol j'$, then by \eqref{eq:asso_height} the definition of $\sfg(G)$ includes the following pair of incompatible inequalities
 \begin{align*}
 x_{\ol j'}-x_{\ol j}&\leq \ha(i,\ol j')-\ha(i,\ol j)\stackrel{\eqref{eq:asso_height}}{<}\ha(i',\ol j')-\ha(i',\ol j),\\  
 x_{\ol j}-x_{\ol j'}&\leq \ha(i',\ol j)-\ha(i',\ol j').  
\end{align*}

It remains to see that if it is an $\IJ$-forest then it is not empty. Notice that it suffices to see it for $\IJ$-trees, since every $\IJ$-forest can be completed into a tree. In equation~\eqref{eq:g_tree} of the upcoming Lemma~\ref{lem:coordinates} we present an explicit point that lies in $\sfg(T)$ for each $\IJ$-tree $T$.
\end{proof}

\begin{lemma}\label{lem:verticesgF}
For each covering $(I,\ol J)$-forest $F$, $\sfg(F)$ is a (bounded) convex polytope of dimension one less than the number of connected components of $F$ and whose 
vertices correspond to the $\IJ$-trees containing $F$:
\begin{equation*}
\sfg(F)=\conv\left\{ \sfg(T)\colon \text{ $T$ is $(I,\ol J)$-tree containing $F$} \right\}.
\end{equation*}
\end{lemma}

\begin{proof}
 This follows essentially from~\cite[Prop.~4.1]{FinkRincon2015} (and \cite[Sec.~3]{DevelinSturmfels2004}), but we include the proof because the translation might not be straightforward. 
 
 If $F$ is covering, then by definition $(i_{\min},\ol j_{\max})\in F$, and no $\ol j\in \ol J$ is isolated. Let $\ol j\in \ol J\setminus \{\ol j_{\max}\}$, and let $i\in I$ such that $(i,\ol j)\in F$. 
 Then the following inequalities hold for $\tilde \sfg(F)$
 \[ \ha(i,\ol j)-\ha(i,\ol j_{\max})\leq x_{\ol j}-x_{\ol j_{\max}}\leq \ha(i_{\min},\ol j)-\ha(i_{\min},\ol j_{\max}),\]
 where all the heights are of the form $\ha(i',\ol j')$ with $i'\prec \ol j'$ and hence both the upper and the lower bound on $x_{\ol j}-x_{\ol j_{\max}}$ are finite. This shows the boundedness of $\sfg(F)$.
 
 The statement concerning the dimension follows directly from~\cite[Prop.~17]{DevelinSturmfels2004}. 
 Moreover, it is straightforward to check that $\sfg(F)\cap \sfg(F')=\sfg(F\cup F')$. Since every facet of $\sfg\ij$ is of the form $\sfg\ij\cap \sfg(i,\ol j')$ for some $\ol j'\in \ol J$, and every face 
 of $\sfg(F)$ is an intersection of such facets, we conclude that every face of $\sfg(F)$ is of the form $\sfg(F')$ for some $\IJ$-forest $F'$ containing $F$. Maximal such forests are $\IJ$-trees, which implies that the vertices of $\sfg(F)$ are of the form $\sfg(T)$ for $\IJ$-trees containing $F$.
\end{proof}

It only remains to provide the actual coordinates of the vertices $\sfg(T)$. 
\begin{lemma}\label{lem:coordinates}
For an $\IJ$-tree $T$ and for $\ol k\in\ol J\setminus\{\ol j_{\max}\}$, consider the sequence $P_T({\ol k})$ of arcs traversed in the unique oriented path from~$\ol k$ to~$\ol j_{\max}$ in $T$.
Then the coordinate $\ol k$ of $\sfg(T)$ is:
\begin{equation}\label{eq:g_tree}
\sfg(T)_{\ol k}:=\sum_{(i,\ol j)\in P_T(\ol k)}\pm \ha({i,\ol j}),
\end{equation}
where the sign of each summand is positive if $(i,\ol j)$ is traversed from $\ol j$ to $i$ and negative otherwise. 
\end{lemma}

The result follows from the following lemma, which shows that the point described in~\eqref{eq:g_tree} fulfills the desired constraints. 

\begin{lemma}\label{lem:proofcoordinates}
If $(i,\ol j)\in T$ then $\sfg(T)_{\ol k}-\sfg(T)_{\ol j}\leq \ha(i,\ol k)-\ha(i,\ol j)$ for all $\ol k \in \ol J$. Equality holds if and only if $(i,\ol k)\in T$.
\end{lemma}

\begin{proof}
Let $P_T({\ol k},\ol j)$ be the sequence of arcs traversed in the unique oriented path from~$\ol k$ to~$\ol j$ in $T$ (so $P_T(\ol k)$ in Lemma~\ref{lem:coordinates} corresponds to $P_T(\ol k,\ol j_{\max})$.
Then 
\begin{equation}\label{eq:cost}
\sfg(T)_{\ol k}-\sfg(T)_{\ol j}=\sum_{(i,\ol j)\in P_T(\ol k,\ol j)}\pm \ha({i,\ol j}),
\end{equation}
where the sign of each summand is positive if $(i,\ol j)$ is traversed from $\ol j$ to $i$ and negative otherwise. We call the right hand side of \eqref{eq:cost} the \defn{cost} of the path~$P_T(\ol k,\ol j)$, and we denote it by \defn{$\cost{P_T(\ol k,\ol j)}$}.

It is clear that, if $(i,\ol k)\in T$, then $\cost{P_T(\ol k,\ol j)} = \cost{\ol k\rightarrow i\rightarrow \ol j} = \ha(i,\ol k)-\ha(i,\ol j)$ as desired. If $\ol k\prec i$ then the strict inequality holds because $\ha(i,\ol k)=\infty$.  
We need to prove that,
if $i \prec \ol k$ and $(i,\ol k)\notin T$, then $\cost{P_T(\ol k,\ol j)} < \cost{\ol k\rightarrow i\rightarrow \ol j}$. We do it for the case when $\ol j \prec \ol k$. The case $i \prec \ol k \prec \ol j$
follows similar ideas. 

It is not hard to see that $P_T(\ol k,\ol j)$ is always of the form 
\[\ol k=\ol j_0\ra i_1\ra \ol j_1\ra \cdots\ra \ol j_r\ra i_1'\ra \ol j_1'\ra \cdots\ra i_s' \ra \ol j_s'=\ol j,\]
where
\[i'_1\prec  \cdots \prec i'_s\prec  \ol j_s' \prec \cdots \prec \ol j_1'\prec i_r\prec \cdots \prec  i_1\prec  \ol j_0'\prec \cdots \prec \ol j_r,\]
$r\geq 0$ and $s\geq 1$.
From \eqref{eq:asso_height} one deduces that 
\begin{align*}
 \cost{\ol j_{r-1}\ra i_r\ra\ol j_r\ra i'_1}&< \cost{\ol j_{r-1}\ra i'_1},\text{ and }\\
 \cost{\ol j_{r}\ra i_1'\ra\ol j_1'\ra i'_2}&< \cost{\ol j_{r}\ra  i'_2}.
\end{align*}
These two operations can be performed to simplify the path to $\ol k\rightarrow i_s'\rightarrow \ol j$, for the last $i_s'\in I$ in the path from $\ol k$ to $\ol j$. Since these operations increase the cost at every step 
\[
\cost{P_T(\ol k,\ol j)} < \cost{\ol k\rightarrow i_s'\rightarrow \ol j}.
\] 
It is not hard to see that $i_s'\prec i$. Using Equation~\eqref{eq:asso_height} for $i_s' \prec i \prec \ol j \prec \ol k$ we get that $\cost{P_T(\ol k,\ol j)} < \cost{\ol k\rightarrow i\rightarrow \ol j}$.
\end{proof}

\begin{theorem}
The polyhedral cells of the $\IJ$-associahedron $\Asso_{I,\ol J}(\ha)$ are the convex polytopes $\sfg(F)$ ranging through all covering $(I,\ol J)$-forests~$F$. The coordinates of its vertices are given by Equation~\eqref{eq:g_tree} ranging over all $\IJ$-trees $T$.
\end{theorem}

\begin{proof}
The associahedron $\Asso_{I,\ol J}(\ha)$ is defined as the polyhedral complex of bounded cells induced by the arrangement of inverted tropical hyperplanes ${\arrgtA=(H_i)_{i\in I}}$ determined by Equation~\eqref{eq:tropical_hyp_vertices}. Each such $H_i$ subdivides the space $\RR^{\ol J}$ into at most $|\ol J|$ regions $\tilde \sfg{\ij}$ for $\ol j\in \ol J$ (region $\tilde \sfg{\ij}$ is empty when $\ha(i,\ol j)=\infty$). 
By construction, each cell induced by $\arrgtA$ in $\RR^{\ol J\setminus \{\ol j_{\max}\}}$ is the intersection of at least one region from each $H_i$. Hence, the cells induced by $\arrgtA$ are the polyhedra $g(G)$ for all subgraphs $G$ of~$K_{I,\ol J}$ with no isolated nodes from $I$.

By Lemma~\ref{lem:emptycrossing}, the cell $g(G)$ is non-empty if and only if $G$ is an $(I,\ol J)$-forest. 
To see that it is covering, we first show that if $F$ is an $\IJ$-forest with no isolated nodes in $I$ but at least one isolated node $\ol j\in\ol J$ then $g(F)$ is not bounded. For this, note that all the inequalities involved in the definition of $\tilde g(F)$ that contain the variable $x_{\ol j}$ are of the form $x_{\ol j} - x_{\ol j'}\leq \ha(i,\ol j)-\ha(i,\ol j')$ for some non isolated vertices~$\ol j'$. Decreasing the value of the $\ol j$th entry of any feasible point in $\tilde g(F)$ creates an infinite ray contained in $\tilde g(F)$, and so $g(F)$ is not bounded. 

A similar argument works for the case where $(i_{\min},\ol j_{\max})\notin F$ and $i_{\min}$ is not isolated. Indeed, let $\ol k\prec \ol j_{\max}$ be the largest neighbor of $i_{\min}$. For $x\in g(F)$ and $\lambda>0$, let $x'\in \RR^{\ol J\setminus \{\ol j_{\max}\}}$ be such that $x_{\ol j}'=x_{\ol j}+\lambda$ for $\ol j\preceq \ol k$ and $x_{\ol j}'=x_{\ol j}$ for $\ol j\succ \ol k$. Note that the only inequalities of $g(F)$ that $x'$ might violate are those of the form $x'_{\ol j} - x'_{\ol j'}\leq \ha(i,\ol j)-\ha(i,\ol j')$ with $\ol j\preceq \ol k$. But if $\ol j'\preceq k$ then $x'_{\ol j} - x'_{\ol j'}=x_{\ol j} - x_{\ol j'}$, and if $\ol j'\succ \ol k$, then $i\succ \ol k\succeq \ol j$ (otherwise $(i_{\min},\ol k)$ and $(i,\ol j')$ would cross) and thus $\ha(i,\ol j)=\infty$ and the inequality is trivially verified. 

Since we have already seen that if $F$ is a covering $\IJ$-forest then $g(F)$ is a bounded polytope (Lemma~\ref{lem:verticesgF}), this shows 
that the cells $\Asso_{I,\ol J}(\ha)$ are the convex polytopes $\sfg(F)$ ranging through all covering $(I,\ol J)$-forests~$F$.
To see that different covering $(I,\ol J)$-forests give rise to different cells, consider two different covering $(I,\ol J)$-forests $F$ and $F'$ and assume without loss of generality that $F\nsubseteq F'$. The intersection $g(F)\cap g(F')=g(F\cup F')$ is either empty (if $F\cup F'$ has a crossing) or of lower dimension than $g(F)$ (adding a non-crossing arc reduces the dimension by Lemma~\ref{lem:verticesgF}).

Finally, the vertices of~$\Asso_{I,\ol J}(\ha)$ correspond to $\IJ$-trees $T$ by Lemma~\ref{lem:verticesgF}. Their coordinates $g(T)$ are determined by Equation~\eqref{eq:g_tree} in Lemma~\ref{lem:coordinates}.
\end{proof}

In particular, we recover a tropical realization of the classical associahedron when we take $\IJ=([n],[\ol n])$ (tropical realizations of associahedra had already been found, see~\cite{JoswigKulas2010,RoteSantosStreinu2003}).
\begin{corollary}\label{cor:classicalassociahedron}
$\Asso([n],[\ol n])$ is a classical $(n-1)$-dimensional associahedron.
\end{corollary}

\begin{example}\label{ex:ExampleIJtrees}
 We want to illustrate this construction with a concrete example. Take $I=\{0,1,3,4,6,7\}$ and $\ol J=\{\ol 2,\ol 5,\ol 8\}$. There are seven $\IJ$-trees, depicted in Figure~\ref{fig:IJTamaris} (right), together with 
 their $\IJ$-Tamari ordering.
We use the non-crossing height function $\ha(i,\ol j)=-(j-i)^2$, which produces the following values
\[
\ha(i,\ol j)= 
\begin{blockarray}{ccccccc}
   & 0&1&3&4&6&7 \\
    \begin{block}{c(cccccc)}
    \ol 8 & -64 & -49 & -25 & -16 & -4 & -1 \\
    \ol 5 & -25 & -16 & -4 & -1 & \infty & \infty \\
    \ol 2 & -4 & -1 & \infty & \infty & \infty & \infty\\
    \end{block}
  \end{blockarray}
\]
For the tree $T_1$, we get that the path from $\ol 2$ to $\ol 8$ is $\ol 2\to 1\to \ol 8$, and the path from $\ol 5$ to $\ol 8$ is $\ol 5\to 4\to \ol 8$. Therefore,
\begin{align*}
 \sfg(T_1)_{\ol 2}&=\ha(1,\ol 2)-\ha(1,\ol 8)=-1+49=48, \text{ and }\\
 \sfg(T_1)_{\ol 5}&=\ha(4,\ol 5)-\ha(4,\ol 8)=-1+16=15.
\end{align*}
Hence, $\sfg(T_1)=(48,15)$. If we do the same procedure for the remaining six $\IJ$-trees, we obtain:
\begin{align*}
\sfg(T_2)&=(48,21),&
\sfg(T_3)&=(48,33),&
\sfg(T_4)&=(54,39),\\
\sfg(T_5)&=(60,15),&
\sfg(T_6)&=(60,21),&
\sfg(T_7)&=(60,39).
\end{align*}
Plotting these points, and all the cells corresponding to covering $\IJ$-forests, we get our realization of the $\IJ$-associahedron in Figure~\ref{fig:ExampleIJasso}. In terms of tropical hyperplanes we get $H_0,H_1,H_3,H_4,H_6$ and $H_7$ centered at the following vertices:
{
\begin{align*}
v_0&=(-4,-25,-64)\sim (60,39,0),&
v_1&=(-1,-16,-49)\sim (48,33,0),\\
v_3&=(\infty, -4,-25)\sim (\infty, 21, 0),&
v_4&=(\infty,-1,-16)\sim(\infty,15,0),\\
v_6&=(\infty,\infty,-4)\sim(\infty,\infty,0),&
v_7&=(\infty,\infty,-1)\sim(\infty,\infty,0).
\end{align*}
}
Note that since some of the entries are infinity we get degenerate tropical hyperplanes. For instance $H_4$ and $H_6$ are just horizontal lines while $H_6$ and $H_7$ do not appear because they lie at infinity.

 \begin{figure}[hptb]
  \includegraphics[width=\linewidth]{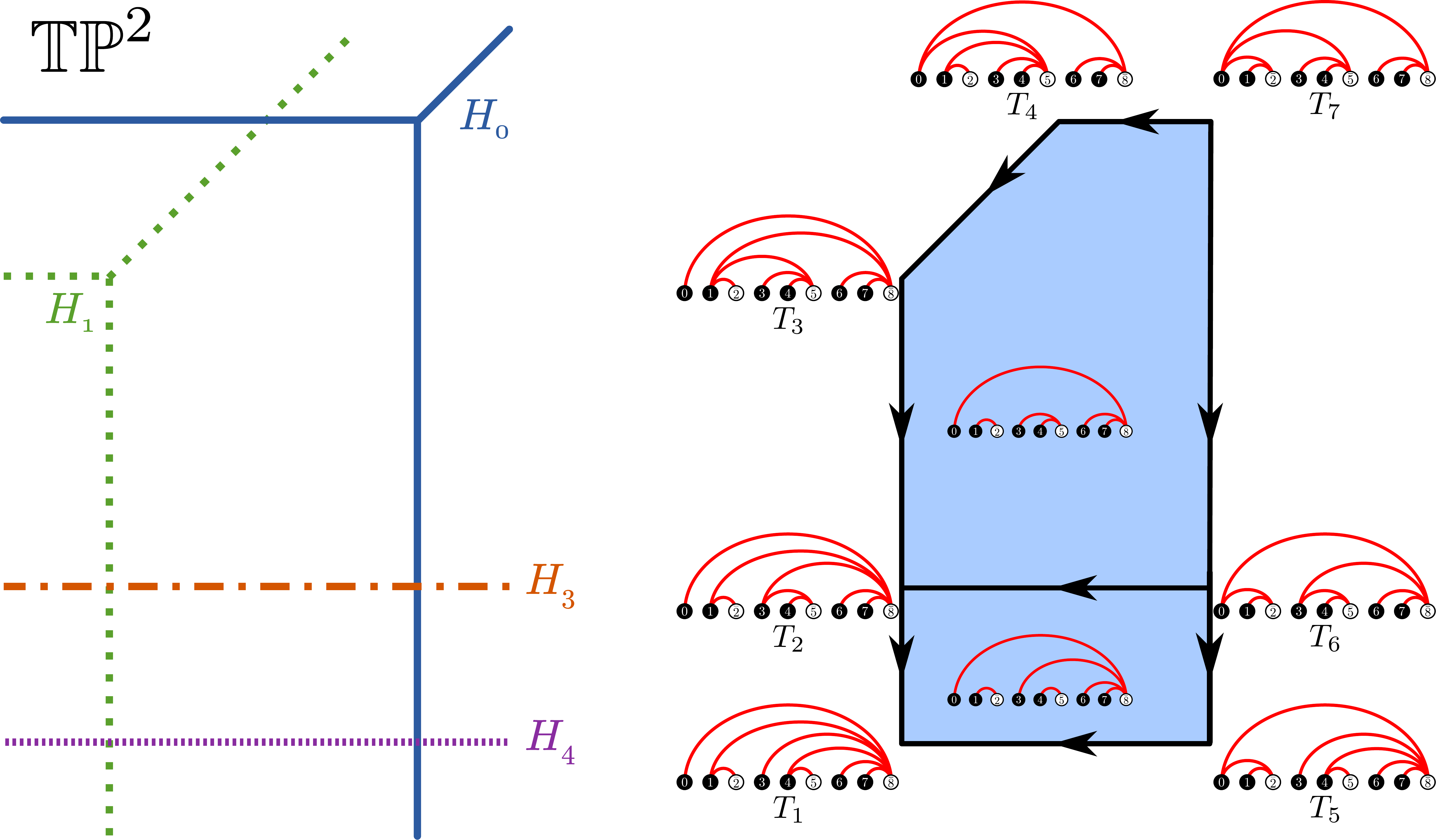}
  \caption{Left: Arrangement of tropical hyperplanes whose bounded faces constitute the $\IJ$-associahedron from Example~\ref{ex:ExampleIJtrees} (hyperplanes $H_6,H_7$ lie at infinity, and are not included). Right: $\IJ$-associahedron with its vertices labeled by $\IJ$-trees, its $2$-dimensional faces labeled by minimal covering $\IJ$-forests, and its edges oriented according to the $\IJ$-Tamari order.}\label{fig:ExampleIJasso}
 \end{figure}
\end{example}


\subsection{Orientation of the edges}

We will now study some combinatorial and geometric properties of $\Asso_{I,\ol J}(\ha)$. A first interesting property of the classical associahedron that is preserved in this setup is that we can use a linear functional to orient its graph according to the $\IJ$-Tamari lattice. (In retrospect, this provides a geometric proof of Lemma~\ref{lem:shellingorder}, because linear functionals can be used to define shelling orders.)
\begin{theorem}\label{thm:linearOrientation}
 The linear functional $-\sum_{\ol j\in \ol J\setminus \ol j_{\max}} x_{\ol j} $ orients the edges of $\Asso_{I,\ol J}(\ha)$ according to the $\IJ$-Tamari covering relations.
\end{theorem}
\begin{proof}
 Let $i_1\prec i_2\prec\ol j_1\prec\ol j_2$, and let $T,T'$ be two $\IJ$-trees such that $T'$ is obtained from $T$ by flipping $(i_1,\ol j_1)$ to $(i_2,\ol j_2)$. That is, $T'$ covers $T$ in the $\IJ$-Tamari order. We want to study the vector $\sfg(T')-\sfg(T)$.
 
 Let us focus on the $\ol k$th coordinate (see Equation~\eqref{eq:g_tree}). If the path $P_T(\ol k)$ from $\ol k$ to $\ol j_{\max}$ in~$T$ does not use the arc $(i_1,\ol j_1)$, then $P_T(\ol k)=P_{T'}(\ol k)$ and $\sfg(T')_{\ol k}-\sfg(T)_{\ol k}=0$. 
 
 If $P_T(\ol k)$ contains the arc $(i_1,\ol j_1)$, then $P_{T'}(\ol k)$ is the symmetric difference of $P_{T}(\ol k)$ with the $4$-cycle $\ol j_1\rightarrow i_1 \rightarrow \ol j_2\rightarrow i_2 \rightarrow \ol j_1$. From the fact that $(i_1,\ol j_1)$ induces an increasing flip, we know that in any path to $\ol j_{\max}$ it should be traversed from $\ol j_1$ to $i_1$ (compare Figure~\ref{fig:IJflipschema}). Hence, we obtain that
 \[\sfg(T')_{\ol k}-\sfg(T)_{\ol k}=\ha(i_1,\ol j_2)-\ha(i_2,\ol j_2)+\ha(i_2,\ol j_1)-\ha(i_1,\ol j_1)\stackrel{\eqref{eq:asso_height}}{<}0,
 \]
 where we use that $\ha$ is non-crossing height function (cf.~Proposition~\ref{prop:assheights}).
 
 Therefore, on an $\IJ$-increasing flip, no coordinate of $\sfg(T)$ increases and at least one decreases, which finishes our claim.
\end{proof}


\subsection{The support}

The \defn{support $\supp(\cK)$} of a polyhedral complex $\cK$ embedded in~$\RR^d$ is defined as the union of all polyhedral cells in $\cK$.
Bergeron's pictures of $m$-Tamari lattices~\cite[Figures~4 and 6]{Bergeron2012} suggest that the supports of $\IJ$-associahedra are classical associahedra. 
We show that this holds for an important family of pairs $\IJ$, including those corresponding to $m$-Tamari lattices.

\begin{theorem}\label{thm:supportAssociahedron}
Let $\ol J'=\{\ol j\in \ol J\ \colon\ \exists\ i_1\prec i_2 \prec \ol j\}$.
Then $\supp(\AssoIJ(\ha))$ is convex if and only if $I\sqcup \ol J'\setminus \max {\ol J}$ does not have a consecutive pair of elements of $\ol J$.
In this case, $\AssoIJ(\ha)$ is a regular polyhedral  subdivision of a classical associahedron of dimension $(|\ol J'|-1)$.
\end{theorem}

\begin{figure}[htpb]
\centering 
 \includegraphics[width=.5\linewidth]{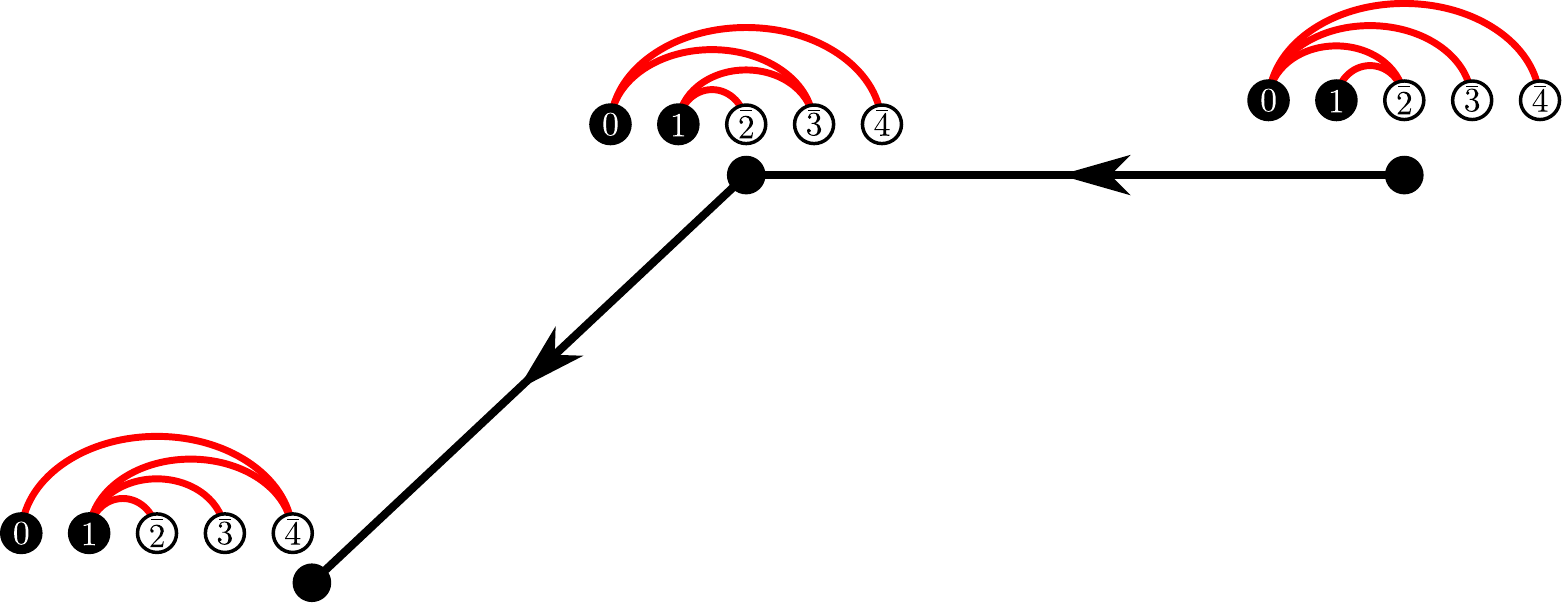}
 \caption{The $(\{0,1\},\{\ol2,\ol3,\ol 4\})$-associahedron.}\label{fig:folded}
\end{figure}

\begin{proof}
Throughout the proof we omit $\ha$ in the notation for $\AssoIJ(\ha)$ and write~$\AssoIJ$ for simplicity.
\item
\emph{Convexity:} We show first that if $\ol j_1,\ol j_2$ are two consecutive elements of $\ol J$ in $I\sqcup \ol J'\setminus \max {\ol J}$, then $\supp(\AssoIJ)$ is not convex. The proof is based in the non-convex example in Figure~\ref{fig:folded} (where $j_1=2$ and $j_2=3$). 

Let $i_1\prec i_2$ be the two smallest elements of $I$, and let $F$ be a covering $\IJ$-forest. If $(i_1,\ol j_2)\in F$, then every point in $\sfg(F)$ fulfills
\begin{equation}\label{eq:forestcase1}
 x_{\ol j_2}\geq \ha(i_1,\ol j_2)-\ha(i_1,\ol j_{\max}).
\end{equation}
Otherwise, $(i,\ol j_2)\in F$ for some $i \prec \ol j_1\prec \ol j_2$ (because $\ol j_1$ and $\ol j_2$ are consecutive and $i\prec \ol j_2$). In this case, every point in $\sfg(F)$ fulfills
\begin{equation}\label{eq:forestcase2}
 x_{\ol j_1}-x_{\ol j_2}\leq \ha(i,\ol j_1)-\ha(i,\ol j_{2})\leq \ha(i_2,\ol j_1)-\ha(i_2,\ol j_{2}),
\end{equation}
where for the second inequality we combine \eqref{eq:asso_height} together with $i_2\preceq i \prec \ol j_1\prec \ol j_2$.

Now consider the non-crossing arcs $(i_1,\ol j_{\max})$, $(i_1,\ol j_2)$, $(i_1,\ol j_1)$ and $(i_2,\ol j_1)$ and complete them to an $\IJ$-tree $T_1$. The following relations hold for the point $\sfg(T_1)$:
\begin{align}
 \sfg(T_1)_{\ol j_2}&=\ha(i_1,\ol j_2)-\ha(i_1,\ol j_{\max}),\text{ and }\\
 \sfg(T_1)_{\ol j_1}-\sfg(T_1)_{\ol j_2}&=
 \ha(i_1,\ol j_1)-\ha(i_1,\ol j_2) \stackrel{\eqref{eq:asso_height}}{>}\ha(i_2,\ol j_1)-\ha(i_2,\ol j_2).
\end{align}
And consider also an $\IJ$-tree $T_2$ containing the arcs $(i_2,\ol j_{\max})$, $(i_2,\ol j_2)$, $(i_2,\ol j_1)$ and $(i_1,\ol j_{\max})$. The following relations hold for the point $\sfg(T_2)$:
\begin{align}
 \sfg(T_2)_{\ol j_2}&=\ha(i_2,\ol j_2)-\ha(i_2,\ol j_{\max})\stackrel{\eqref{eq:asso_height}}{<}\ha(i_1,\ol j_2)-\ha(i_1,\ol j_{\max}),\text{ and }\\
 \sfg(T_2)_{\ol j_1}-\sfg(T_2)_{\ol j_2}&=\ha(i_2,\ol j_1)-\ha(i_2,\ol j_2).
\end{align}
Therefore, the midpoint between $\sfg(T_1)$ and $\sfg(T_2)$ does not fulfill neither \eqref{eq:forestcase1} nor~\eqref{eq:forestcase2}, and hence does not belong to $\supp(\AssoIJ)$.

\item \emph{Associahedral support:}
Assume now that there are no two consecutive elements of~$\ol J$, we will show that the support is an associahedron. The proof is by induction on the number of consecutive pairs of elements of $I$. If there is only one, the one involving $\min I$, then up to redundant elements (see part (1) in Remark~\ref{rem:affineasso}) $\IJ$ is an alternating sequence
 and hence $\AssoIJ$ is a $(|\ol J'|-1)$-associahedron by Corollary~\ref{cor:classicalassociahedron}. 
 
 Assume that $i_0\in I\setminus\{\min I\}$ is immediately followed by an element in $I$, and let $I':=I\setminus i_0$. Then, by induction hypothesis, $\Asso_{I', \ol J}$ is a 
 polyhedral subdivision of a full-dimensional classical associahedron. We consider also the set of regions $S_{\ol j}=\sfg(i_0,\ol j)$ with $\ol j$ ranging through $\ol J$. These are the maximal
 cells of a polyhedral subdivision $S$ of $\RR^{|\ol J|-1}$. We will prove that $\AssoIJ$ is the refinement of $\Asso_{I', \ol J}$ with respect to $S$. 
 Since the common refinement of regular subdivisions is regular --- because the sum of the piecewise affine convex functions supporting the corresponding liftings is a regular lifting of the refinement --- this will directly imply that this subdivision of $\AssoIJ$ is regular.
 
 On the one hand, consider an arbitrary cell of $\Asso_{I', \ol J}$, which is of the form~$\sfg(F)$ for a covering $(I', \ol J)$-forest $F$, and a region $S_{\ol j}$.
 By Lemma~\ref{lem:emptycrossing} $\sfg(F)\cap S_{\ol j}={\sfg(F\cup (i_0,\ol j))}$ is not empty if and only if $F\cup (i_0,\ol j)$ is
 an $\IJ$-forest. This means that $\sfg(F\cup (i_0,\ol j))$ is a cell of $\AssoIJ$, and similarly that every cell in the refinement of~$\Asso_{I', \ol J}$ with respect to $S$ is a cell of~$\AssoIJ$.
 
It remains to see that every cell of $\AssoIJ$ arises this way. 
This is equivalent to showing that removing the arcs incident to $i_0$ in any covering $\IJ$-forest always produces a covering $(I', \ol J)$-forest.  
To show this, note that the removal of an arc $(i_0,\ol j)$ involving $i_0$ does not isolate any element of $\ol J$ because $\ol j$ is connected to another element of $I$: its immediately preceding element or $\min I$. This element is always different from $i_0$ since it is not $\min I$ and is not followed by an element in~$\ol J$. Hence, the $(I',\ol J)$-forest obtained after removing the incident arcs to $i_0$ is still covering, which finishes the proof.
\end{proof}

Translating the previous result to lattice paths, we obtain the following corollary. 

\begin{corollary}\label{cor:convex_regsubdivision}
The $\nu$-associahedron~$\Asso_\nu(\ha)$ is convex only if $\nu$ does not have two non-initial consecutive north steps. In this case, $\Asso_\nu(\ha)$ is a regular polyhedral subdivision of a classical associahedron of dimension equal to the number of non-initial north steps in $\nu$. 
\end{corollary}

This result holds in particular for the Fuss-Catalan and rational Catalan cases, as well as for any path $\nu$ lying above a line. 

\begin{remark}[Non-convex and non-pure $\IJ$-associahedra]
As noticed in Theorem~\ref{thm:supportAssociahedron} the $\IJ$-associahedron is not always convex. For example, for $(I,\ol J)=(\{0,1,2,3\},\{\ol 4, \ol 5,\ol 6\})$ the $\IJ$-associahedron is neither  convex nor pure; see Figure~\ref{fig:degenerate}.
\end{remark}

\begin{figure}[htpb]
\centering 
 \includegraphics[width=\linewidth]{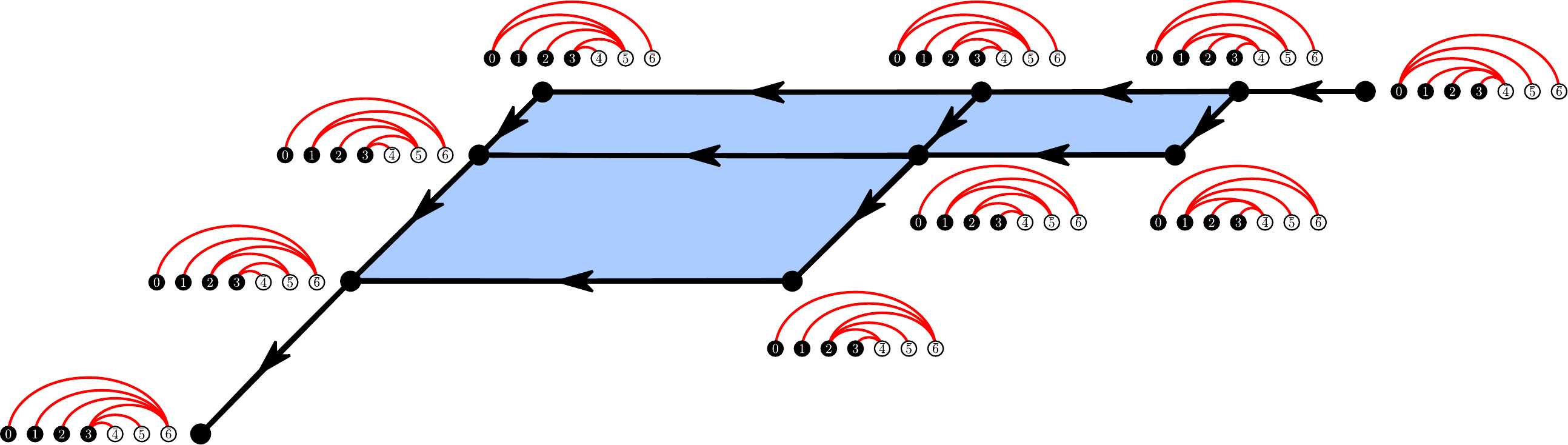}
 \caption{The $(\{0,1,2,3\},\{\ol 4, \ol 5,\ol 6\})$-associahedron is not pure.}\label{fig:degenerate}
\end{figure}

\begin{remark}[Affinely equivalent $\IJ$-associahedra]
\label{rem:affineasso}
The are two operations that do not change much the associahedron:

\noindent
(1) \emph{Removing redundant elements:}  
Observe that if we add north steps at the beginning (or east steps at the end) of a lattice path~$\nu$, the corresponding $\nu$-Tamari lattice does not change. This reflects the fact that
all the elements of $\ol J$ immediately after the first element $i_{\min}$ are always leaves connected to $i_{\min}$ in every $\IJ$-tree, and analogously for the elements of $I$ immediately
before the last element $\ol j_{\max}$. In the corresponding associahedra, this gets reflected in an affine isomorphism. 

\noindent
(2) \emph{The reverse operation~$\reverse\IJ$:}  this operation corresponds to switching the role of the two simplices of $\productp{I}{\ol J}$. The symmetry of the factors induces a duality between arrangements of $n+1$ tropical hyperplanes in $\TP^{m}$ and arrangements of $m+1$ tropical hyperplanes in $\TP^n$. If we apply \cite[Lemma~22]{DevelinSturmfels2004} we obtain that $\Asso_{I,\ol J}(\ha)$ is piecewise-affinely isomorphic to $\Asso_{\reverse{I,\ol J}}(\ha)$.

The piecewise-affine isomorphism cannot be replaced by an affine isomorphism, since when they are not full dimensional, $\IJ$-associahedra appear sometimes ``folded''. For example, Figure~\ref{fig:folded}
shows the $(\{0,1\},\{\ol2,\ol 3,\ol 4\})$-associahedron (that is, the $\sfE\sfN\sfN$-associahedron), which is a `folded' image in $\RR^2$ of the $(\{0,1,2\},\{\ol 3,\ol 4\})$-associahedron (that is, the $\sfE\sfE\sfN$-associahedron), which is a subdivision of a $1$-dimensional associahedron.
\end{remark}


\subsection{Cells and faces}

Every face of the associahedron is a Cartesian product of associahedra. This fact has a combinatorial and a geometric counterpart reflected in $\AssoIJ$.

The first concerns the combinatorial type of the cells of $\AssoIJ$. The following proposition
follows directly from Lemma~\ref{lem:links}, after noticing that the cells of $\AssoIJ$ are combinatorially dual to links of covering $\IJ$-forests in~$\asstri{I,\ol J}$.

\begin{proposition}\label{prop:cells}
The polyhedral cells of the polyhedral complex $\AssoIJ$ are isomorphic to Cartesian products of (classical) associahedra.
\end{proposition}

For the second consequence of Lemma~\ref{lem:links}, which has a more geometric flavor, we restrict to the setup of Theorem~\ref{thm:supportAssociahedron} where the support is an associahedron. We want
to understand the restriction of $\AssoIJ(\ha)$ to the faces of this associahedron.
That is, the polyhedral subdivisions $H\cap\AssoIJ(\ha)$ for supporting hyperplanes $H$ of the associahedron. These are those such that for every $\IJ$-tree $T$, $\sfg(T)$ is contained in the closed halfspace $H^\leq$. 

\begin{proposition}
Assume that $|I|\geq |\ol J|$, $I\sqcup \ol J\setminus (\max\ol J)$ has no consecutive pair of elements of $\ol J$, 
and that for every $\ol j\in \ol J$ there are $i_1, i_2\in I$ such that $i_1\prec i_2\prec \ol j$.
Then, for every supporting hyperplane $H$ of $\AssoIJ(\ha)$, $\AssoIJ(\ha)\cap H$ is a Cartesian product of $\IJ$-associahedra.
\end{proposition}
\begin{proof}
It suffices to consider the case when $H$ supports a facet of the associahedron $\supp(\AssoIJ(\ha))$. In this case, $H$ is of the form $H_{\ol k,i,\ol j}=\{x_{\ol k}-x_{\ol j}= \ha(i,\ol k)-\ha(i,\ol j)\}$
for an unavoidable arc $\ij$ of $\AssoIJ(\ha)$ (that is, either $i$ and $\ol j$ are consecutive in $I\sqcup \ol J$ or  $(i,\ol j)=(\min I,\max \ol J)$), and an element $i\prec \ol k\in \ol J$.
 
By Lemma~\ref{lem:proofcoordinates}, $\sfg(T)\in H_{\ol k,i,\ol j}$ if and only if $(i,\ol k)\in T$. The result is then a direct consequence of Proposition~\ref{lem:links}, since $\AssoIJ(\ha)\cap \in H_{\ol k,i,\ol j}$ would be dual to the link of $(i,\ol k)$ in $\asscomp{I,\ol J}$.
\end{proof}


\part{Type \texorpdfstring{$B$}{B}}

\section{The cyclohedral triangulation}\label{sec:cyctri}

Whereas the (type $A$) associahedron concerns triangulations of a polygon, its type $B$ analogue, the cyclohedron, deals with centrally symmetric triangulations of a centrally symmetric polygon. Our definitions of $\nn$-trees and the $\nn$-associahedral triangulation can be extended to produce type~$B$ analogues using a cyclic symmetry of $\productn$ (compare~\cite{CeballosPadrolSarmiento2015}).

Let $P_{2n+2}$ be a convex $(2n+2)$-gon, with its edges labeled counterclockwise from $0$ to~${2n+1}$ modulo $n+1$, as in Figure~\ref{fig:cstriang2tree}, and let $T$ be a \emph{centrally symmetric triangulation of $P_{2n+2}$} (henceforth abbreviated cs-triangulation). We associate to~$T$ a spanning tree of $\bipartitebp{n}{n}$ according to the following procedure:
First, replace each boundary edge $i$ of $P_{2n+2}$ with an arc $(i,\ol i)$ of $\bipartitebn$, such that $i$ comes before~$\ol i$ when the boundary of $P_{2n+2}$ is traversed in the counterclockwise order. We will have this order  in mind throughout. Second, replace each diagonal of $T$ of $P_{2n+2}$ with an arc $(i,\ol j)$, where $i$ and $\ol j$ are the first and last edges of $P_{2n+2}$ that are covered by the diagonal after a radial projection from the center of $P_{2n+2}$ to its boundary (we implicitly assume $P_{2n+2}$ is a regular polygon). Here, the main diagonal of $P_{2n+2}$ gets replaced with two arcs of the form $(i,\ol{i+n}) \pmod{n+1}$, corresponding to the two opposite directions of radial projection from the center. We end with two copies of a subgraph of $\bipartitebp{n}{n}$, which is actually a spanning tree of $\bipartitebp{n}{n}$, as we show in Lemma~\ref{lem:cyclicIJtreesaretrees}. 

Trees obtained with this procedure are called \defn{cyclic $\nn$-trees}. An example is shown in Figure~\ref{fig:cstriang2tree}. We draw cyclic $\nn$-trees on (the surface of) a cylinder to make the parallel with $\nn$-trees more evident. Observe that we may draw all arcs increasingly by having them \emph{wind} around the cylinder, as necessary.

\begin{figure}[htpb]
 \centering 
 \includegraphics[width=.9\linewidth]{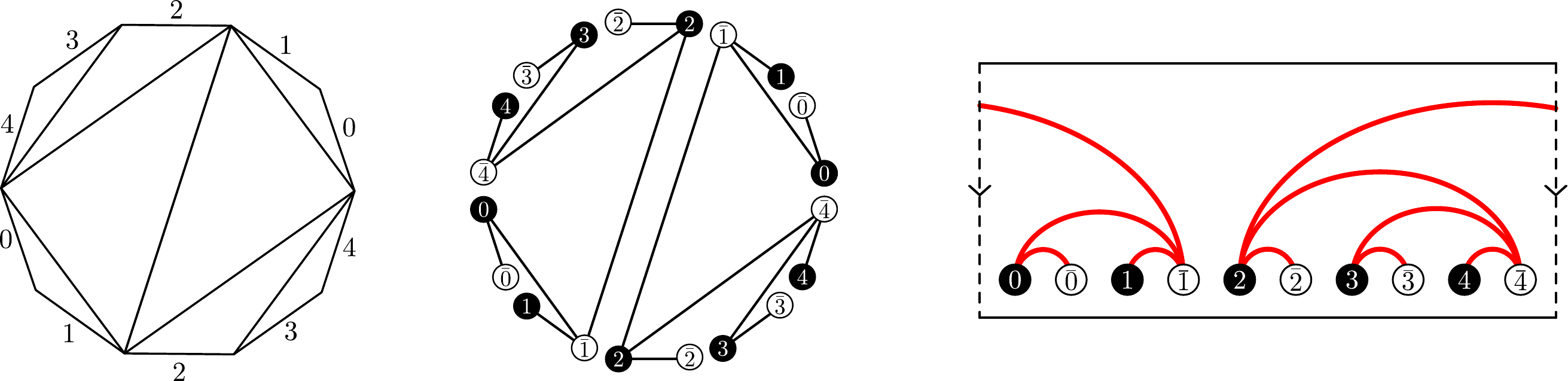}
 \caption{From a cs-triangulation to a non-crossing alternating tree on the cylinder.}\label{fig:cstriang2tree}
\end{figure}

Recall that, when seen as subsimplices of $\productn$, ordinary $([n],[\ol n])$-trees cover the subpolytope~$\catblock{n}\subset\productn$. In contrast, if we take all cyclic $([n],[\ol n])$-trees we obtain a full triangulation of $\productn$. We call it the \defn{$n$-cyclohedral triangulation~$\cyctri{n}$} because two maximal simplices of $\cyctri{n}$ are adjacent if and only if the corresponding cs-triangulations differ by a flip. Hence the dual of $\cyctri{n}$ is the (simple) cyclohedron.

\begin{theorem}\label{thm:cyclohedrontriangulation}
The set of cyclic $\nn$-trees indexes the maximal simplices of a flag regular triangulation of $\productn$, which we call the \defn{$n$-cyclohedral triangulation~$\cyctri{n}$}.
\end{theorem}

Theorem~\ref{thm:cyclohedrontriangulation} lies in effect at the quintessence of our constructions and results. To prove it, we split it into two statements. We show in Section~\ref{sec:triangulation} that $\cyctri{n}$ is a flag triangulation, and in Section~\ref{sec:regularity} that $\cyctri{n}$ is regular.

For the moment, let us present some facts relating to the cyclohedral triangulation $\cyctri{n}$ that parallel results of the $n$-associahedral triangulation in Section~\ref{sec:assTri}.

As with the associahedral triangulation, it is natural to consider the restriction of the cyclohedral triangulation $\cyctri{n}$ to faces $\productp{I}{\ol J}$ of $\productn$, where $I\subseteq[n]$, $\ol J\subseteq[\ol n]$. This setting gives rise to the notion of \defn{cyclic $\IJ$-trees}. As the adjective cyclic suggests, it amounts to dropping the condition of being increasing and adapting the definition of non-crossing to a cyclic setup.

\begin{definition}\label{def:cyclicIJtree}
Let $I\subseteq [n]$ and $\ol J\subseteq [\ol n]$ be nonempty subsets, for some $n\in\NN$. A \defn{cyclic $(I,\ol J)$-forest} is a subgraph $F$ of $\bipartitep{I}{J}$ that is \textbf{cyclically non-crossing}, where two arcs $(i,\ol j),\ (i',\ol j') \in F$ \defn{cyclically cross} if one of the following conditions holds:
\begin{align}
\label{eq:cycliccross1}
j-i'<j-i \ &\text{ and }\ j-i' < j'-i'   \pmod{n+1}, \\
\label{eq:cycliccross2} \text{or }\ j'-i<j-i \ &\text{ and }\ j'-i < j'-i'  \pmod{n+1},
\end{align}
(see Figure~\ref{fig:cycliccrossings}). A \defn{cyclic $(I,\ol J)$-tree} is a maximal cyclic $(I,\ol J)$-forest.
\end{definition}

\begin{figure}[htpb]
\centering 
 \includegraphics[width=\linewidth]{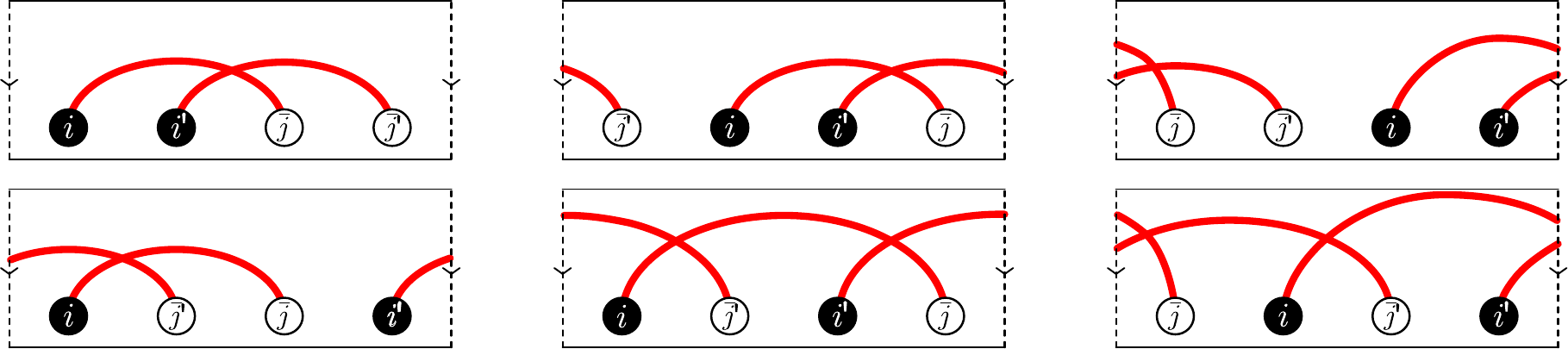}
 \caption{Pairs of arcs that infringe Conditions~\eqref{eq:cycliccross1} or~\eqref{eq:cycliccross2}, and therefore cyclically cross.}\label{fig:cycliccrossings}
\end{figure}

An argument parallel to the one in Remark~\ref{rem:bijectionTriTree} shows that cyclic $\nn$-trees are indeed in bijection with cs-triangulations of a $(2n+2)$-gon $\ngon{2n+2}.$

\begin{corollary}\label{IJcyclic_triangulation}
Let $I\subset \NN$ and $\ol J\subset \ol \NN$ be nonempty finite subsets. The set of cyclic $(I,\ol J)$-trees indexes the maximal simplices of a flag regular triangulation of $\productp{I}{J}$. We call it the \defn{$\IJ$-cyclohedral triangulation $\cyctri{I,\ol J}$} 
\end{corollary}


\subsection{The \texorpdfstring{$n$}{n}-cyclohedral triangulation is a flag triangulation}
\label{sec:triangulation}

We start by verifying that the name \emph{tree} in Definition~\ref{def:cyclicIJtree} (and hence also in Definition~\ref{def:IJtree}) is justified.

\begin{lemma}
\label{lem:cyclicIJtreesaretrees}
Let $I\subseteq [n]$ and $\ol J\subseteq [\ol n]$ be nonempty subsets, for some $n\in\NN$. Then cyclic $\IJ$-trees are spanning trees of $\bipartitep{I}{J}$.
\end{lemma}
\begin{proof}
 Let $G$ be a maximal cyclically non-crossing subgraph of $\bipartitep{I}{J}$. The proof that $G$ is a spanning tree is by induction on $|I|+|\ol J|$, and trivial when $|I|=1$ or $|\ol J|=1$.

 Assume that $|I|,|\ol J|\geq 2$ and consider elements $i\in I$ and $\ol j\in \ol J$ that are consecutive in the cyclic order of the elements of $I\sqcup \ol J$. One easily verifies that the arc $(i,\ol j)$ cannot cross any other arc, and hence it belongs to $G$ by maximality. Moreover, either $i$ or $\ol j$ is a leaf of $G$. Indeed, the condition of being consecutive implies that for each $i'\neq i$ and $\ol j'\neq \ol j$ the arcs $(i',\ol j)$ and $(i,\ol j')$ would cross, since then $j-i<j-i' \pmod{n+1}$ and $j-i<j'-i \pmod{n+1}$, which is Condition~\eqref{eq:cycliccross1}. Assume without loss of generality that $i$ is a leaf. We conclude by induction after observing that $G\setminus (i,\ol j)$ is a maximal cyclically non-crossing subgraph of $\bipartitep{I\setminus i}{J}$. 
\end{proof}

\begin{proposition}
\label{prop:cycIsTri}
 The set of cyclic $([n],[\ol n])$-trees indexes the maximal simplices of a triangulation $\cyctri{n}$ of $\productn$ which is a flag simplicial complex
\end{proposition}

Our proof of Proposition~\ref{prop:cycIsTri} makes use of the following two lemmas. 

\begin{lemma}[{cf.~\cite[Lemma 6.2.8]{DeLoeraRambauSantos2010}}]\label{lem:bipartite}
 Let $G$ be a connected subgraph of $\bipartiteb$, and $P_G=\conv\{(\bfe_i,\bfe_{\ol j})\ :\ (i,\ol j)\in G\}\subset \product$ the associated subpolytope of $\product$. Triangulations of $P_G$ are in bijection with families $\cT$ of $\frac{\vol(P_G)}{(n+m)!}$ spanning trees of $G$ such that 
 there is no cycle of $G$ that alternates between two trees $T_1$ and $T_2$ of~$\cT$.
\end{lemma}

\begin{lemma}\label{lem:singlematching}
Let $I\subset\NN,\ol J\subset \ol\NN$ be finite subsets of the same cardinality. There exists a unique cyclically non-crossing matching in $\bipartiteij$.
\end{lemma}
\begin{proof}
We can assume that every prefix of $I\sqcup\ol J$ has at least as many elements in $I$ as in $\ol J$. Indeed, cyclic rotations do not change non-crossing matchings, and we can cyclically rotate $I\sqcup \ol J$ so that for every prefix is dominated by $I$. There is always at least one such rotation, but there can be many.
 
Now consider the matching that matches each element $i\in I$ to the smallest $\ol j$ such that the interval $[i,\ol j]$ fulfills $| I\cap [i,\ol j] | =| \ol J\cap [i,\ol j] |$. Notice that this matching is always non-crossing by construction: If $\ij$ is an arc, every $i'\in [i,\ol j]\cap I$ is matched to a $\ol j'\in [i,\ol j]\cap \ol J$.

This is the only cyclically non-crossing matching with this support. If $i$ was matched with some $\ol j'\prec \ol j$, then some $i'\in [i,\ol j']\cap I$ would have to have a neighbor outside this interval (there are not enough elements of $\ol J$ in the interval to match all of them), inducing a crossing. Similarly, if $i$ was matched with some $\ol j'\succ \ol j$, then $\ol j$ would have to be matched to some $i'\succ i$, and a symmetric argument would apply.
\end{proof}

\begin{proof}[Proof of Proposition~\ref{prop:cycIsTri}]
 We use the characterization of triangulations of subpolytopes of $\product$ from Lemma~\ref{lem:bipartite}. By Lemma~\ref{lem:cyclicIJtreesaretrees}, cyclic $([n],[\ol n])$-trees are spanning trees of $\bipartitebp{n}{n}$, so they index maximal dimensional simplices in $\productn$.
 
 To check that they intersect properly, we need to verify that no pair of cyclic $([n],[\ol n])$-trees spans an alternating cycle (a cycle that alternates between arcs of the two trees). Assume there was such a cycle. The two trees would then induce a pair of disjoint perfect matchings on the support of the cycle. These would be cyclically non-crossings because the trees are. This is a contradiction because there is a single such a matching by Lemma~\ref{lem:singlematching}.
 
 To finish the proof, observe that there are exactly $\binom{2n}{n}$ cyclic $([n],[\ol n])$-trees. Indeed, the number of cs-triangulations of $P_{2n+2}$ is given by $n+1$ choices of the main diagonal of $P_{2n+2}$ times $\frac{1}{n+1}\binom{2n}{n}$ triangulations of $P_{n+2}$ for every choice of main diagonal. 
 By the unimodularity of $\productn$ (cf.~\cite[Prop.~6.2.11]{DeLoeraRambauSantos2010}), every triangulation of $\productn$ has the same number of facets, so they must form a triangulation.
 
Finally, note that the triangulation is flag because its minimal nonfaces are cyclically crossing pairs of arcs of $\bipartitebn$.
\end{proof}


\subsection{The \texorpdfstring{$n$}{n}-cyclohedral triangulation is a regular triangulation}\label{sec:regularity}

To study the regularity of $\cyctri{n}$, let us define the \defn{length $\lth\ij$} of an arc $\ij$ as: 
\begin{equation}
 \lth\ij=j-i \pmod {n+1}.
\end{equation}
Thus, $\lth\ij=j-i$ if $j\geq i$ and $\lth\ij=n+1+(j-i)$ if $j<i$. Note that $\lth+1$ is the distance from $i$ to $\ol j$, moving counterclockwisely in our pictures  (cf. Figure~\ref{fig:cstriang2tree}). Therefore, we can extend it to arbitrary pairs of nodes. 

\begin{proposition}
\label{prop:cycheights}
Let $\hb\colon\{(i,\ol j)\in [n]\times[\ol n]\}\to \RR$ be a height function. The regular triangulation of $\productn$ induced by $\hb$ is $\cyctri{n}$ if and only if: 
\begin{equation}\label{eq:cycheights}
\hb(i,\ol j)+\hb(i',\ol{j'})<\hb(i',\ol j)+\hb(i',\ol{j})
\end{equation}
whenever $(i,\ol j),(i',\ol{j'})$ are cyclically non-crossing.  We say any such height function is \defn{cyclically non-crossing}.
\end{proposition}

This result 
immediately follows from the following characterization for regular flag unimodular triangulations of lattice polytopes~\cite[Lemma~3.3]{SantosStumpWelker2014}, which we present directly specialized to triangulations of $\product$.

\begin{lemma}[{\cite[Lemma~3.3]{SantosStumpWelker2014}}]\label{lem:regularitycharacterization}
A flag triangulation $T$ of $\product$ is the regular triangulation corresponding to the height vector $h\colon [n]\times [\ol m]\to \RR$ if and only if for every pair $\left((i,\ol j),(i',\ol {j'})\right)$ forming an edge of the complex $T$, we have
\begin{equation*}
h(i,\ol j)+h(i',\ol {j'})<h(i',\ol {j})+h(i,\ol{j'}).
\end{equation*}
\end{lemma}

\begin{lemma}
\label{lem:cyc-non-crossing}
The height function $\hb\ij=f(\lth\ij)$ is cyclically non-crossing for every strictly increasing, strictly concave function $f$. Explicit examples of cyclically non-crossing height functions are $\hb\ij=\sqrt{\lth\ij}$ or $\hb\ij=-c^{-\lth\ij}$ for some $c>1$.
\end{lemma}

In particular, $\cyctri{n}$ is the pulling triangulation of $\productn$ with respect to every order of $[n]\times[\ol n]$ that extends the partial order $(i,\ol j)\prec (i',\ol j') \Leftrightarrow \lth(i,\ol j) > \lth (i',\ol j')$.

\begin{proof}
Since the crossing property and the length $\lth\ij$ are invariant under cyclic rotations, there are essentially two cases to consider which are illustrated in Figure~\ref{fig:nc_regularity2}. 
 \begin{figure}[htpb]
 \includegraphics[width=\textwidth]{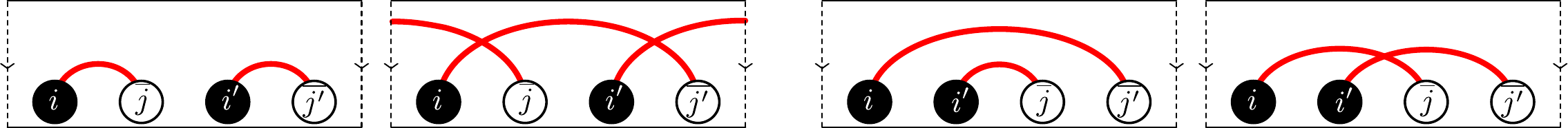}  
 \caption{The two matchings corresponding to each of the two configurations.}\label{fig:nc_is_regular}
 \label{fig:nc_regularity2}
 \end{figure}
In the first case (Figure~\ref{fig:nc_is_regular} left), the statement directly follows from monotonicity of $f$, because we have $\lth\ij<\lth(i',\ol{j})$ and $\lth(i',\ol{j'})<\lth(i,\ol {j'})$. 
The second case (Figure~\ref{fig:nc_is_regular} right) follows from $f$ being strictly concave, in exactly the same way as in the proof of Lemma~\ref{lem:non-crossing}.
\end{proof}

\begin{remark}[Non-crossing and non-nesting triangulations]

The triangulation $\asstri{n}$ was considered in~\cite{SantosStumpWelker2014} as a \defn{non-crossing complex}. Indeed, when restricted to $\{(i,\ol j)\colon i\prec \ol j\}$ only the case in Figure~\ref{fig:nc_is_regular} right is relevant, and the choice of the height function always selects the non-crossing matching over the non-nesting one.

They also consider \defn{non-nesting complexes}. If we restrict to $\{(i,\ol j)\colon i\prec \ol j\}$, then this is the well-known 
\defn{staircase triangulation} restricted to $\catblock{n}$. It amounts to choosing always the non-nesting matching over the non-crossing, which can be done by using as height function one that is convex on the length of the arcs.

The triangulation $\cyctri{n}$ is a type~$B_n$ analogue of $\asstri{n}$. It is obtained by cyclically rotating $\asstri{n}$.
The same procedure can be applied to get a type~$B_n$ analogue of the non-nesting triangulation. One obtains this way a full triangulation of $\productn$ that was presented in~\cite{CeballosPadrolSarmiento2015} under the name of \defn{Dyck path triangulation}.

While the cyclohedral triangulation (the type~$B_n$ non-crossing triangulation) is the pulling triangulation decreasing by length, the Dyck path triangulation (the type~$B_n$ non-nesting triangulation) is the pushing triangulation increasing by length. The first is the triangulation that prioritizes short arcs (counterclockwise), while the second is the triangulation that tries to avoid long arcs (counterclockwise). They have the same behavior in some cases (Figure~\ref{fig:nc_is_regular} right) but opposite in other cases (Figure~\ref{fig:nc_is_regular} left).

It is natural to ask for the other two cases: the pulling triangulation increasing by length and the pushing triangulation decreasing by length. However, observe that the counterclockwise length from $i$ to $\ol j$ is minus the clockwise length from $\ol j$ to $i$. Hence, this way one recovers mirror images of the same triangulations.
\end{remark}


\section{The cyclic \texorpdfstring{$\protect{\IJ}$}{IJ}-Tamari poset}\label{sec:cycTamari}

A type~$B_n$ analogue to the Tamari lattice was discovered independently by Thomas~\cite{Thomas06} and Reading~\cite{Reading06}. In this section we define a poset structure on cyclic $\IJ$-trees that extends the $\IJ$-Tamari lattice and generalizes the $B_n$ Tamari lattice. Specifically, Thomas' type $B_n$ Tamari lattice coincides with the cyclic $\IJ$-Tamari lattice when $I=\ol J$ (see Figure~\ref{fig:ThomasVSus}). However, as we will see, the $\IJ$-Tamari poset is not always lattice.

If $T$ and $T'$ are cyclic $\IJ$-trees related by a \defn{flip} (they share all arcs but one), 
we say that the flip is \defn{increasing} if it replaces an arc $(i,\ol j)$ with an arc $(i', j')$ with $i< i'$. In this case, we write $T <_{\IJ} T'$. In Figure~\ref{fig:cyclicflips}, we have schematically depicted the six possible increasing flips on a cyclic $\IJ$-tree.

\begin{figure}[htpb]
\centering 
 \includegraphics[width=.75\linewidth]{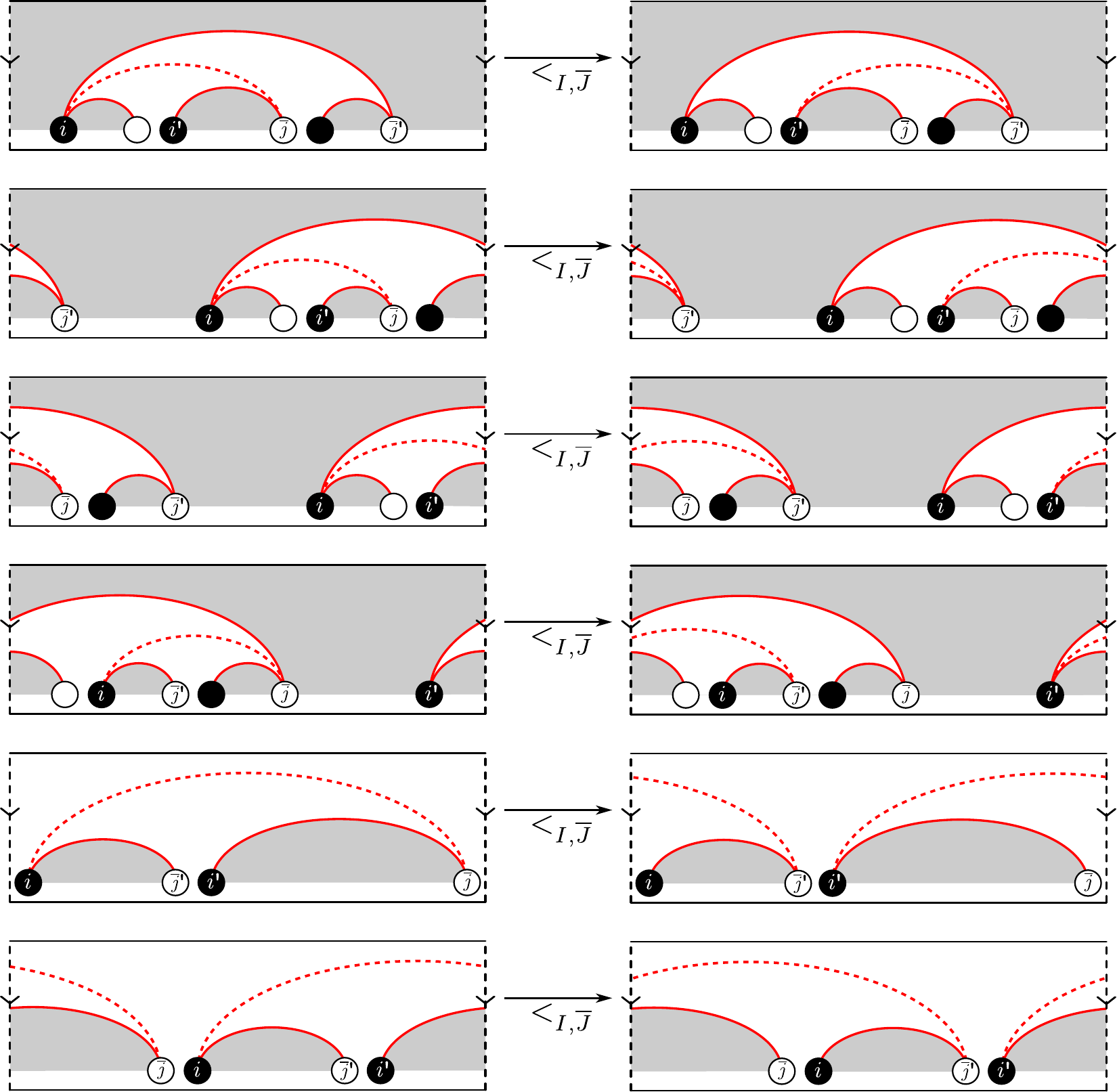}
 \caption{Schematic representations of possible cyclic flips on an $\IJ$-tree.}\label{fig:cyclicflips}
\end{figure}

\begin{lemma}
\label{lem:cyclicIJposet}
 The transitive closure of the relation $T <_{I,\ol J} T'$ is a partial order on the set of cyclic $(I,\ol J)$-trees. 
\end{lemma}
\begin{proof}
 This is a consequence of the acyclicity of $<_{I,\ol J}$, which follows from the fact that if $T <_{I,\ol J} T'$ then $\sum_{(i,\ol j)\in T} i < \sum_{(i',\ol j')\in T'} i' $.
\end{proof}

We denote this partial order $\Tam{I,\ol J}^B$, the \defn{cyclic $\IJ$-Tamari poset}.

\begin{figure}[htpb]
\includegraphics[width=\linewidth]{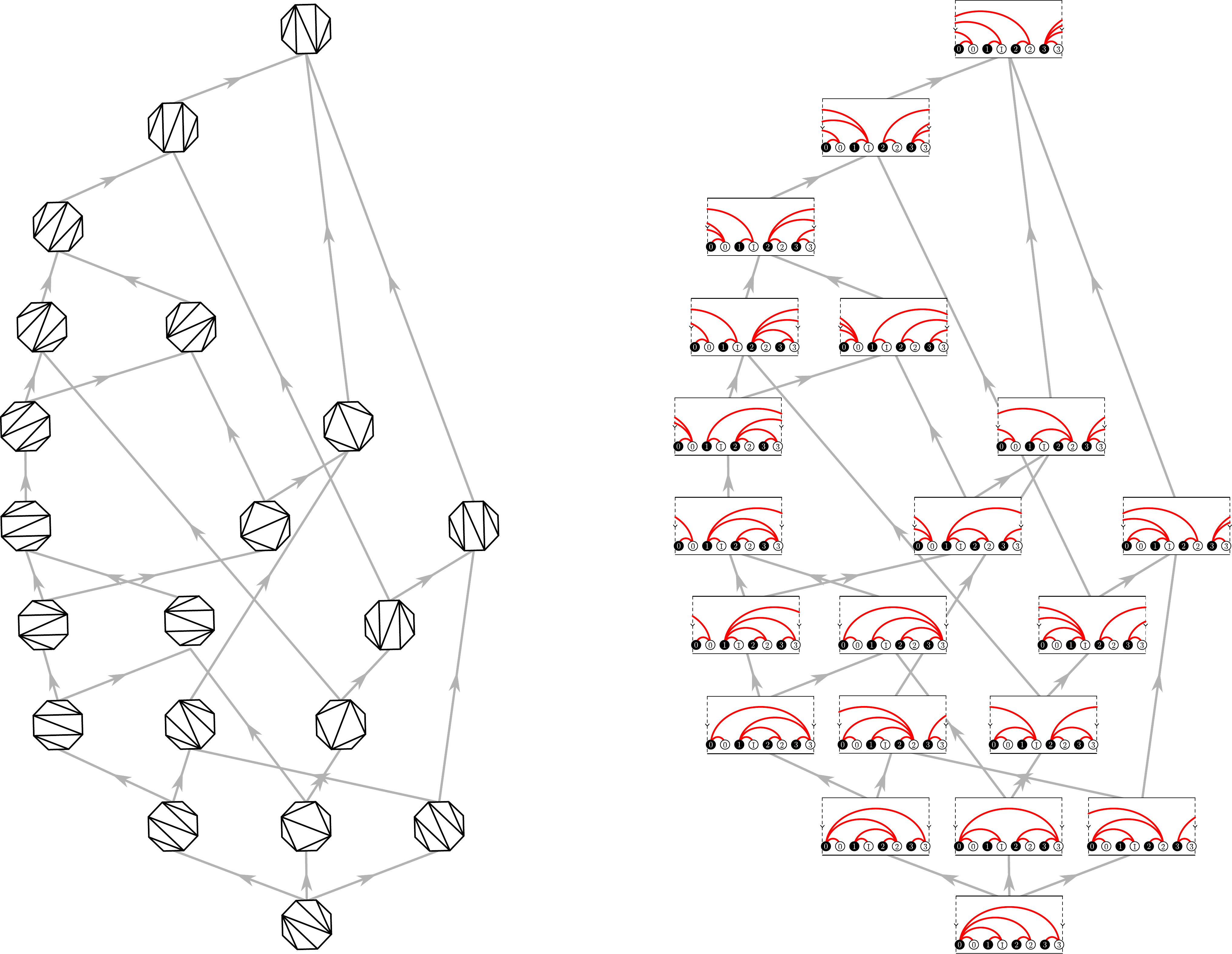}
 \caption{The Hasse diagram of the type $B_n$ Tamari lattice for $n=3$ from~\cite[Figure~5]{Thomas06} on the left, and the Hasse diagram of the cyclic $([3],[\ol 3])$-Tamari lattice on the right.}\label{fig:ThomasVSus}
\end{figure}

Both definitions of the covering relation, for Thomas' $B_n$ Tamari lattice and for the cyclic $\IJ$-Tamari poset, are stable under removal of vertices not involved in the flip. Since every flip involves at most eight vertices of the polygon, if the partial orders coincide for cs-triangulations of an octagon then they coincide for any $(2n+2)$-gon. This can be checked in Figure~\ref{fig:ThomasVSus}. As a consequence: 
 
\begin{lemma}
\label{lem:ThomasEqualsIJ}
The $B_n$ Tamari lattice of~\cite{Thomas06} is isomorphic to $\Tam{[n],[\ol n]}^B$, the cyclic $\nn$-Tamari poset.
\end{lemma}

\begin{remark}
One can define the canopy of a cyclic $\nn$-tree as the set of nodes not forming a leaf, the same way we did in Section~\ref{sec:canopy}. The fibers of the canopy map partition the set of cyclic $\nn$-trees, and there is a bijection between cyclic $\IJ$-trees and $\nn$-trees of canopy $I\sqcup\ol J$. The order relation $<_{I,\ol J}$ for the cyclic $\IJ$-Tamari poset is the one induced by the cyclic $\nn$-Tamari poset (i.e. the $B_n$ Tamari lattice) under this identification.
\end{remark}

\begin{remark}
\label{rem:cyclicNotNice}
Unfortunately, it so happens that the agreeable properties of the $\IJ$-Tamari poset are in general not shared by its cyclic counterpart. On the one hand, Figure~\ref{fig:notInterval} shows a 4-element chain in the cyclic $([4],[\ol 4])$-Tamari poset whose smallest and largest elements have the canopy $I=\{0,4\}, \ol J=\{\ol 1, \ol 2, \ol 3\}$, whereas its two intermediate elements have the canopy $I=\{0,2,4\}, \ol J=\{\ol 1, \ol 3\}$. Thus, we cannot possibly expect cyclic $\IJ$-Tamari posets to be intervals of larger cyclic $\nn$-Tamari posets.

On the other hand, it is not difficult to come up with cyclic $\IJ$-Tamari posets that have multiple minima or maxima, so we also have to rule out the possibilities that general cyclic $\IJ$-Tamari posets are either lattices, meet semi-lattices or join semi-lattices (see Figures~\ref{fig:manyExtrema} and~\ref{fig:IJcyclohedron}). It would be interesting to characterize those pairs of subsets $I\subseteq[n],\ol J\subseteq[\ol n]$ for which the cyclic $\IJ$-Tamari poset displays any of these (semi-)lattice structures (as for example happens in Figure~\ref{fig:IJcyclohedron}).
 \begin{figure}[htpb]
 \centering
 \begin{subfigure}{0.4\textwidth}
 \centering
\includegraphics[width=0.5\linewidth]{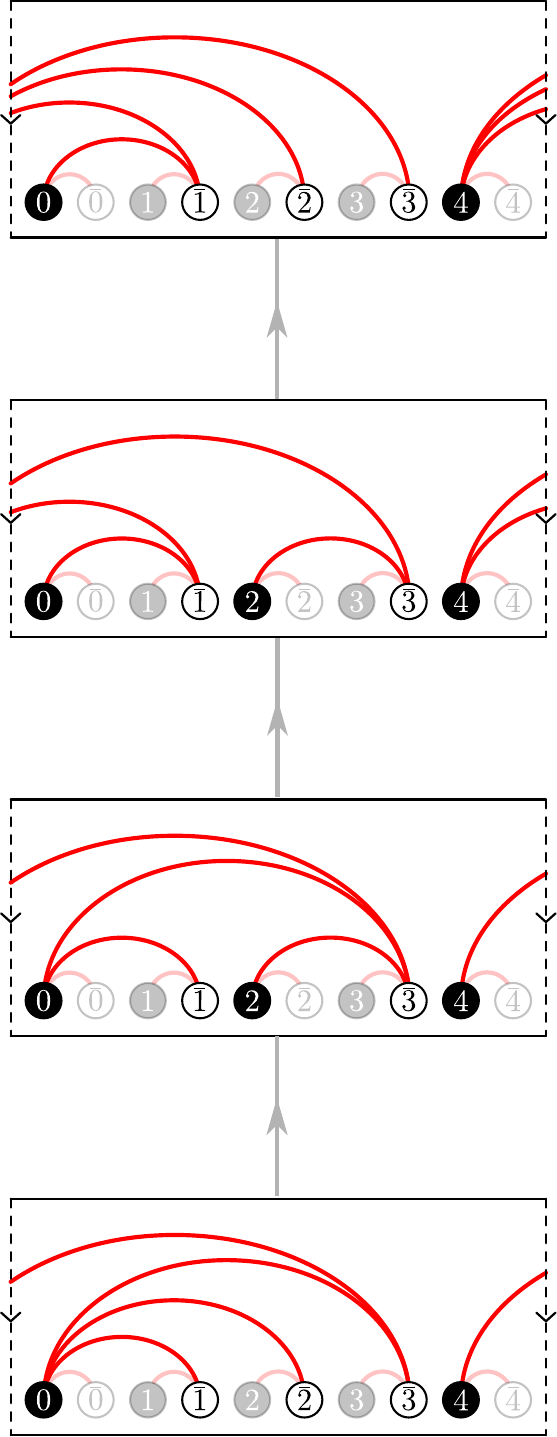}
\caption{A 4-element chain in the cyclic $([4,\ol 4])$-Tamari poset whose elements alternate between two canopies.}\label{fig:notInterval}
\end{subfigure}%
 \begin{subfigure}{0.6\textwidth}
 \centering
\includegraphics[width=.5\linewidth]{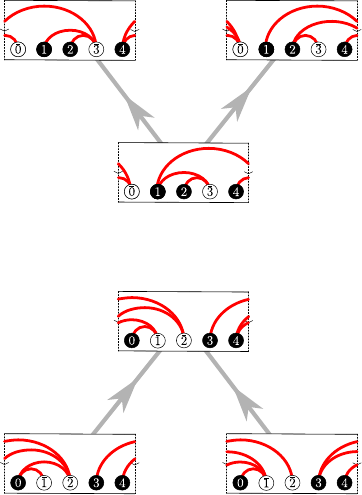}
 \caption{The cyclic $\IJ$-poset for $I=\{1,2,4\}, \ol J=\{\ol 0,\ol 3\}$ has two maxima (top), whereas for $I=\{0,3,4\}, \ol J=\{\ol 1,\ol 2\}$ it has two minima (bottom).}\label{fig:manyExtrema}
\end{subfigure}
\caption{Cyclic $\IJ$ posets need not be intervals nor (semi-)lattices.}
\end{figure}
\end{remark}

\begin{remark}
In general, applying the reversing operation taking a pair $\IJ$ to $\reverse\IJ$ (see Section~\ref{sec:canopy}) does not produce a poset dual to the cyclic $\IJ$-Tamari poset $\Tam{I,\ol J}^B$ (in contrast with what happened with $\Tam{I,\ol J}$, cf. Proposition~\ref{prop:reverse}). For example, while for $I=\{1,2,4\},\ol J=\{\ol 0,\ol 3\}$ the poset $\Tam{I,\ol J}^B$ consists one minimum and two maxima (see Figure~\ref{fig:manyExtrema}), $\Tam{\reverse{I,\ol J}}^B$ is a $3$-element chain. Even though the triangulations $\cyctri{I,\ol J}$ and $\cyctri{\reverse{I,\ol J}}$ are essentially equivalent (because cyclic $\reverse{\IJ}$-trees are mirror images of cyclic $\IJ$-trees), the covering relation of the cyclic $\IJ$-Tamari lattice depends subtly on the relative position of $I$ and $\ol J$. 

This symmetry operation, as well as doing cyclic rotations on the pair $\IJ$, alter the poset structure but not the underlying graph of its Hasse diagram.  
It would be interesting to understand the relation between the corresponding cyclic $\IJ$-Tamari posets.
\end{remark}

\begin{question}
 Is there a partial order on cyclic $\IJ$-trees (different to ours) possessing the structure of a lattice? Can this be made in such a way that its Hasse diagram is the edge graph of the $\IJ$-cyclohedron introduced below (see Section~\ref{sec:cyclicnuassociahedron})?
\end{question}


\section{The cyclic \texorpdfstring{$\IJ$}{IJ}-Tamari complex}

Throughout this subsection, we fix two nonempty subsets $I\subseteq [n], \ol J\subseteq [\ol n]$ for some $n\in\NN$. 

\begin{definition}
\label{def:CyclicTamariComplex}
The \defn{cyclic $\IJ$-Tamari complex $\cyccomp{I,\ol J}$} is the flag simplicial complex of cyclic $\IJ$-forests. That is $\cyccomp{I,\ol J}$ is the simplicial complex on the arcs of $\bipartitep{I}{J}$ whose minimal non-faces are pairs of cyclically non-crossing arcs (cf.~Definition~\ref{def:cyclicIJtree}).
\end{definition}

\begin{proposition}
\label{prop:bijectioncyclic}
$\cyccomp{[n],[\ol n]}$ is isomorphic to the join of an $n$-simplex and the boundary complex of a simplicial $n$-cyclohedron.
\end{proposition}

\begin{lemma}\label{lem:cyclicinteriorsimplex}
The interior simplices of $\cyccomp{I,\ol J}$ are naturally indexed by cyclic \defn{$\IJ$-forests} (see Definition~\ref{def:cyclicIJtree}) that have no isolated nodes. We refer to such cyclic {$\IJ$-forests} as \defn{covering cyclic $\IJ$-forests}.
\end{lemma}

Again, as in the classical situation, links of faces of cyclic $\IJ$-Tamari complexes are joins of smaller Tamari complexes and at most one cyclic Tamari complex. To give a precise result parallel to Lemma~\ref{lem:links}, we need some additional conventions to handle the cyclic nature of  cyclic Tamari complexes.

Let $F$ be a cyclic $\IJ$-forest. As in Lemma~\ref{lem:links}, we will write $\link_{\cyccomp{I,\ol J}}F$ as a join of complexes ranging over arcs of $F$, and the contributions of individual arcs  $\ij\in F$ will be determined by the restriction $I_{\ij}$ (resp. $\ol{J_{\ij}}$) of $I$ (resp. $\ol J$) to: 
\begin{equation*}
[i,\ol j]\setminus\left(\bigcup_{\substack{(i',\ol j')\in F\\ [i',\ol j']\subsetneq [i,\ol j]}}]i',\ol j'[\right).
\end{equation*}
Here, if $i\prec \ol j$, then  $[i,\ol j]$ and $]i,\ol j[$ denote the usual closed and open intervals, and we take $I_{\ij}$, $\ol{J_{\ij}}$ and $\asscomp{I_{\ij},\ol{J_{\ij}}}$ as in Lemma~\ref{lem:links}. On the other hand, if $\ol j\prec i$ we set $[i,\ol j]=[0,\ol j]\sqcup [i,\ol n]$, $]i, \ol j[=[0,\ol j[\ \sqcup\ ]i,\ol n] $, and interpret containment of intervals accordingly. We then define $\asscomp{I_{\ij},\ol{J_{\ij}}}$ as follows: First, we modify the restrictions $I_{\ij},\ol{J_{\ij}}$ by replacing every $i\in I_{\ij}$ such that $i'\prec i$ with $i'+n+1$ and every $\ol{j'}\in \ol{J_{\ij}}$ such that $\ol{j'}\prec i$ with $\ol{j'+n+1}$, where the addition is \textbf{not} in modular arithmetic. Then, we construct  $\asscomp{I_{\ij},\ol{J_{\ij}}}$ using these modified restrictions, and ultimately take the complexes obtained when regarding the index set $I_{\ij}\sqcup\ol{J_{\ij}}$ modulo $n+1$.

The proof of the next lemma is parallel to that of Lemma~\ref{lem:cycliclinks} following the schema depicted in Figure~\ref{fig:cyclicIJfaces}. We leave the details to the reader.
\begin{lemma}
\label{lem:cycliclinks}
Up to cone points, the link of a cyclic $\IJ$-forest $F$ in $\cyccomp{I,\ol J}$ is a join of Tamari complexes and at most one cyclic Tamari complex. Precisely:
\begin{equation}
\link_{\cyccomp{I,\ol J}}(F)\ast F\cong \left(\bigast_{(i,\ol j)\in F} \asscomp{I_{(i,\ol j)},\ol{J_{(i,\ol j)}}}\setminus (i,\ol j)\right)\ast\cyccomp{I',\ol{J'}},
\end{equation}
where $I'=I\setminus \bigcup_{(i',\ol j')\in F}]i',\ol{j'}[$ and $\ol{J'}=\ol J\setminus \bigcup_{(i',\ol j')\in F}]i',\ol{j'}[$. In particular, the link of an arc $\ij$ in $\cyccomp{I,\ol J}$ is up to cone points a join of a Tamari complex and possibly a cyclic Tamari complex:
\begin{equation}
\link_{\cyccomp{I,\ol J}}(\ij)\ast \ij\cong  \left( \asscomp{I_{\ij},\ol{J_{\ij}}}\setminus (i,\ol j) \right) \ast \cyccomp{I',\ol{J'}},
\end{equation}

Moreover, if $F$ is a covering cyclic $\IJ$-forest, then $\link_{\cyccomp{I,\ol J}} (F)$ is a join of boundary complexes of simplicial associahedra and at most one simplicial cyclohedron. 
\end{lemma}

Modulo minor technicalities, the proof of Lemma~\ref{lem:cycliclinks} follows the same arguments as the proof of Lemma~\ref{lem:links}, and is omitted.

\begin{figure}[htpb]
 \centering
\includegraphics[width=\linewidth]{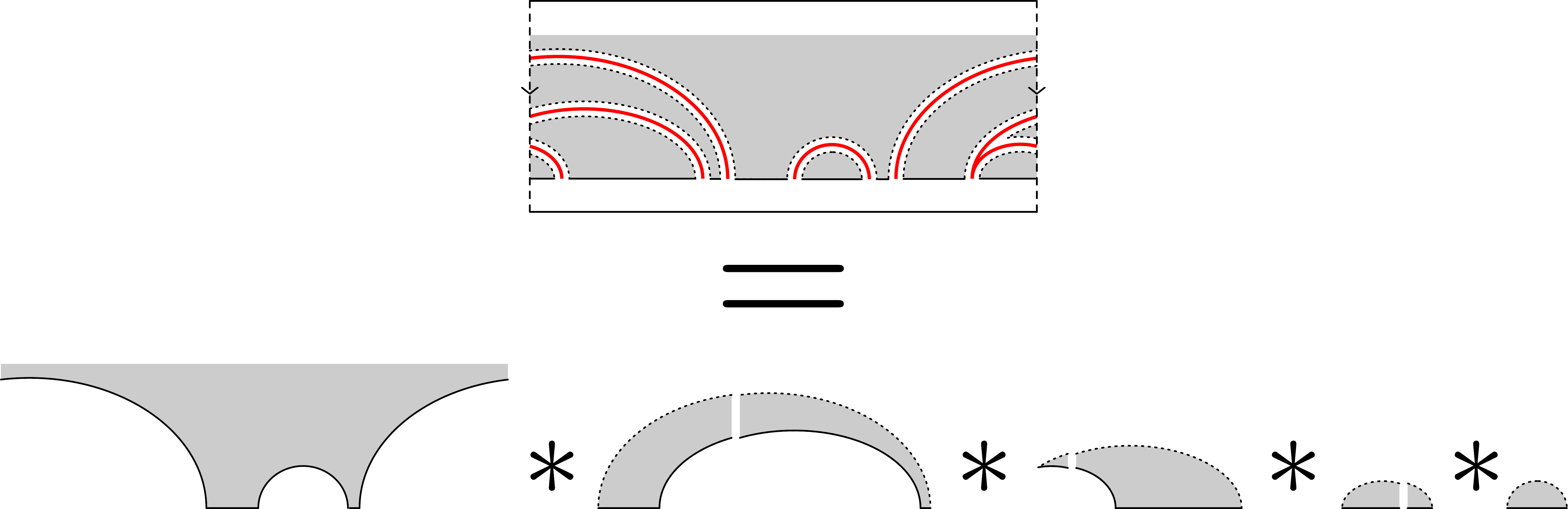}
 \caption{The link of cyclic $\IJ$-forest is a join of Tamari complexes and at most one cyclic Tamari complex (first complex in the join). Dashed arcs represent removed cone points.} 
 \label{fig:cyclicIJfaces}
\end{figure}


\subsection{The \texorpdfstring{$h$}{h}-vector of \texorpdfstring{$\cyctri{I,\ol J}$}{C IJ}}
Even though we lack a nice interpretation of cyclic $\IJ$-trees in terms of lattice paths, to contrast with Theorem~\ref{thm:thehvector}, it is nevertheless easy to compute the $h$-vectors of~$\cyctri{I,\ol J}$.

\begin{theorem}
\label{thm:typeBhVector}
The $h$-vector of $\cyctri{I,\ol J}$ has entries
\[
h_k(\cyctri{I,\ol J})=\binom{|I|-1}{k}\binom{|\ol J|-1}{k}.
\]
\end{theorem}
\begin{proof}
 The $h$-vector of a triangulation of an equidecomposable polytope, such as $\productp{I}{\ol J}$, does not depend on the triangulation~\cite[Section~8.5.3]{DeLoeraRambauSantos2010}.
 The $h$-vector of a triangulation of $\product$ is known to be (see~\cite{BilleraCushmanSanders1988})
 \begin{equation*}
  h_k=\binom{n}{k}\binom{m}{k}.
 \end{equation*}
\end{proof}

\begin{remark}
As we have seen in type $A$, any linear extension of the $\IJ$-Tamari lattice gives a shelling order of the $\IJ$-Tamari complex (Lemma~\ref{lem:shellingorder}). It turns out that the analogous result is not true in type $B$: consider the cyclic $\IJ$-Tamari poset illustrated at the bottom of Figure~\ref{fig:manyExtrema} for $I=\{0,3,4\}$ and $\ol J=\{\ol 1,\ol 2\}$. The $\IJ$-Tamari complex in this case is a three dimensional simplicial complex consisting of two tetrahedra $T_1$ and $T_2$ (the minimal elements of the poset) glued along an edge, and one tetrahedron $T_3$ (the maximal element of the poset) containing this common edge and one other vertex from each $T_1$ and $T_2$. The linear extension $T_1<T_2<T_3$ of the lattice is not a shelling order because $T_1\cap T_2$ is not a codimension 1 face of $T_2$.     
\end{remark}

Notice how, when $|I|=mn+1$ and $|\ol J|=n+1$, we recover the Fuss-Catalan analogues $\binom{mn}{k}\binom{n}{k}$ of the Narayana numbers of type $B_n$ (cf.~\cite[Section~10]{FominReading05}, \cite[Proposition~5.1]{Athanasiadis2005}).
\begin{definition}
We define the \defn{Fuss-Catalan cyclic Tamari complex} (resp. \defn{poset}) as the $\IJ$-Tamari complex (resp. poset) for any pair $\IJ$ giving the sequence $\bullet \circ (\bullet^m \circ)^n$.
\end{definition} 

Some examples of Fuss-Catalan cyclic Tamari posets are illustrated in Figure~\ref{fig:FussCatalanCyclohedra}.


\section{A tropical realization of the \texorpdfstring{\protect{$\IJ$}}{IJ}-cyclohedron}
\label{sec:cyclicnuassociahedron}

In the following, fix nonempty subsets $I\subseteq[n],\ \ol J \subseteq[\ol n]$, for some $n$, along with a cyclically non-crossing height function $\hb:I\times\ol J\to \RR$, that is, a height function that induces $\cyctri{I,\ol J}$ as a regular triangulation of $\productp{I}{\ol J}$ (cf. Proposition~\ref{prop:cycheights} and Lemma~\ref{lem:cyc-non-crossing}). 

In this section, we construct a polyhedral complex whose $1$-skeleton is isomorphic to the Hasse diagram of the cyclic $\IJ$-Tamari  poset. Like in Section~\ref{sec:nuassociahedron}, such complex is obtained by restricting the arrangement of tropical hyperplanes dual to the $\IJ$-cyclohedral triangulation $\cyctri{I,\ol J}$ to the subcomplex of bounded cells. However, the situation now is somewhat simpler because $\cyctri{I,\ol J}$ is a triangulation of $\productp{I}{\ol J}$, so we do not need to consider infinite heights and the corresponding tropical hyperplanes are all nondegenerate. We therefore make a condensed presentation omitting most of the proofs, which are completely parallel to those in Section~\ref{sec:nuassociahedron}.

Let $\arrgtB=(H_i)_{i\in I}$ be the arrangement of \emph{inverted tropical hyperplanes} given by  
\begin{equation*}
H_i=\{x\in \TP^{{|\ol J|}-1}: \max_{\ol j\in \ol J}\{ -\hb(i,\ol j)+x_{\ol j} \} \ \text{is attained twice}\}.
\end{equation*}

\begin{definition}\label{def:nucyclohedron}
The \defn{$\IJ$-cyclohedron $\Cyclo_{I,\ol J}(\hb)$} is the polyhedral complex of bounded cells induced by the arrangement of inverted tropical hyperplanes $\arrgtB$. We often omit mentioning $\hb$ when it is clear from the context.
\end{definition}

\begin{theorem}
\label{thm:cyclicnurealization}
The $\IJ$-cyclohedron $\Cyclo_{I,\ol J}(\hb)$ is a polyhedral complex whose poset of cells is anti-isomorphic to the poset of interior faces of the cyclic $\IJ$-Tamari complex. 
In particular, 
\begin{enumerate}
\item Its faces correspond to \emph{covering} cyclic $\IJ$-forests.
\item Its vertices correspond to cyclic $\IJ$-trees.
\item Two vertices are connected if and only if the corresponding cyclic trees are connected by a flip. That is, the edge graph of $\Cyclo_{I,\ol J}(\hb)$ is isomorphic to the Hasse diagram of the cyclic $\IJ$-Tamari poset.
\end{enumerate}
\end{theorem}

\begin{theorem}\label{thm:cyclicverticesandcoordinates}
Each cyclic $\IJ$-tree $T$ labels a vertex $\sfg(T)$ of $\Cyclo_{I,\ol J}(\hb)$ whose coordinate $\ol k$ is:
\begin{equation}\label{eq:g_treecyclic}
\sfg(T)_{\ol k}:=\sum_{(i,\ol j)\in P(\ol k)}\pm \ha({i,\ol j}),\qquad \ol k\in\ol J\setminus\{\ol j_{\max}\},
\end{equation}
where $P({\ol k})$ is the sequence of arcs traversed in the unique oriented path from~$\ol k$ to~$\ol j_{\max}$ in $T$ and 
the sign of each summand is positive if $(i,\ol j)$ is traversed from $\ol j$ to $i$ and negative otherwise.

Each covering cyclic $(I,\ol J)$-forest $F$ labels a cell $\sfg(F)$ of $\Cyclo_{I,\ol J}(\hb)$ that is a (bounded) convex polytope of dimension one less than the number of connected components of $F$, and whose vertices correspond to the cyclic $\IJ$-trees containing $F$:
\begin{align*}
\sfg(F)&=\conv\left\{ \sfg(T)\colon \text{ $T$ is a cyclic $(I,\ol J)$-tree containing $F$} \right\}\\
       &=\tilde \sfg(F)\cap \{ x_{\ol j_{\max}} = 0\},
\end{align*}
where $\sfg(F)$ and $\tilde \sfg(F)$ are as in Definition~\ref{def:cells}.
\end{theorem}

Again, this reduces to a tropical realization of the classical cyclohedron when $\IJ=([n],[\ol n])$.
\begin{corollary}\label{cor:classicalcyclohedron}
$\Cyclo([n],[\ol n])$ is a classical $n$-dimensional cyclohedron.
\end{corollary}

\begin{remark}
\label{ex:nonorientable}
The $\IJ$-cyclohedron for $I=\{0,1,3,4,6,7\}$ and $\ol J=\{\ol 2, \ol 5, \ol 8\}$ is shown in Figure~\ref{fig:IJcyclohedron}. Observe that, in contrast with the $\IJ$-associahedron, the edges of the $\IJ$-cyclohedron cannot be oriented in accordance with the cyclic $\IJ$-Tamari order using a linear functional. 
 \begin{figure}[hptb]
  \includegraphics[width=\linewidth]{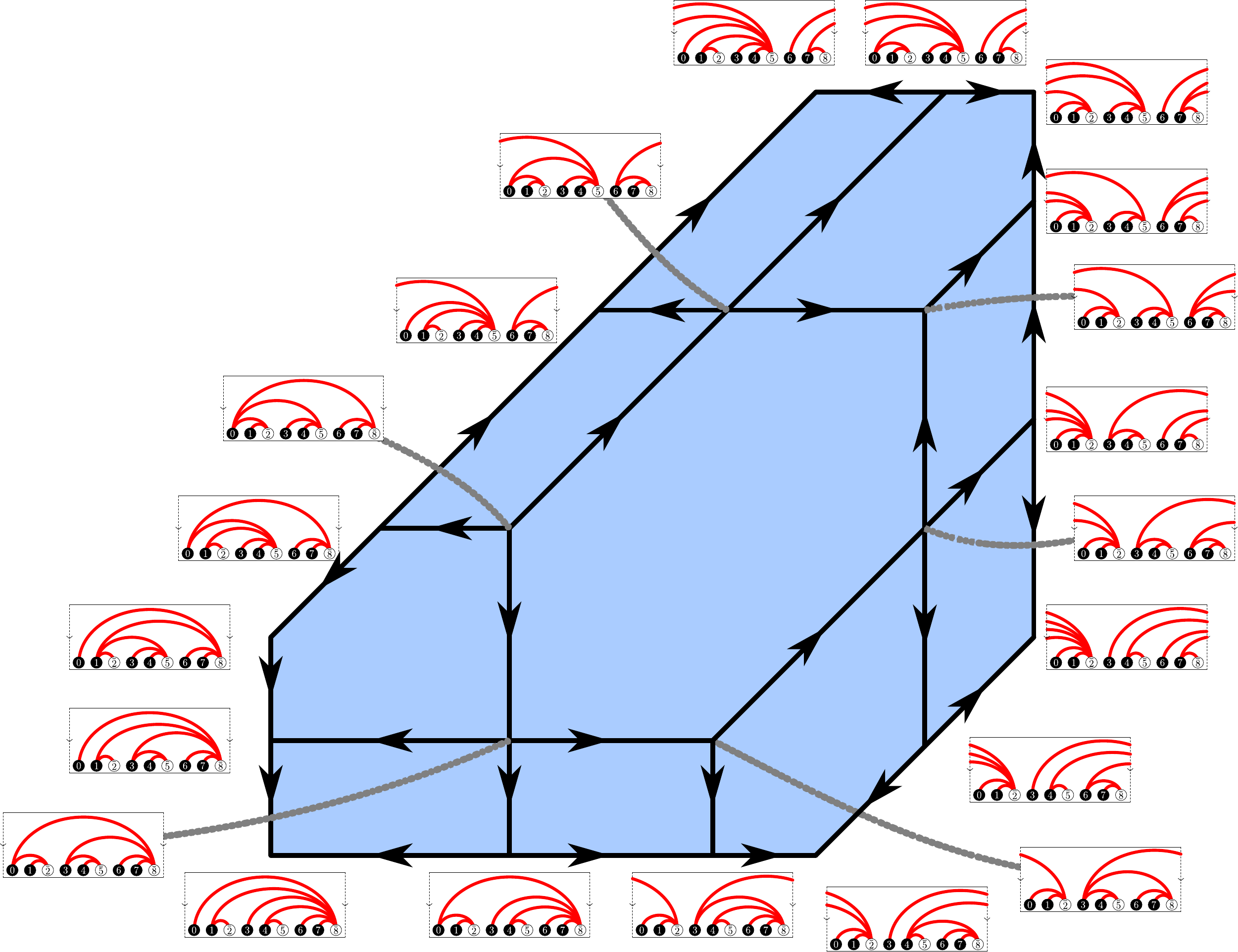}
  \caption{$IJ$-cyclohedron for $I=\{0,1,3,4,6,7\}, \ol J=\{\ol 2, \ol 5, \ol 8\}$.}\label{fig:IJcyclohedron}
 \end{figure}
\end{remark}

\begin{proposition}
 $\Cyclo_{I,\ol J}$ is piecewise-affinely isomorphic to $\Cyclo_{\reverse{I,\ol J}}$.
\end{proposition}

The $\IJ$-cyclohedron is neither always pure or convex. In fact, the $\IJ$-cyclohedron for $\IJ=(\{0,1,2,3\},\{\ol 4, \ol 5, \ol 6\})$ coincides with
the $\IJ$-associahedron depicted in Figure~\ref{fig:degenerate}. 

However, once more we have that for specially nice pairs $\IJ$, the support is just a classical cyclohedron.
The condition here is that $\IJ$ does not have two cyclically consecutive elements of $\ol J$ (when $|I|\geq 2$ and $|\ol J|\geq 3$), where two elements are cyclically consecutive if they are either consecutive or the first and last elements of $I\sqcup \ol J$. 
Actually, this is slightly simpler than the associahedral case (Theorem~\ref{thm:supportAssociahedron}) because the first and last element are not considered special and there is no need to work with special elements of $\ol J$. 

There are some trivial cases which we consider separately:  
If $|I|=1$ or $|\ol J|=1$ then $\CycloIJ(\hb)$ is a point; if $|I|>1$ and $|\ol J|=2$ then $\CycloIJ(\hb)$ is a subdivision of a segment. The other cases are covered by the following result.

\begin{theorem}\label{thm:supportCyclohedron}
Let $I,\ol J$ be finite subsets of $\NN$ with $|I|\geq 2$ and $|\ol J|\geq 3$. Then $\supp(\CycloIJ(\hb))$ is convex if and only if 
$\ol J$ does not have a cyclically consecutive pair of elements.
In this case, $\CycloIJ(\hb)$ is a regular polyhedral subdivision of a classical cyclohedron of dimension $(|\ol J|-1)$.
\end{theorem}
\begin{proof}
We prove convexity and cyclohedral support in two steps.
\item
\emph{Convexity:} Assume that $|I|\geq 2$ and $|\ol J|\geq 3$ and that there are two cyclically consecutive elements of $\ol J$.
It is not hard to see that if all the elements in $\ol J$ are (cyclically) consecutive, then up to relabeling $\CycloIJ(\hb)$ is an $\IJ$-Associahedron (because the great arc can only be placed in a $\ol J$ to $I$ transition, and in this case there is only one). Hence the proof of Theorem~\ref{thm:supportAssociahedron} holds. 

We assume now that $\ol J$ has two cyclically consecutive elements, but that they are not all in a row. Then there is a cyclic sequence starting with two elements of $\ol J$ ($\ol j_2,\ol j_3$), followed by a non-empty sequence of elements of $I$ (starting with $i_1$), then a non-empty sequence of elements of $\ol J$ (starting with $\ol j_1$), and finishing with an element of $I$ ($i_2$). After a cyclic shift, we can assume that $i_1$ is the smallest element of $I\sqcup \ol J$.
\begin{center}
\includegraphics[width=.5\linewidth]{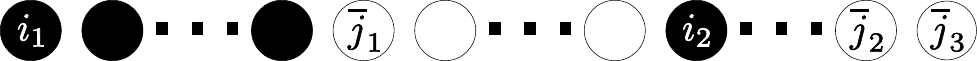} 
\end{center}

Now let $F$ be a covering $\IJ$-forest, and let $(i,\ol j_3)\in F$ be an arc adjacent to~$\ol j_3$.

If $i_1\preceq i\prec i_2$, then $i_1\preceq i\prec \ol j_1\prec \ol j_3$ and every point of $\sfg(F)$ fulfills
\begin{equation}\label{eq:cycforestcase1}
 x_{\ol j_1}-x_{\ol j_3}\leq \hb(i,\ol j_1)-\hb(i,\ol j_3)\stackrel{\eqref{eq:cycheights}}{\leq}\hb(i_1,\ol j_1)-\hb(i_1,\ol j_3).
\end{equation}
Otherwise, if $i_2\preceq i$, then $i_2\preceq i\prec \ol j_2\prec \ol j_3$ and every point of $\sfg(F)$ fulfills
\begin{equation}\label{eq:cycforestcase2}
 x_{\ol j_2}-x_{\ol j_3}\leq \hb(i,\ol j_2)-\hb(i,\ol j_3)\stackrel{\eqref{eq:cycheights}}{\leq}\hb(i_2,\ol j_2)-\hb(i_2,\ol j_3).
\end{equation}

Now we consider the non-crossing arcs $(i_1,\ol j_1), (i_1,\ol j_2), (i_1,\ol j_3)$ and complete them to a cyclic $\IJ$-tree $T_1$. It fulfills
\begin{align}
 \sfg(T_1)_{\ol j_1}-\sfg(T_1)_{\ol j_3}&=\hb(i_1,\ol j_1)-\hb(i_1,\ol j_3), \text{ and }\\
 \sfg(T_1)_{\ol j_2}-\sfg(T_1)_{\ol j_3}&=\hb(i_1,\ol j_2)-\hb(i_1,\ol j_3)\stackrel{\eqref{eq:cycheights}}{>}\hb(i_2,\ol j_2)-\hb(i_2,\ol j_3).
\end{align}
Analogously, complete the non-crossing arcs $(i_2,\ol j_1), (i_2,\ol j_2), (i_2,\ol j_3)$ to a cyclic $\IJ$-tree $T_2$, which fulfills
\begin{align}
 \sfg(T_2)_{\ol j_1}-\sfg(T_1)_{\ol j_3}&=\hb(i_2,\ol j_1)-\hb(i_2,\ol j_3)\stackrel{\eqref{eq:cycheights}}{>}\hb(i_1,\ol j_1)-\hb(i_1,\ol j_3), \text{ and }\\
 \sfg(T_2)_{\ol j_2}-\sfg(T_1)_{\ol j_3}&=\hb(i_2,\ol j_2)-\hb(i_2,\ol j_3).
\end{align}
Therefore, the midpoint of $\sfg(T_1)$ and $\sfg(T_2)$ does not fulfill neither \eqref{eq:cycforestcase1} nor \eqref{eq:cycforestcase2} and does not belong to $\supp(\CycloIJ(\hb))$.

\item
\emph{Cyclohedral support:} We only sketch the proof, which follows very closely that of (the second part of) Theorem~\ref{thm:supportAssociahedron}. Assume that $i_0$ immediately followed by an element of $I$ in the cyclic order, and let $I'=I\setminus i_0$. Then by induction 
$\Cyclo_{I',\ol J}$ is a regular subdivision of a $(|\ol J|-1)$-dimensional cyclohedron. Each cell of $\Cyclo_{I',\ol J}$ is refined into at most $|\ol J|$ cells by intersecting with the cones $\sfg(i_0,\ol j)$.

It remains to see that each cell of $\CycloIJ$ arises this way; that is, that removing $i_0$ from a covering cyclic $\IJ$-forest always
produces a covering cyclic $(I',\ol J)$-forest. Indeed, 
removing an arc $(i_0,\ol j)$ does not isolate $\ol j$ because it must be connected to its immediately preceding element (which belongs to $I$ and cannot be $i_0$).
\end{proof}

We finish the section like we did Section~\ref{sec:nuassociahedron}, studying the cells of $\CycloIJ$.

\begin{proposition}\label{prop:cells_IJcyclohedron}
The polyhedral cells of the polyhedral complex $\CycloIJ$ are isomorphic to Cartesian products of (classical) associahedra and at most one (classical) cyclohedron.
\end{proposition}


\section*{Acknowledgements}
We want to thank Fr\'ed\'eric Chapoton for showing us a beautiful picture of the $2$-Tamari lattice for $n=4$, posted in Fran\c{c}ois Bergeron's web page, which impulsed this project; Fran\c{c}ois Bergeron for letting us reproduce his drawings in Figure~\ref{fig:Bergeron}; and Michael Joswig, Georg Loho,Vincent Pilaud, Francisco Santos and Christian Stump for many interesting discussions. 
We also thank an anonymous referee for his/her valuable comments and suggestions, and specially for pointing out two important observations in Remark~\ref{rem:othergeometricrealization} and Remark~\ref{rem:Reading_latticecongruences}.
We have greatly benefited from \texttt{polymake}~\cite{polymake} and \texttt{SageMath}~\cite{sagemath} to produce visualizations of three-dimensional objects, and would like to sincerely thank their development teams and contributors.



\providecommand{\bysame}{\leavevmode\hbox to3em{\hrulefill}\thinspace}
\providecommand{\MR}{\relax\ifhmode\unskip\space\fi MR }
\providecommand{\MRhref}[2]{%
  \href{http://www.ams.org/mathscinet-getitem?mr=#1}{#2}
}
\providecommand{\href}[2]{#2}

\end{document}